\documentclass[reqno]{amsart}
\usepackage{mathrsfs}
\usepackage{color}
\usepackage{amsmath}
\usepackage{amsfonts}
\usepackage{amssymb}
\usepackage{graphicx}
\usepackage{hyperref}

 \newtheorem{Theorem}{Theorem}[section]
 \newtheorem{Corollary}[Theorem]{Corollary}
 \newtheorem{Lemma}[Theorem]{Lemma}
 \newtheorem{Proposition}[Theorem]{Proposition}
 
 \newtheorem{Definition}[Theorem]{Definition}

 \newtheorem{Remark}[Theorem]{Remark}

 \numberwithin{equation}{section}

\begin{document}

\title[Boundary points, minimal $L^2$ integrals and concavity property \uppercase\expandafter{\romannumeral5}]
 {Boundary points, minimal $L^2$ integrals and concavity property \uppercase\expandafter{\romannumeral5}---vector bundles}

\author{Qi'an Guan}
\address{Qi'an Guan: School of
Mathematical Sciences, Peking University, Beijing 100871, China.}
\email{guanqian@math.pku.edu.cn}

\author{Zhitong Mi}
\address{Zhitong Mi: Institute of Mathematics, Academy of Mathematics
and Systems Science, Chinese Academy of Sciences, Beijing, China
}
\email{zhitongmi@amss.ac.cn}
\author{Zheng Yuan}
\address{Zheng Yuan: School of
Mathematical Sciences, Peking University, Beijing 100871, China.}
\email{zyuan@pku.edu.cn}

\thanks{}



\date{\today}

\dedicatory{}

\commby{}


\begin{abstract}
In this article, for singular hermitian metrics on holomorphic vector bundles, we consider minimal $L^2$ integrals on sublevel sets of plurisubharmonic functions on weakly pseudoconvex K\"ahler manifolds related to modules at boundary points of the sublevel sets, and establish a concavity property of the minimal $L^2$ integrals. As applications, we present a necessary condition for the concavity degenerating to linearity, a strong openness property of the modules and a twisted version, an effectiveness result of the strong openness property of the modules, and an optimal support function related to the modules.
\end{abstract}

\maketitle

\section{Introduction}
The strong openness property of multiplier ideal sheaves \cite{GZSOC} (Demailly's strong openness conjecture \cite{DemaillySoc}: $\mathcal{I}(\varphi)=\mathcal{I}_+(\varphi):=\mathop{\cup} \limits_{\epsilon>0}\mathcal{I}((1+\epsilon)\varphi)$)
 is an important feature of multiplier ideal sheaves, which was called "opened the door to new types of approximation techniques" (see e.g. \cite{GZSOC,McNeal and Varolin,K16,cao17,cdM17,FoW18,DEL18,ZZ2018,GZ20,ZZ2019,ZhouZhu20siu's,FoW20,KS20,DEL21}),
where the multiplier ideal sheaf $\mathcal{I}(\varphi)$ was defined as the sheaf of germs of holomorphic functions $f$ such that $|f|^2e^{-\varphi}$ is locally integrable (see e.g. \cite{Tian,Nadel,Siu96,DEL,DK01,DemaillySoc,DP03,Lazarsfeld,Siu05,Siu09,DemaillyAG,Guenancia}),
and $\varphi$ is a plurisubharmonic function on a complex manifold $M$ (see \cite{Demaillybook}).

Guan-Zhou \cite{GZSOC} proved the strong openness property (the 2-dimensional case was proved by Jonsson-Musta\c{t}\u{a} \cite{JonssonMustata}).
After that, using the strong openness property, Guan-Zhou \cite{GZeff} proved a conjecture about volumes growth of the sublevel sets of quasi-plurisubharmonic functions which was posed by Jonsson-Musta\c{t}\u{a} (Conjecture J-M for short, see \cite{JonssonMustata}).

Considering the minimal $L^2$ integrals on sublevel sets of a plurisubharmonic function with respect to a module at a boundary point of the sublevel sets, Bao-Guan-Yuan \cite{BGY-boundary} (see also \cite{GMY-boundary2}) established a concavity property of the minimal $L^2$ integrals, which deduces an approach to Conjecture J-M \textbf{independent} of the strong openness property.

In this article, for singular hermitian metrics on holomorphic vector bundles, we consider minimal $L^2$ integrals on sublevel sets of plurisubharmonic functions on weakly pseudoconvex K\"ahler manifolds related to modules at boundary points of the sublevel sets,
and obtain a concavity property of minimal $L^2$ integrals. As applications, we present a necessary condition for the concavity degenerating to linearity, a strong openness property of the modules and a twisted version, an effectiveness result of the strong openness property of the modules, and an optimal support function related to the modules.

\subsection {Singular hermitian metrics on vector bundles}
Let $M$ be an $n-$dimensional complex manifold. Let $E$ be a rank $r$ holomorphic vector bundle over $M$ and $\bar{E}$ the conjugate of $E$. Let $h$ be a section of the vector bundle $E^*\otimes {\bar{E}}^*$ with measurable coefficients, such that $h$ is an almost everywhere positive definite hermitian form on $E$; we call such an $h$ a \textit{measurable metric} on $E$.

We would like to use the following definition for singular hermitian metrics on vector bundles in this article which is a modified version of the definition in \cite{CA}.
\begin{Definition} Let $M$, $E$ and $h$ be as above and $\Sigma\subset M$ be a closed set of measure zero. Let $\{M_j\}_{j=1}^{+\infty}$ be a sequence of relatively compact subsets of $M$ such that $M_1 \Subset  M_2\Subset  ...\Subset
M_j\Subset  M_{j+1}\Subset  ...$ and $\cup_{j=1}^{+\infty} M_j=M$. Assume that for each $M_j$, there exists a sequence of hermitian metrics $\{h_{j,s}\}_{s=1}^{+\infty}$ on $M_j$ of class $C^2$ such that

\centerline{$\lim\limits_{s\to+\infty}h_{j,s}=h$\ \ \ point-wisely on $M_j\backslash \Sigma$.}

We call the collection of data $(M,E,\Sigma,M_j,h,h_{j,s})$ a singular hermitian metric (s.h.m. for short) on $E$.
\label{singular metric}
\end{Definition}

\begin{Remark}[see \cite{CA}] Let $M$, $E$, $\Sigma$, $h$ be as in Definition \ref{singular metric}. Assume that there exists a sequence of hermitian metrics $\tilde{h}_s$ of class $C^2$ such that

\centerline{$\lim\limits_{s\to+\infty}\tilde{h}_s=h$\ \ \ in the $C^2-$topology on $M\backslash \Sigma$.}

The authors of \cite{CA} called such a collection of data $(X,E,\Sigma,h,\tilde{h}_s)$ a singular hermitian metric on $E$. They called $\Theta_h(E_{X\backslash \Sigma})$ the curvature of $(X,E,\Sigma,h,\tilde{h}_s)$ and denoted it by $\Theta_h(E)$. $\Theta_h(E)$ has continuous coefficients and values in $Herm_h(E)$ away from $\Sigma$; they denoted the a.e.-defined associated hermitian form on $TX\otimes E$ by the same symbol $\Theta_h(E)$.
\label{singular metric in CA}
\end{Remark}

We use the following definition of singular version of Nakano positivity in this article. Let $\omega$ be a hermitian metric on $M$, $\theta$ be a hermitian form on $TM$ with continuous coefficients and $(M,E,\Sigma,M_j,h,h_{j,s})$ be a s.h.m in the sense of Definition \ref{singular metric}.
\begin{Definition}Let things be as above. We write:
$$\Theta_h(E)\ge^s_{Nak} \theta \otimes Id_E$$
if the following requirements are met.

For each $M_j$, there exist a sequence of continuous functions $\lambda_{j,s}$ on $\overline{M_j}$ and a continuous function $\lambda_j$ on $\overline{M_j}$ subject to the following requirements:
\par
(1.2.1) for any $x\in \Omega:\ |e_x|_{h_{j,s}}\le |e_x|_{h_{j,s+1}},$ for any $s\in \mathbb{N}$ and any $e_x\in E_x$;
\par
(1.2.2) $\Theta_{h_{j,s}}(E)\ge_{Nak} \theta-\lambda_{j,s}\omega\otimes Id_E$ on $M_j$;
\par
(1.2.3)  $\lambda_{j,s}\to 0$ a.e. on $M_j$;
\par
(1.2.4) $0\le \lambda_{j,s}\le \lambda_j$ on $M_j$, for any $s$.
\label{singular nak}
\end{Definition}

We would also like to recall the following notation of singular version of Nakano positivity in \cite{CA}. Let $\omega$ be a hermitian metric on $M$, $\tilde{\theta}$ be a hermitian form on $TM$ with continuous coefficients and $(X,E,\Sigma,h,\tilde{h}_s)$ be a s.h.m in the sense of Remark \ref{singular metric in CA}.
\begin{Remark}[see \cite{CA}]Let things be as above. In \cite{CA}, the authors wrote
$$\Theta_h(E)\ge^s_{Nak} \tilde{\theta} \otimes Id_E$$
if the following requirements are met.

There exist a sequence of hermitian forms $\tilde{\theta}_s$ on $TM\otimes E$ with continuous coefficients, a sequence of continuous functions $\tilde{\lambda}_s$ on $M$ and a continuous function $\tilde{\lambda}$ on $M$ subject to the following requirements:
\par
(1.2.1) for any $x\in X:\ |e_x|_{\tilde{h}_s}\le |e_x|_{\tilde{h}_{s+1}},$ for any $s\in \mathbb{N}$ and any $e_x\in E_x$;
\par
(1.2.2) $\tilde{\theta}_s\ge_{Nak} \tilde{\theta}\otimes Id_E$;
\par
(1.2.3) $\Theta_{\tilde{h}_s}(E)\ge_{Nak} \tilde{\theta}_s-\tilde{\lambda}_s\omega\otimes Id_E$;
\par
(1.2.4) $\tilde{\theta}_s\to \Theta_h(E)$ a.e on $M$;
\par
(1.2.5)  $\tilde{\lambda}_s\to 0$ a.e on $M$;
\par
(1.2.6) $0\le \tilde{\lambda}_s\le \tilde{\lambda}$, for any $s$.
\end{Remark}

\begin{Remark}\label{example of singular metric}Let $M$ be a weakly pseudoconvex K\"ahler manifold. Let $\varphi$ be a plurisubharmonic function on $M$. Using regularization of quasi-plurisubharmonic function (see Theorem \ref{regularization on cpx mfld}), we know that $h:=e^{-\varphi}$ is a singular metric on $E:=M\times \mathbb{C}$ in the sense of Definition \ref{singular metric} and $h$ satisfies $\Theta_h(E)\ge^s_{Nak} 0$ in the sense of Definition \ref{singular nak}. We will prove Remark \ref{example of singular metric} in appendix (see Remark \ref{psh fun is singular metric}).
\end{Remark}

We recall the following definitions which can be referred to \cite{CA}.

\begin{Definition}[see \cite{CA}]Let $h$ be a measurable metric on $E$. Let $\mathcal{I}(h)$ be the analytic sheaf of germs of holomorphic functions on $M$ defined as follows:
\\
$\mathcal{I}(h)_x:=\{f_x\in\mathcal{O}_{X,x}:|f_xe_x|^2_h$ is integrable in some neighborhood of $x$, $\forall e_x\in \mathcal{O}(E)_x\}$.

Analogously, we define an analytic sheaf $\mathcal{E}(h)$ by setting:
\center{$\mathcal{E}(h)_x:=\{e_x\in \mathcal{O}(E)_x: |e_x|^2_h$ is integrable in some neighborhood of $x\}$.}
\end{Definition}

\subsection{Main result: minimal $L^2$ integrals and concavity property}\label{sec:Main result}
Let $M$ be a complex manifold. Let $X$ and $Z$ be closed subsets of $M$. We call that a triple $(M,X,Z)$ satisfies condition $(A)$, if the following two statements hold:

$\uppercase\expandafter{\romannumeral1}.$ $X$ is a closed subset of $M$ and $X$ is locally negligible with respect to $L^2$ holomorphic functions; i.e., for any local coordinated neighborhood $U\subset M$ and for any $L^2$ holomorphic function $f$ on $U\backslash X$, there exists an $L^2$ holomorphic function $\tilde{f}$ on $U$ such that $\tilde{f}|_{U\backslash X}=f$ with the same $L^2$ norm;

$\uppercase\expandafter{\romannumeral2}.$ $Z$ is an analytic subset of $M$ and $M\backslash (X\cup Z)$ is a weakly pseudoconvex K\"ahler manifold.

Let $M$ be an $n-$dimensional complex manifold.
Assume that $(M,X,Z)$  satisfies condition $(A)$. Let $K_M$ be the canonical line bundle on $M$.
Let $dV_M$ be a continuous volume form on $M$. Let $F\not\equiv 0$ be a holomorphic function on $M$. Let $\psi$ be a plurisubharmonic function on $M$.
Let $E$ be a holomorphic vector bundle on $M$ with rank $r$.
Let $\hat{h}$ be a smooth metric on $E$. Let $h$ be a \textit{measurable metric} on $E$. Denote $\tilde{h}:=he^{-\psi}$. Let $(M,E,\Sigma,M_j,\tilde{h},\tilde{h}_{j,s})$ be a singular metric on $E$. Assume that $\Theta_{\tilde{h}}(E)\ge^s_{Nak} 0$.

  Let $(V,z)$ be a local coordinate near a point $p$ of $M$ and $E|_V$ is trivial. Let $g\in H^0(V,\mathcal{O}(K_M\otimes E))$ and $g=\hat{g}\otimes e$ locally, where $\hat{g}$ is a holomorphic $(n,0)$ form on $V$ and $e$ is a local section of $E$ on $V$. We define $|g|^2_{h_0}|_V={\sqrt{-1}}^{n^{2}}g\wedge \bar{g}\langle e,e\rangle_{h_0}$, where $h_0$ is any (smooth or singular) metric on $E$. Note that $|g|^2_{h_0}|_V$ is invariant under the coordinate change and $|g|^2_{h_0}$ is a globally defined $(n,n)$ form on $M$.

Let $T\in [-\infty,+\infty)$.
Denote that
$$\Psi:=\min\{\psi-2\log|F|,-T\}.$$
For any $z \in M$ satisfying $F(z)=0$,
we set $\Psi(z)=-T$.
Note that for any $t\ge T$,
the holomorphic function $F$ has no zero points on the set $\{\Psi<-t\}$.
Hence $\Psi=\psi-2\log|F|=\psi+2\log|\frac{1}{F}|$ is a plurisubharmonic function on $\{\Psi<-t\}$.

\begin{Definition}
We call that a positive measurable function $c$ (so-called ``\textbf{gain}") on $(T,+\infty)$ is in class $\tilde{P}_{T,M,\Psi,h}$ if the following two statements hold:
\par
$(1)$ $c(t)e^{-t}$ is decreasing with respect to $t$;
\par
$(2)$ For any $t_0> T$, there exists a closed subset $E_0$ of $M$ such that $E_0\subset Z\cap \{\Psi(z)=-\infty\}$ and for any compact subset $K\subset M\backslash E_0$, $|e_x|_h^2c(-\psi)\ge C_K|e_x|^2_{\hat{h}}$ for any $ x\in K\cap\{\Psi<-t_0\}$, where $C_K>0$ is a constant and $e_x\in E_x$.
\end{Definition}

Let $z_0$ be a point in $M$. Denote that $\tilde{J}(E,\Psi)_{z_0}:=\{f\in H^0(\{\Psi<-t\}\cap V,\mathcal{O}(E)\}): t\in \mathbb{R}$ and $V$ is a neighborhood of $z_0\}$. We define an equivalence relation $\backsim$ on $\tilde{J}(E,\Psi)_{z_0}$ as follows: for any $f,g\in \tilde{J}(\Psi)_{z_0}$,
we call $f \backsim g$ if $f=g$ holds on $\{\Psi<-t\}\cap V$ for some $t\gg T$ and open neighborhood $V\ni z_0$.
Denote $\tilde{J}(E,\Psi)_{z_0}/\backsim$ by $J(E,\Psi)_{z_0}$, and denote the equivalence class including $f\in \tilde{J}(E,\Psi)_{z_0}$ by $f_{z_0}$.

If $z_0\in \cap_{t>T} \{\Psi<-t\}$, then $J(E,\Psi)_{z_0}=\mathcal{O}(E)_{z_0}$ (the stalk of the sheaf $\mathcal{O}(E)$ at $z_0$), and $f_{z_0}$ is the germ $(f,z_0)$ of holomorphic section $f$ of $E$. If $z_0\notin \cap_{t>T} \overline{\{\Psi<-t\}}$, then  $J(E, \Psi)_{z_0}$ is trivial.

 Let $f_{z_0},g_{z_0}\in J(E,\Psi)_{z_0}$ and $(q,z_0)\in \mathcal{O}_{M,z_0}$. We define $f_{z_0}+g_{z_0}:=(f+g)_{z_0}$ and $(q,z_0)\cdot f_{z_0}:=(qf)_{z_0}$.
Note that $(f+g)_{z_0}$ and $(qf)_{z_0}$ ($\in J(E,\Psi)_{z_0}$) are independent of the choices of the representatives of $f,g$ and $q$. Hence $J(E, \Psi)_{z_0}$ is an $\mathcal{O}_{M,z_0}$-module.

Let $dV_M$ be a continuous volume form on $M$. Recall that $h$ is a measurable metric on $E$. For $f_{z_0}\in J(E,\Psi)_{z_0}$ and $a\ge 0$, we call $f_{z_0}\in I\big(h,a\Psi\big)_{z_0}$ if there exist $t\gg T$ and a neighborhood $V$ of $z_0$,
such that $\int_{\{\Psi<-t\}\cap V}|f|^2_he^{-a\Psi}dV_M<+\infty$.
Note that $I\big(h,a\Psi\big)_{z_0}$ is an $\mathcal{O}_{M,z_0}$-submodule of $J(E,\Psi)_{z_0}$.
If $z_0\in \cap_{t>T} \{\Psi<-t\}$, then $I_{z_0}=\mathcal{O}(E)_{z_0}$, where $I_{z_0}:=I\big(\hat{h}_1,0\Psi \big)_{z_0}$ and $\hat{h}_1$ is a smooth metric on $E$.

Let $Z_0$ be a subset of $\cap_{t>T} \overline{\{\Psi<-t\}}$. Let $f$ be an $E$-valued holomorphic $(n,0)$ form on $\{\Psi<-t_0\}\cap V$, where $V\supset Z_0$ is an open subset of $M$ and $t_0\ge T$
is a real number.
Let $J_{z_0}$ be an $\mathcal{O}_{M,z_0}$-submodule of $J(E,\Psi)_{z_0}$ such that $I\big(h,\Psi\big)_{z_0}\subset J_{z_0}$,
where $z_0\in Z_0$.
Denote $J:=\cup_{z_0\in Z_0}J_{z_0}$.
Denote the \textbf{minimal $L^{2}$ integral} related to $J$
\begin{equation}
\label{def of g(t) for boundary pt}
\begin{split}
\inf\Bigg\{ \int_{ \{ \Psi<-t\}}|\tilde{f}|^2_hc(-\Psi): \tilde{f}\in
H^0(\{\Psi<-t\},\mathcal{O} (K_M\otimes E)  ) \\
\&\, (\tilde{f}-f)_{z_0}\in
\mathcal{O} (K_M)_{z_0} \otimes J_{z_0},\text{for any }  z_0\in Z_0 \Bigg\}
\end{split}
\end{equation}
by $G(t;c,\Psi,h,J,f)$, where $t\in[T,+\infty)$, $c$ is a nonnegative function on $(T,+\infty)$.
Without misunderstanding, we denote $G(t;c,\Psi,h,J,f)$ by $G(t)$ for simplicity. For various $c(t)$, we denote $G(t;c,\Psi,h,J,f)$ by $G(t;c)$ respectively for simplicity.

In this article, we obtain the following concavity property of $G(t)$.

\begin{Theorem}
\label{main theorem}
Let $c\in\tilde{P}_{T,M,\Psi,h}$. If there exists $t \in [T,+\infty)$ satisfying that $G(t)<+\infty$, then $G(h^{-1}(r))$ is concave with respect to  $r\in (\int_{T_1}^{T}c(t)e^{-t}dt,\int_{T_1}^{+\infty}c(t)e^{-t}dt)$, $\lim\limits_{t\to T+0}G(t)=G(T)$ and $\lim\limits_{t \to +\infty}G(t)=0$, where $h(t)=\int_{T_1}^{t}c(t_1)e^{-t_1}dt_1$ and $T_1 \in (T,+\infty)$.
\end{Theorem}

\begin{Remark}
	\label{infty2}Let $c\in\tilde{P}_{T,M,\Psi,h}$.	If  $\int_{T_1}^{+\infty}c(t)e^{-t}dt=+\infty$ and $f_{z_0}\notin
\mathcal{O} (K_M)_{z_0} \otimes J_{z_0}$ for some  $ z_0\in Z_0$, then $G(t)=+\infty$ for any $t\geq T$. Thus, when there exists $t \in [T,+\infty)$ satisfying that $G(t)\in(0,+\infty)$, we have $\int_{T_1}^{+\infty}c(t)e^{-t}dt<+\infty$ and $G(\hat{h}^{-1}(r))$ is concave with respect to  $r\in (0,\int_{T}^{+\infty}c(t)e^{-t}dt)$, where $\hat{h}(t)=\int_{t}^{+\infty}c(l)e^{-l}dl$.
\end{Remark}

For any $t\ge T$, denote
\begin{equation}\nonumber
\begin{split}
\mathcal{H}^2(t;c,f):=\Bigg\{\tilde{f}:\int_{ \{ \Psi<-t\}}|\tilde{f}|^2_hc(-\Psi)<+\infty,\  \tilde{f}\in
H^0(\{\Psi<-t\},\mathcal{O} (K_M\otimes E)  ) \\
\& (\tilde{f}-f)_{z_0}\in
\mathcal{O} (K_M)_{z_0} \otimes J_{z_0},\text{for any }  z_0\in Z_0  \Bigg\},
\end{split}
\end{equation}
where $f$ is an $E$-valued holomorphic $(n,0)$ form on $\{\Psi<-t_0\}\cap V$ for some $V\supset Z_0$ is an open subset of $M$ and some $t_0\ge T$ and $c(t)$ is a positive measurable function on $(T,+\infty)$.

As a corollary of Theorem \ref{main theorem}, we give a necessary condition for the concavity property degenerating to linearity.
\begin{Corollary}
\label{necessary condition for linear of G}
Let $c\in\tilde{P}_{T,M,\Psi,h}$.
Assume that $G(t)\in(0,+\infty)$ for some $t\ge T$, and $G(\hat{h}^{-1}(r))$ is linear with respect to $r\in[0,\int_T^{+\infty}c(s)e^{-s}ds)$, where $\hat{h}(t)=\int_{t}^{+\infty}c(l)e^{-l}dl$.

Then there exists a unique $E$-valued holomorphic $(n,0)$ form $\tilde{F}$ on $\{\Psi<-T\}$
such that $(\tilde{F}-f)_{z_0}\in\mathcal{O} (K_M)_{z_0} \otimes J_{z_0}$ holds for any  $z_0\in Z_0$,
and $G(t)=\int_{\{\Psi<-t\}}|\tilde{F}|^2_hc(-\Psi)$ holds for any $t\ge T$.

Furthermore
\begin{equation}
\begin{split}
  \int_{\{-t_1\le\Psi<-t_2\}}|\tilde{F}|^2_ha(-\Psi)=\frac{G(T_1;c)}{\int_{T_1}^{+\infty}c(t)e^{-t}dt}
  \int_{t_2}^{t_1}a(t)e^{-t}dt
  \label{other a also linear}
\end{split}
\end{equation}
holds for any nonnegative measurable function $a$ on $(T,+\infty)$, where $T\le t_2<t_1\le+\infty$ and $T_1 \in (T,+\infty)$.
\end{Corollary}

\begin{Remark}
\label{rem:linear}
If $\mathcal{H}^2(t_0;\tilde{c},f)\subset\mathcal{H}^2(t_0;c,f)$ for some $t_0\ge T$, we have
\begin{equation}
\begin{split}
  G(t_0;\tilde{c})=\int_{\{\Psi<-t_0\}}|\tilde{F}|^2_h\tilde{c}(-\Psi)=
  \frac{G(T_1;c)}{\int_{T_1}^{+\infty}c(t)e^{-t}dt}
  \int_{t_0}^{+\infty}\tilde{c}(s)e^{-s}ds,
  \label{other c also linear}
\end{split}
\end{equation}
 where $\tilde{c}$ is a nonnegative measurable function on $(T,+\infty)$ and $T_1 \in (T,+\infty)$. Thus, if $\mathcal{H}^2(t;\tilde{c})\subset\mathcal{H}^2(t;c)$ for any $t>T$, then $G(\hat{h}^{-1}(r);\tilde c)$ is linear with respect to $r\in[0,\int_T^{+\infty}c(s)e^{-s}ds)$.
\end{Remark}

\subsection{Applications}
In this section, we give some applications of Theorem \ref{main theorem}.

\subsubsection{Strong openness property of $I(h,a\Psi)_{z_0}$}\label{sec:soc}

In this section, we give an estimate of $|f|^2_{h}$ on sublevel sets of $\Psi$, which implies the strong openness property of $I(h,\Psi)_{z_0}$.

Let $M$ be an $n-$dimensional weakly pseudoconvex K\"ahler manifold, and let $dV_M$ be a continuous volume form on $M$. Let $K_M$ be the canonical line bundle on $M$.
 Let $F\not\equiv 0$ be a holomorphic function on $M$. Let $\psi$ be a plurisubharmonic function on $M$.
Let $E$ be a holomorphic vector bundle on $M$ with rank $r$. We call a measurable metric $\hat h$ on $E$ has a \textit{positive locally lower bound} if for any compact subset $K$ of $M$, there exists a constant $C_K>0$ such that $\hat h\ge C_K  h_1$ on $K$, where $ h_1$ is a smooth metric on $E$.
Let $h$ be a measurable metric on $E$  satisfying that $h$ has a positive locally  lower bound.

Denote that
$$\Psi:=\min\{\psi-2\log|F|,0\}.$$ Let $z_{0}\in M$. Recall that $\hat f_{z_0}\in I(h,a\Psi)_{z_0}$ if and only if there exist $t\gg 0$ and a neighborhood $V$ of $z_0$ such that $\int_{\{\Psi<-t\}\cap V}|\hat f|^2_he^{-a\Psi}dV_M<+\infty$, where $a\ge0$. Denote that
$$I_+(h,a\Psi)_{}z_0:=\cup_{s>a}I(h,s\Psi)_{z_0}.$$
Let $f$ be an $E$-valued holomorphic $(n,0)$ form on $\{\Psi<-t_0\}$ such that $f_{z_0}\in \mathcal{O}(K_M)_{z_0}\otimes I(h,0\Psi))_{z_0}$. Denote that
$$a_{z_0}^f(\Psi;h):=\sup\{a\ge0:f_{z_0}\in (\mathcal{O}(K_M)\otimes I(h,2a\Psi))_{z_0}\}.$$
Especially, $a_{z_0}^f(\Psi;h)$ is the jumping number $c_{z_0}^f(\psi)$ (see \cite{JM13}), when $F\equiv1$, $\psi(z_0)=-\infty$,   $E$ is the trivial line bundle and $h\equiv1$.

\begin{Theorem}\label{p:DK}
	Assume that $a_{z_0}^f(\Psi;h)<+\infty$ and $\Theta_{\tilde h}(E)\ge^s_{Nak} 0$, where $\tilde h:=he^{-2a_{z_0}^f(\\\Psi;h)\psi}$. Then we have $a_{z_0}^f(\Psi;h)>0$ and
	\begin{displaymath}
		\frac{1}{r^2}\int_{\{a_{z_0}^f(\Psi;h)\Psi<\log r\}}|f|^2_h\ge G(0;c\equiv1,\Psi,h,I_+(h,2a_{z_0}^f(\Psi;h)\Psi)_{z_0},f)>0
	\end{displaymath}
	holds for any $r\in(0,e^{-a_{z_0}^f(\Psi;h)t_0}]$, where the definition of $G(0;c\equiv1,\Psi,h,I_+(h,2a_{z_0}^f(\Psi;h)\Psi)_{z_0},f)$ can be found in Section \ref{sec:Main result}.
\end{Theorem}

Theorem \ref{p:DK} implies the following strong openness property of $I(h,a\Psi)_{z_0}$.

\begin{Corollary}
	\label{thm:soc}
	$I(h,a\Psi)_{z_0}=I_+(h,a\Psi)_{z_0}$ holds for any $a\ge0$ satisfying $\Theta_{he^{-2a\psi}}\ge_{Nak}^s0$.
\end{Corollary}

When $E$ is the trivial line bundle and $h=e^{-\varphi}$, where $\varphi$ is a plurisubharmonic function on $M$, Theorem \ref{p:DK} and Corollary \ref{thm:soc} can be referred to \cite{GMY-boundary2}.

\begin{Remark}
	\label{r:soc-inner}Let $F\equiv1$ and $\psi(z_0)=-\infty$. Note that $z_0\in\cap_{t\ge T}\{\Psi<-t\}$ and $I(h,a\Psi)_{z_0}=\mathcal E(he^{-a\psi})_{z_0}$, then  Corollary \ref{thm:soc} is a vector bundle version of the strong openness property of multiplier ideal sheaves \cite{GZSOC}.
\end{Remark}

\subsubsection{Effectiveness of the strong openness property of $I(h,\Psi)_{z_0}$}
In this section, we give an effectiveness result of the strong openness property of $I(h,\Psi)_{z_0}$ (Corollary \ref{thm:soc}).
We follow the notations and assumptions in Section \ref{sec:soc}. Let $f$ be an $E$-valued holomorphic $(n,0)$ form on $\{\Psi<0\}$, and denote that
\begin{displaymath}
	\begin{split}
		\frac{1}{K_{\Psi,f,h,a}(z_0)}:=\inf\Bigg\{\int_{\{\Psi<0\}}|\tilde f|_h^2&e^{-(1-a)\Psi}:\tilde f\in H^0(\{\Psi<0\},\mathcal{O}(K_M\otimes E))\\
		& \& \,(\tilde f-f)_{z_0}\in\mathcal{O}(K_M)_{z_0}\otimes I_+(h,2a_{z_0}^{f}(\Psi;h)\Psi)_{z_0}\Bigg\},
	\end{split}
\end{displaymath}
where $a\in(0,+\infty)$.

We present the following effectiveness result of the strong openness property of $I(h,\Psi)_{z_0}$.
\begin{Theorem}
	\label{thm:effe}Assume that $\Theta_{\tilde h}(E)\ge_{Nak}^s0$, where $\tilde h:=he^{-2a_{z_0}^f(\Psi,h)\psi}$.
	Let $C_1$ and $C_2$ be two positive constants. If there exists $a>0$, such that
	
	$(1)$ $\int_{\{\Psi<0\}}|f|_h^2e^{-\Psi}\leq C_1$;
	
	$(2)$ $\frac{1}{K_{\Psi,f,h,a}(z_0)}\geq C_2$.
	
	Then for any $q>1$ satisfying $$\theta_a(q)>\frac{C_1}{C_2},$$ we have
$f_{z_0}\in \mathcal{O}(K_M)_{z_0}\otimes I(h,q\Psi)_{z_0}$, where $\theta_a(q)=\frac{q+a-1}{q-1}$.	\end{Theorem}

\subsubsection{A twisted version of the strong openness property of $I(h,a\Psi)_o$}

 Let $D\subseteq\mathbb{C}^n$ be a pseudoconvex domain containing the origin $o$, and let $\psi$ be a plurisubharmonic function on $D$. Let $F\not\equiv0$ be a holomorphic function on $D$.
Denote that
$$\Psi:=\min\{\psi-2\log|F|,0\}.$$
For any $z \in M$ satisfying $F(z)=0$, we set $\Psi(z)=0$. Let $E$ be a holomorphic vector bundle on $D$ with rank $r$, and
let $h$ be a measurable metric on $E$ satisfying that $h$ has a positive locally lower bound.

 It is clear that the following two statements are equivalent:

$(1)$ The strong openness property of $I(h,a\Psi)_o$ (Corollary \ref{thm:soc}): $I(h,a\Psi)_{o}=I_+(h,a\Psi)_o$ for any $a\ge0$ satisfying  $\Theta_{he^{-a\psi}}\ge_{Nak}^s0$;

 $(2)$ $f_{o}\not\in I(h,2a_{o}^f(\Psi;h)\Psi)_{o}$ for any $f_{o}\in I(h,0\Psi)_{o}$ satisfying $a_{o}^f(\Psi;h)<+\infty$ and $\Theta_{he^{-a_{o}^f(\Psi;h)\psi}}\ge_{Nak}^s0$.

 We present a twisted version of the strong openness property of $I(h,a\Psi)_{o}$.

\begin{Theorem}
	\label{p:soc-twist}
Let $a(t)$ be a positive measurable function on $(-\infty,+\infty)$. If  one of the following conditions holds:
	
	$(1)$ $a(t)$ is decreasing  near $+\infty$;
	
	$(2)$ $a(t)e^t$ is increasing near $+\infty$,
	
	 then the following two statements are equivalent:
	
	$(A)$ $a(t)$ is not integrable near $+\infty$;
	
	$(B)$ for any $\Psi$, $h$ and  $f_{o}\in I(h,0\Psi)_{o}$ satisfying $a_{o}^f(\Psi;h)<+\infty$ and $\Theta_{he^{-a_{o}^f(\Psi;h)\psi}}\ge_{Nak}^s0$, we have
	$$|f|_h^2e^{-2a_{o}^{f}(\Psi;h)\Psi}a(-2a_{o}^{f}(\Psi;h)\Psi)\not\in L^1(U\cap\{\Psi<-t\})$$ for any neighborhood $U$  of $o$ and any $t>0$.	
\end{Theorem}

When $E$ is the trivial line bundle, Theorem \ref{p:soc-twist} can be referred to \cite{GMY-boundary2}. When $F\equiv1$, $\psi(o)=-\infty$ and $E$ is the trivial line bundle, Theorem \ref{p:soc-twist} is a twisted version of the strong openness property of multiplier ideal sheaves (some related results can be referred  to \cite{GZ-soc17}, \cite{chen18} and \cite{GY-twisted}).

\subsubsection{An optimal support function related to $I(h,\Psi)$}
Let $M$ be an $n-$dimensional complex manifold.
Let $X$ and $Z$ be closed subsets of $M$ such that $(M,X,Z)$  satisfies condition $(A)$. Let $K_M$ be the canonical line bundle on $M$.
Let $F\not\equiv 0$ be a holomorphic function on $M$. Let $\psi$ be a plurisubharmonic function on $M$.
Let $E$ be a holomorphic vector bundle on $M$ with rank $r$, and
let $h$ be a measurable metric on $E$ satisfying that $\Theta_{he^{-\psi}}\ge_{Nak}^s0$ and $h$ has a positive locally lower bound.

Denote that
$$\Psi:=\min\{\psi-2\log|F|,0\}$$
and $M_t:=\{z\in M:-t\le\Psi(z)<0\}$.
Let $Z_0$ be a subset of $M$, and let  $f$ be an $E$-valued holomorphic $(n,0)$ form on $\{\Psi<0\}$.
Denote
\begin{displaymath}\begin{split}
		\inf\bigg\{\int_{M_t}|\tilde f|_h^2:&f\in H^0(\{\Psi<0\},\mathcal{O}(K_M\otimes E))\\
		&\&\,(\tilde f-f)_{z_0}\in \mathcal{O}(K_M)\otimes I(h,\Psi)_{z_0}\,\text{for any }  z_0\in Z_0\bigg\}	
\end{split}\end{displaymath}
by $C_{\Psi,f,h,t}(Z_0)$ for any $t\ge0$. When $C_{\Psi,f,h,t}(Z_0)=0$ or $+\infty$, we set $\frac{\int_{M_t}|f|_h^2e^{-\Psi}}{C_{\Psi,f,h,t}(Z_0)}=+\infty$.

We obtain the following optimal support function of $\frac{\int_{M_t}|f|_h^2e^{-\Psi}}{C_{\Psi,f,h,t}(Z_0)}$.
\begin{Proposition}
	\label{p:1}
	Assume that $\int_{\{\Psi<-l\}}|f|_h^2<+\infty$ holds for any $l>0$. Then the inequality
	\begin{equation}
		\label{eq:0304a}
		\frac{\int_{M_t}|f|_h^2e^{-\Psi}}{C_{\Psi,f,h,t}(Z_0)}\ge\frac{t}{1-e^{-t}}
	\end{equation}
	holds for any $t\ge0$, where $\frac{t}{1-e^{-t}}$ is the optimal support function.
\end{Proposition}

When $E$ is the trivial line bundle and $h\equiv1$, Proposition \ref{p:1} can be referred to \cite{GMY-boundary2}.

Take $M=\Delta\subset\mathbb{C}$, $Z_{0}=o$ the origin of $\mathbb{C}$, $F\equiv1$ and $\psi=2\log{|z|}$. Let $E$ is the trivial line bundle, $h\equiv1$ and $f\equiv dz$.
It is clear that $\int_M|f|^2_h<+\infty$. By direct calculations, we have $C_{\Psi,f,h,t}(Z_0)=2\pi(1-e^{-t})$ and $\int_{M_t}|f|^2_he^{-\Psi}=2t\pi$.
Then $\frac{\int_{M_t}|f|^2e^{-\Psi}}{C_{f,\Psi,t}(Z_0)}=\frac{t}{1-e^{-t}}$, which shows the optimality of the support function $\frac{t}{1-e^{-t}}$.

\section{Preparations}
\subsection{$L^2$ methods}
Let $X$ be an $n-$dimensional weakly pseudoconvex K\"ahler manifolds. Let $\psi$ be a plurisubharmonic function on $M$. Let $F$ be a holomorphic function on $X$. We assume that $F$ is not identically zero. Let $E$ be a rank $r$ holomorphic vector bundle over $X$. Let $\hat{h}$ be a smooth metric on $E$.
Let $(X,E,\Sigma,M_j,h,h_{j,s})$ be a singular hermitian metric on $E$.  Assume that $\Theta_h(E)\ge^s_{Nak}0$.

Let $\delta$ be a positive integer. Let $T$ be a real number. Denote
$$\tilde{M}:=\max\{\psi+T,2\log|F|\}$$
and
$$\Psi:=\min\{\psi-2\log|F|,-T\}.$$
If $F(z)=0$ for some $z \in M$, we set $\Psi(z)=-T$.

Let $c(t)$ be a positive measurable function on $[T,+\infty)$ such that $c(t)e^{-t}$ is decreasing with respect $t$. We have the following lemma.

\begin{Lemma}
\label{L2 method}
Let $B\in(0,+\infty)$ and $t_0>T$ be arbitrarily given. Let $f$ be an $E$-valued holomorphic $(n,0)$ form on $\{\Psi<-t_0\}$ such that
\begin{equation}
\label{estimate on sublevel set of lemma2.1}
\int_{\{\Psi<-t_0\}\cap K}|f|_{\hat{h}}^2<+\infty,
\end{equation}
for any compact subset $K\subset X$, and
\begin{equation}
\label{estimate on annaul of lemma2.1}
\int_{M}\frac{1}{B}\mathbb{I}_{\{-t_0-B<\Psi<-t_0\}}|fF|^2_h<+\infty.
\end{equation}
 Then there exists an $E$-valued holomorphic $(n,0)$ form $\tilde{F}$ on $X$ such that
\begin{equation*}
  \begin{split}
      & \int_{X}|\tilde{F}-(1-b_{t_0,B}(\Psi))fF^{1+\delta}|^2_he^{v_{t_0,B}(\Psi)-\delta\tilde{M}}c(-v_{t_0,B}(\Psi)) \\
      \le & \left(\frac{1}{\delta}c(T)e^{-T}+\int_{T}^{t_0+B}c(s)e^{-s}ds\right)
       \int_{X}\frac{1}{B}\mathbb{I}_{\{-t_0-B<\Psi<-t_0\}}|fF|^2_h,
  \end{split}
\end{equation*}
where $b_{t_0,B}(t)=\int^{t}_{-\infty}\frac{1}{B} \mathbb{I}_{\{-t_0-B< s < -t_0\}}ds$,
$v_{t_0,B}(t)=\int^{t}_{-t_0}b_{t_0,B}(s)ds-t_0$.
\end{Lemma}

We would like to recall the following notations in section \ref{sec:Main result}.
Let $M$ be an $n-$dimensional complex manifold.
Assume that $(M,X,Z)$  satisfies condition $(A)$. Let $K_M$ be the canonical line bundle on $M$.
Let $dV_M$ be a continuous volume form on $M$. Let $F\not\equiv 0$ be a holomorphic function on $M$. Let $\psi$ be a plurisubharmonic function on $M$.
Let $E$ be a holomorphic vector bundle on $M$ with rank $r$.
Let $\hat{h}$ be a smooth metric on $E$. Let $h$ be a \textit{measurable metric} on $E$. Denote $\tilde{h}:=he^{-\psi}$. Let $(M,E,\Sigma,M_j,\tilde{h},\tilde{h}_{j,s})$ be a singular metric on $E$. Assume that $\Theta_{\tilde{h}}(E)\ge^s_{Nak} 0$. Let $c(t)\in \tilde{P}_{T,M,\Psi,h}$.

 Let $T\in[-\infty,+\infty)$.  Denote $$\Psi:=\min\{\psi-2\log|F|,-T\}.$$
If $F(z)=0$ for some $z \in M$, we set $\Psi(z)=-T$.
Let $T_1>T$ be a real number. Denote $\tilde{M}:=\max\{\psi+T_1,2\log|F|\}$. Denote
$$\Psi_1:=\min\{\psi-2\log|F|,-T_1\}.$$
If $F(z)=0$ for some $z \in M$, we set $\Psi_1(z)=-T_1$.

It follows from Lemma \ref{L2 method} that we have the following lemma.
\begin{Lemma}
Let $(M,X,Z)$  satisfy condition $(A)$. Let $B \in (0, +\infty)$ and $t_0> T_1>T$ be arbitrarily given.
Let $f$ be an $E$-valued holomorphic $(n,0)$ form on $\{\Psi< -t_0\}$ such that
\begin{equation}
\int_{\{\Psi<-t_0\}} {|f|}^2_hc(-\Psi)<+\infty,
\label{condition of lemma 2.2}
\end{equation}
Then there exists an $E$-valued holomorphic $(n,0)$ form $\tilde{F}$ on $M$  such that
\begin{equation}
\begin{split}
&\int_{M}|\tilde{F}-(1-b_{t_0,B}(\Psi_1))fF^{1+\delta}|^2_{\tilde{h}}e^{v_{t_0,B}(\Psi_1)-\delta \tilde{M}}c(-v_{t_0,B}(\Psi_1))\\
\le & \left(\frac{1}{\delta}c(T_1)e^{-T_1}+\int_{T_1}^{t_0+B}c(t)e^{-t}dt\right)\int_M \frac{1}{B} \mathbb{I}_{\{-t_0-B< \Psi_1 < -t_0\}}  {|fF|}^2_{\tilde{h}}
,
\end{split}
\end{equation}
where  $b_{t_0,B}(t)=\int^{t}_{-\infty}\frac{1}{B} \mathbb{I}_{\{-t_0-B< s < -t_0\}}ds$ and
$v_{t_0,B}(t)=\int^{t}_{-t_0}b_{t_0,B}(s)ds-t_0$.
\label{L2 method for c(t)}
\end{Lemma}

\begin{proof} We note that $\{\Psi<-t_0\}=\{\Psi_1<-t_0\}$ and $\Psi_1=\Psi=\psi-2\log|F|$ on $\{\Psi<-t_0\}$. It follows from inequality \eqref{condition of lemma 2.2}, $\tilde{h}=he^{-\psi}$ and $c(t)e^{-t}$ is decreasing with respect to $t$ that
$$\int_M \frac{1}{B} \mathbb{I}_{\{-t_0-B< \Psi_1 < -t_0\}}  {|fF|}^2_{\tilde{h}}
<+\infty.$$

  As $c(t)\in \tilde{P}_{T,M,\Psi,h}$, $\{\Psi<-t_0\}=\{\Psi_1<-t_0\}$ and $\Psi_1=\Psi$ on $\{\Psi<-t_0\}$, there  exists a closed subset $E_0\subset Z\cap \{\Psi=-\infty\}$ such that for any compact subset $K\subset M\backslash E_0$, $|e|_h^2c(-\Psi)\ge C_K|e|^2_{\hat{h}}$ on $K\cap \{\Psi_1<-t_0\} $, where $C_K>0$ is a constant and $e$ is any $E$-valued holomorphic $(n,0)$ form on $\{\Psi_1<-t_0\}$.
It follows from inequality \eqref{condition of lemma 2.2} that we have
\begin{equation}\nonumber
\int_{K\cap \{\Psi_1<-t_0\}}|f|^2_{\hat{h}}<+\infty.
\end{equation}

As $(M,X,Z)$  satisfies condition $(A)$,  $M\backslash (Z\cup X)$ is a weakly pseudoconvex K\"ahler manifold. It follows from Lemma \ref{L2 method} that there exists an $E$-valued holomorphic $(n,0)$ form $\tilde{F}_Z$ on $M\backslash (Z\cup X)$ such that
\begin{equation*}
  \begin{split}
      & \int_{M\backslash (Z\cup X)}|\tilde{F}_Z-(1-b_{t_0,B}(\Psi_1))fF^{1+\delta}|^2_{\tilde{h}}e^{v_{t_0,B}(\Psi_1)-\delta \tilde{M}}c(-v_{t_0,B}(\Psi_1)) \\
      \le & \left(\frac{1}{\delta}c(T_1)e^{-T_1}+\int_{T_1}^{t_0+B}c(s)e^{-s}ds \right)
       \int_{M}\frac{1}{B}\mathbb{I}_{\{-t_0-B<\Psi_1<-t_0\}}|fF|^2_{\tilde{h}}<+\infty.
  \end{split}
\end{equation*}

For any $z\in \left((Z\cup X)\backslash E_0\right)$, there exists an open neighborhood $V_z$ of $z$ such that $V_z\Subset M\backslash E_0$.

As $(M,E,\Sigma,M_j,\tilde{h},\tilde{h}_{j,s})$ is a singular metric on $E$ and $\Theta_{\tilde{h}}(E)\ge^s_{Nak} 0$,  there exist a relatively compact subset $M_{j'}\subset M$ containing $V_z$ and a $C^2$ smooth metric $\tilde{h}_{j',1}\le \tilde{h}$ on $V_z\subset M_{j'}$.
Note that $\delta \tilde{M}$ is a plurisubharmonic function on $M$. As $c(t)e^{-t}$ is decreasing with respect to $t$ and $v_{t_0,B}(\Psi_1)\ge -t_0-\frac{B}{2}$, we have $c(-v_{t_0,B}(\Psi_1))e^{v_{t_0,B}(\Psi_1)}\ge c(t_0+\frac{B}{2})e^{-t_0-\frac{B}{2}}>0$. Denote $C:=\inf\limits_{V_z}e^{v_{t_0,B}(\Psi_1)-\delta M}c(-v_{t_0,B}(\Psi_1))$, we know $C>0$. On $V_z$, as both $\hat{h}$ and $\tilde{h}_{j',1}$ are continuous, we have $\tilde{h}_{j',1}\le \tilde{C}\hat{h}$ for some $\tilde{C}>0$. Then we have
\begin{equation*}
  \begin{split}
  &\int_{V_z\backslash (Z\cup X)}|\tilde{F}_Z|^2_{\tilde{h}_{j',1}}\\
  \le & 2\int_{V_z\backslash (Z\cup X)}|\tilde{F}_Z-(1-b_{t_0,B}(\Psi_1))fF^{1+\delta}|^2_{\tilde{h}_{j',1}}
  +2\int_{V_z\backslash (Z\cup X)}|(1-b_{t_0,B}(\Psi_1))fF^{1+\delta}|^2_{\tilde{h}_{j',1}}
 \\
 \le& 2\int_{V_z\backslash (Z\cup X)}|\tilde{F}_Z-(1-b_{t_0,B}(\Psi_1))fF^{1+\delta}|^2_{\tilde{h}}
  +2\sup_{V_z}|F^{1+\delta}|^2\int_{\{\Psi_1<-t_0\}\cap V_z}|f|^2_{\tilde{h}_{j',1}}\\
  \le &
  \frac{2}{C}\bigg(\int_{M\backslash (Z\cup X)}|\tilde{F}_Z-(1-b_{t_0,B}(\Psi_1))fF^{1+\delta}|^2_{\tilde{h}}e^{v_{t_0,B}(\Psi_1)-\delta M}c(-v_{t_0,B}(\Psi_1))\bigg)\\
  &+\tilde{C}\sup_{V_z}|F^{1+\delta}|^2\int_{\{\Psi_1<-t_0\}\cap V_z}|f|^2_{\hat{h}}\\
  <&+\infty.
  \end{split}
\end{equation*}

As $Z\cup X$ is locally negligible with respect to $L^2$ holomorphic function, we can find an $E$-valued holomorphic extension $\tilde{F}_{E_0}$ of $\tilde{F}_Z$ from $M\backslash (Z\cup X)$ to $M\backslash E_0$ such that

\begin{equation*}
  \begin{split}
      & \int_{M\backslash E_0}|\tilde{F}_{E_0}-(1-b_{t_0,B}(\Psi_1))fF^{1+\delta}|^2_{\tilde{h}}e^{v_{t_0,B}(\Psi_1)-\delta \tilde{M}}c(-v_{t_0,B}(\Psi_1)) \\
      \le & \left(\frac{1}{\delta}c(T_1)e^{-T_1}+\int_{T_1}^{t_0+B}c(s)e^{-s}ds\right)
       \int_{M}\frac{1}{B}\mathbb{I}_{\{-t_0-B<\Psi_1<-t_0\}}|fF|^2_{\tilde{h}}.
  \end{split}
\end{equation*}

 Note that $E_0\subset\{\Psi=-\infty\}\subset\{\Psi<-t_0\}$ and $\{\Psi<-t_0\}$ is open, then for any $z\in E_0$, there exists an open neighborhood $U_z$ of $z$ such that $U_z\Subset\{\Psi<-t_0\}=\{\Psi_1<-t_0\}$. As $(M,E,\Sigma,M_j,\tilde{h},\tilde{h}_{j,s})$ is a singular metric on $E$ and $\Theta_{\tilde{h}}(E)\ge^s_{Nak} 0$,  there exist a relatively compact subset $M_{j''}\subset M$ containing $U_z$ and a $C^2$ smooth metric $\tilde{h}_{j'',1}\le \tilde{h}$ on $U_z\subset M_{j''}$.
 As $v_{t_0,B}(t)\ge -t_0-\frac{B}{2}$, we have $c(-v_{t_0,B}(\Psi_1))e^{v_{t_0,B}(\Psi_1)}\ge c(t_0+\frac{B}{2})e^{-t_0-\frac{B}{2}}>0$. Note that $\delta \tilde{M}$ is plurisubharmonic on $M$. Thus we have
\begin{equation*}
  \begin{split}
      & \int_{U_z\backslash E_0}|\tilde{F}_{E_0}-(1-b_{t_0,B}(\Psi))fF^{1+\delta}|^2_{\tilde{h}_{j'',1}} \\
      \le  & \int_{U_z\backslash E_0}|\tilde{F}_{E_0}-(1-b_{t_0,B}(\Psi))fF^{1+\delta}|^2_{\tilde{h}} \\
      \le & \frac{1}{C_1}\int_{U_z\backslash E_0}|\tilde{F}_{E_0}-(1-b_{t_0,B}(\Psi_1))fF^{1+\delta}|^2_{\tilde{h}}e^{v_{t_0,B}(\Psi_1)-\delta \tilde{M}}c(-v_{t_0,B}(\Psi_1))
      <+\infty,
  \end{split}
\end{equation*}
where $C_1$ is some positive number.

As $U_z\Subset\{\Psi<-t_0\}$, we have
\begin{equation*}
  \begin{split}
       \int_{U_z\backslash E_0}|(1-b_{t_0,B}(\Psi))fF^{1+\delta}|^2_{\tilde{h}_{j'',1}}
      \le
      \left(\sup_{U_z}|F^{1+\delta}|^2\right)\int_{U_z}|f|^2_{\tilde{h}_{j'',1}} <+\infty.
  \end{split}
\end{equation*}
Hence we have
$$\int_{U_z\backslash E_0}|\tilde{F}_{E_0}|^2_{\tilde{h}_{j'',1}}<+\infty. $$
As $E_0$ is contained in some analytic subset of $M$, we can find a holomorphic extension $\tilde{F}$ of $\tilde{F}_{E_0}$ from $M\backslash E_0$ to $M$ such that

\begin{equation}
\begin{split}
&\int_{M}|\tilde{F}-(1-b_{t_0,B}(\Psi_1))fF^{1+\delta}|^2_{\tilde{h}}e^{v_{t_0,B}(\Psi_1)-\delta \tilde{M}}c(-v_{t_0,B}(\Psi_1))\\
\le & \left(\frac{1}{\delta}c(T_1)e^{-T_1}+\int_{T_1}^{t_0+B}c(t)e^{-t}dt\right)\int_M \frac{1}{B} \mathbb{I}_{\{-t_0-B< \Psi_1 < -t_0\}}  |fF|^2_{\tilde{h}}.
\end{split}
\end{equation}
Lemma \ref{L2 method for c(t)} is proved.
\end{proof}

Let $T\in[-\infty,+\infty)$. Let $c(t)\in \tilde{P}_{T,M,\Psi.h}$.
Following the notations in Lemma \ref{L2 method for c(t)} and using the result of Lemma \ref{L2 method for c(t)}, we have the following lemma, which will be used to prove Theorem \ref{main theorem}.
\begin{Lemma}
\label{L2 method in JM concavity}
Let $(M,X,Z)$  satisfy condition $(A)$. Let $B \in (0, +\infty)$ and $t_0>t_1> T$ be arbitrarily given.
Let $f$ be a holomorphic $(n,0)$ form on $\{\Psi< -t_0\}$ such that
\begin{equation}
\int_{\{\Psi<-t_0\}} {|f|}^2_{h}c(-\Psi)<+\infty,
\label{condition of JM concavity}
\end{equation}
Then there exists an $E$-valued holomorphic $(n,0)$ form $\tilde{F}$ on $\{\Psi<-t_1\}$ such that

 \begin{equation*}
  \begin{split}
      & \int_{\{\Psi<-t_1\}}|\tilde{F}-(1-b_{t_0,B}(\Psi))f|^2_h e^{v_{t_0,B}(\Psi)-\Psi}c(-v_{t_0,B}(\Psi)) \\
      \le & \left(\int_{t_1}^{t_0+B}c(s)e^{-s}ds\right)
       \int_{M}\frac{1}{B}\mathbb{I}_{\{-t_0-B<\Psi<-t_0\}}|f|^2_he^{-\Psi},
  \end{split}
\end{equation*}
where $b_{t_0,B}(t)=\int^{t}_{-\infty}\frac{1}{B} \mathbb{I}_{\{-t_0-B< s < -t_0\}}ds$,
$v_{t_0,B}(t)=\int^{t}_{-t_0}b_{t_0,B}(s)ds-t_0$.
\end{Lemma}
\begin{proof}[Proof of Lemma \ref{L2 method in JM concavity}]

Denote that
$$\tilde{\Psi}:=\min\{\psi-2\log|F|,-t_1\}.$$

As $t_0>t_1>T$, we have $\{\tilde{\Psi}<-t_0\}=\{\Psi<-t_0\}$. It follows from inequality \eqref{condition of JM concavity} and Lemma \ref{L2 method for c(t)} that there exists an $E$-valued holomorphic $(n,0)$ form $\tilde{F}_{\delta}$ on $M$ such that

\begin{equation*}
  \begin{split}
      & \int_{M}|\tilde{F}_{\delta}-(1-b_{t_0,B}(\tilde\Psi))fF^{1+\delta}|^2_{\tilde{h}}e^{v_{t_0,B}(\tilde \Psi)-\delta \tilde{M}}c(-v_{t_0,B}(\tilde \Psi)) \\
      \le & \left(\frac{1}{\delta}c(t_1)e^{-t_1}+\int_{t_1}^{t_0+B}c(s)e^{-s}ds\right)
       \int_{M}\frac{1}{B}\mathbb{I}_{\{-t_0-B<\tilde\Psi<-t_0\}}{|fF|}^2_{\tilde{h}}.
  \end{split}
\end{equation*}

Note that on $\{\Psi<-t_1\}$, we have $\Psi=\tilde{\Psi}=\psi-2\log|F|$.  Hence
\begin{equation}\label{1st formula in L2 method JM concavity}
\begin{split}
    & \int_{\{\Psi<-t_1\}}|\tilde{F}_{\delta}-(1-b_{t_0,B}(\Psi))fF^{1+\delta}|^2_{\tilde{h}}e^{v_{t_0,B}( \Psi)-\delta \tilde{M}}c(-v_{t_0,B}(\Psi)) \\
    = &\int_{\{\Psi<-t_1\}}|\tilde{F}_{\delta}-(1-b_{t_0,B}(\tilde\Psi))fF^{1+\delta}|^2_{\tilde{h}}e^{v_{t_0,B}(\tilde \Psi)-\delta \tilde{M}}c(-v_{t_0,B}(\tilde \Psi))\\
     \le &
      \int_{M}|\tilde{F}_{\delta}-(1-b_{t_0,B}(\tilde\Psi))fF^{1+\delta}|^2_{\tilde{h}}e^{v_{t_0,B}(\tilde \Psi)-\delta \tilde{M}}c(-v_{t_0,B}(\tilde \Psi)) \\
     \le & \left(\frac{1}{\delta}c(t_1)e^{-t_1}+\int_{t_1}^{t_0+B}c(s)e^{-s}ds\right)
       \int_{M}\frac{1}{B}\mathbb{I}_{\{-t_0-B<\tilde\Psi<-t_0\}}{|fF|}^2_{\tilde{h}}\\
       =&
     \left(\frac{1}{\delta}c(t_1)e^{-t_1}+\int_{t_1}^{t_0+B}c(s)e^{-s}ds\right)
       \int_{M}\frac{1}{B}\mathbb{I}_{\{-t_0-B<\Psi<-t_0\}}{|fF|}^2_{\tilde{h}}<+\infty.
\end{split}
\end{equation}

 Let $F_{\delta}:=\frac{\tilde{F}_{\delta}}{F^{\delta}}$ be an $E$-valued holomorphic $(n,0)$ form on $\{\Psi<-t_1\}$. Then it follows from \eqref{1st formula in L2 method JM concavity}  that
\begin{equation}\label{2nd formula in L2 method JM concavity}
\begin{split}
    & \int_{\{\Psi<-t_1\}}|F_{\delta}-(1-b_{t_0,B}(\Psi))fF|^2_{\tilde{h}}
    e^{v_{t_0,B}(\Psi)}c(-v_{t_0,B}(\Psi)) \\
    \le&
     \bigg(\frac{1}{\delta}c(t_1)e^{-t_1}+\int_{t_1}^{t_0+B}c(s)e^{-s}ds\bigg)
       \int_{M}\frac{1}{B}\mathbb{I}_{\{-t_0-B<\Psi<-t_0\}}{|fF|}^2_{\tilde{h}}.
\end{split}
\end{equation}

Note that $e^{v_{t_0,B}(\Psi)}c(-v_{t_0,B}(\Psi))\ge \left(c(t_0+\frac{2}{B})e^{-t_0-\frac{2}{B}}\right)>0$. As $c(t)\in \tilde{P}_{T,M,\Psi.h}$, there exists a closed subset $E_0$ of $M$ such that $E_0\subset Z\cap \{\Psi(z)=-\infty\}$ (where $Z$ is an analytic subset of $M$) and for any compact subset $K\subset M\backslash E_0$, $|e_x|_h^2c(-\Psi)\ge C_K|e_x|^2_{\hat{h}}$ for any $ x\in K\cap\{\Psi<-t_0\}$, where $C_K>0$ is a constant and $e_x\in E_x$.
Let $K$ be any compact subset of $M\backslash E_0$. As $(M,E,\Sigma,M_j,\tilde{h},\tilde{h}_{j,s})$ is a singular metric on $E$ and $\Theta_{\tilde{h}}(E)\ge^s_{Nak} 0$,  there exist a relatively compact subset $M_{j_K}\subset M$ containing $K$ and a $C^2$ smooth metric $\tilde{h}_{j_K,1}\le \tilde{h}$ on $K\subset M_{j_K}$. It follows from inequality \eqref{2nd formula in L2 method JM concavity} that we have
 $$\sup_{\delta} \int_{\{\Psi<-t_1\}\cap K}|F_{\delta}-(1-b_{t_0,B}(\Psi))fF|^2_{\tilde{h}_{j_K,1}}<+\infty.$$

We also note that
$$\int_{\{\Psi<-t_1\}\cap K}|(1-b_{t_0,B}(\Psi))fF|^2_{\tilde{h}_{j_K,1}}\le
\left(\sup_{K}|F|^2\right)\int_{\{\Psi<-t_0\}\cap K}|f|^2_{\tilde{h}_{j_K,1}}<+\infty.$$
Then we know that
$$\sup_{\delta} \int_{\{\Psi<-t_1\}\cap K}|F_{\delta}|^2_{\tilde{h}_{j_K,1}}<+\infty.$$
By Montel theorem and diagonal method, there exists a subsequence of $\{F_\delta\}$ (also denoted by $F_\delta$) compactly convergent to a holomorphic $(n,0)$ form $\tilde{F}_1$ on $\{\Psi<-t_1\}\backslash E_0$.
It follows from Fatou's Lemma and inequality \eqref{2nd formula in L2 method JM concavity} that we have

\begin{equation}\label{3rd formula in L2 method JM concavity}
\begin{split}
  & \int_{\{\Psi<-t_1\}\backslash E_0}|\tilde{F}_1-(1-b_{t_0,B}(\Psi))fF|^2_{\tilde{h}}
    e^{v_{t_0,B}(\Psi)}c(-v_{t_0,B}(\Psi)) \\
    \le &\liminf_{\delta\to +\infty} \int_{\{\Psi<-t_1\}\backslash E_0}|F_{\delta}-(1-b_{t_0,B}(\Psi))fF|^2_{\tilde{h}}
    e^{v_{t_0,B}(\Psi)}c(-v_{t_0,B}(\Psi)) \\
      \le &\liminf_{\delta\to +\infty} \int_{\{\Psi<-t_1\}}|F_{\delta}-(1-b_{t_0,B}(\Psi))fF|^2_{\tilde{h}}
    e^{v_{t_0,B}(\Psi)}c(-v_{t_0,B}(\Psi)) \\
       \le&\liminf_{\delta\to +\infty}
     \bigg(\frac{1}{\delta}c(t_1)e^{-t_1}+\int_{t_1}^{t_0+B}c(s)e^{-s}ds\bigg)
       \int_{M}\frac{1}{B}\mathbb{I}_{\{-t_0-B<\Psi<-t_0\}}{|fF|}^2_{\tilde{h}}\\
       \le &\left(\int_{t_1}^{t_0+B}c(s)e^{-s}ds\right)
       \int_{M}\frac{1}{B}\mathbb{I}_{\{-t_0-B<\Psi<-t_0\}}{|fF|}^2_{\tilde{h}}.
\end{split}
\end{equation}

Note that $E_0\subset\{\Psi=-\infty\}\subset\{\Psi<-t_1\}$ and $\{\Psi<-t_1\}$ is open, then for any $z\in E_0$, there exists an open neighborhood $U_z$ of $z$ such that $U_z\Subset\{\Psi<-t_1\}$. As $(M,E,\Sigma,M_j,\tilde{h},\tilde{h}_{j,s})$ is a singular metric on $E$ and $\Theta_{\tilde{h}}(E)\ge^s_{Nak} 0$,  there exist a relatively compact subset $M_{j''}\subset M$ containing $U_z$ and a $C^2$ smooth metric $\tilde{h}_{j'',1}\le \tilde{h}$ on $V_z\subset M_{j''}$. As $v_{t_0,B}(t)\ge -t_0-\frac{B}{2}$, we have $c(-v_{t_0,B}(\Psi_1))e^{v_{t_0,B}(\Psi_1)}\ge c(t_0+\frac{B}{2})e^{-t_0-\frac{B}{2}}>0$. Thus we have
\begin{equation*}
  \begin{split}
      & \int_{U_z\backslash E_0}|\tilde{F}_1-(1-b_{t_0,B}(\Psi))fF|^2_{\tilde{h}_{j'',1}} \\
      \le  & \int_{U_z\backslash E_0}|\tilde{F}_1-(1-b_{t_0,B}(\Psi))fF|^2_{\tilde{h}} \\
      \le & \frac{1}{C_1}\int_{U_z\backslash E_0}|\tilde{F}_1-(1-b_{t_0,B}(\Psi))fF|^2_{\tilde{h}}e^{v_{t_0,B}(\Psi)}c(-v_{t_0,B}(\Psi))
      <+\infty,
  \end{split}
\end{equation*}
where $C_1$ is some positive number.

As $U_z\Subset\{\Psi<-t_1\}$, we have
\begin{equation*}
  \begin{split}
       \int_{U_z\backslash E_0}|(1-b_{t_0,B}(\Psi))fF|^2_{\tilde{h}_{j'',1}}
      \le
      \left(\sup_{U_z}|F|^2\right)\int_{U_z}|f|^2_{\tilde{h}_{j'',1}} <+\infty.
  \end{split}
\end{equation*}
Hence we have
$$\int_{U_z\backslash E_0}|\tilde{F}_1|^2_{\tilde{h}_{j'',1}}<+\infty. $$
As $E_0$ is contained in some analytic subset of $M$, we can find a holomorphic extension $\tilde{F}_0$ of $\tilde{F}_1$ from $\{\Psi<-t_1\}\backslash E_0$ to $\{\Psi<-t_1\}$ such that

\begin{equation}\label{3rd formula in L2 method JM concavity}
\begin{split}
  & \int_{\{\Psi<-t_1\}}|\tilde{F}_1-(1-b_{t_0,B}(\Psi))fF|^2_{\tilde{h}}
    e^{v_{t_0,B}(\Psi)}c(-v_{t_0,B}(\Psi)) \\
       \le &\left(\int_{t_1}^{t_0+B}c(s)e^{-s}ds\right)
       \int_{M}\frac{1}{B}\mathbb{I}_{\{-t_0-B<\Psi<-t_0\}}{|fF|}^2_{\tilde{h}}.
\end{split}
\end{equation}

Denote $\tilde{F}:=\frac{\tilde{F}_0}{F}$. Note that $\tilde{h}=he^{-\psi}$ and on $\{\Psi<-t_1\}$, we have $\Psi=\psi-2\log|F|$. It follows from inequality \eqref{3rd formula in L2 method JM concavity} that we have
\begin{equation}\nonumber
\begin{split}
  & \int_{\{\Psi<-t_1\}}|\tilde{F}-(1-b_{t_0,B}(\Psi))f|^2_h
    e^{v_{t_0,B}(\Psi)-\Psi}c(-v_{t_0,B}(\Psi)) \\
       \le &\left(\int_{t_1}^{t_0+B}c(s)e^{-s}ds\right)
       \int_{M}\frac{1}{B}\mathbb{I}_{\{-t_0-B<\Psi<-t_0\}}|f|^2_he^{-\Psi}.
\end{split}
\end{equation}

Lemma \ref{L2 method in JM concavity} is proved.

\end{proof}
\subsection{Properties of $\mathcal{O}_{M,z_0}$-module $J_{z_0}$}
\label{sec:properties of module}
In this section, we present some properties of $\mathcal{O}_{M,z_0}$-module $J_{z_0}$.

We recall the following property of closedness of holomorphic functions on a neighborhood of $o$.

\begin{Lemma}[Closedness of Submodules, see \cite{G-R}] Let $N$ be a submodule of $\mathcal{O}^q_{\mathbb{C}^n,0}$, $1\le q < +\infty$, let $f_j \in \mathcal{O}^q_{\mathbb{C}^n}(U)$ be a sequence of $q$-tuples holomorphic in an open neighborhood $U$ of the origin. Assume that the $f_j$ converge uniformly in $U$ towards a $q$-tuple $f \in \mathcal{O}^q_{\mathbb{C}^n}(U)$, assume furthermore that all germs $f_{j,0}$ belong to $N$. Then $f_0 \in N$.
\label{closedness}
\end{Lemma}

We recall the following lemma which will be used in the proof of Lemma \ref{closedness of module}.
\begin{Lemma}
	\label{l:converge}
	Let $M$ be a complex manifold. Let $dV_M$ be a continuous volume form on $M$. Let $S$ be an analytic subset of $M$.  	Let $E$ be a holomorphic vector bundle on $M$ with rank $r$.
Let $\hat{h}$ be a smooth metric on $E$. Let $h$ be a \textit{measurable metric} on $E$.

	Let $\{g_j\}_{j=1,2,...}$ be a sequence of nonnegative Lebesgue measurable functions on $M$, which satisfies that $g_j$ are almost everywhere convergent to $g$ on  $M$ when $j\rightarrow+\infty$,  where $g$ is a nonnegative Lebesgue measurable function on $M$. Assume that for any compact subset $K$ of $M\backslash S$, we have  $|e_x|_h^2g_j\ge C_K|e_x|^2_{\hat{h}}$ for any $ x\in K$ and any $j\in\mathbb{Z}_{+}$, where $C_K>0$ is a constant and $e_x$ is any section of $E_x$.
	
 Let $\{F_j\}_{j=1,2,...}$ be a sequence of $E$-valued holomorphic $(n,0)$ forms on $M$. Assume that $\liminf_{j\rightarrow+\infty}\int_{M}|F_j|^2_hg_j\leq C$, where $C$ is a positive constant. Then there exists a subsequence $\{F_{j_l}\}_{l=1,2,...}$, which satisfies that $\{F_{j_l}\}$ is uniformly convergent to an $E$-valued holomorphic $(n,0)$ form $F$ on $M$ on any compact subset of $M$ when $l\rightarrow+\infty$, such that
 $$\int_{M}|F|^2_h g\leq C.$$
\end{Lemma}
\begin{proof} Let $(U\subset\subset M, \theta)$ be a local trivialization of $E$, where $E$ is a holomorphic vector bundle on $M$. For any $f=(f_1,\ldots,f_r)\in H^0(U,\mathcal{O}(E))$, denote $|f|^2_1:=\sum_{i=1}^{r}|f_i|^2$. Then there exists a constant $\lambda>0$ such that $\frac{1}{\lambda} |f|^2_1\le|f|_{\hat{h}}^2\le \lambda |f|^2_1$ on $U$. Let $\tilde{K}$ be any compact subset of $U$ and $\tilde{S}$ be an analytic subset of $M$. By Local Parametrization Theorem (see \cite{Demaillybook}) and Maximum Principle, there exists a compact subset $\tilde{K}_1\subset U\backslash \tilde{S}$ such that
$$\sup_{z\in \tilde{K}}|f(z)|^2_{1}\le \tilde{C}_1 \sup_{z\in \tilde{K}_1}|f(z)|^2_{1}.$$
Hence we have
$$\sup_{z\in \tilde{K}}|f(z)|^2_{\hat{h}}\le \lambda \sup_{z\in \tilde{K}}|f(z)|^2_{1}\le \lambda \tilde{C}_1 \sup_{z\in \tilde{K}_1}|f(z)|^2_{1}\le \lambda^2 \tilde{C}_1\sup_{z\in \tilde{K}_1}|f(z)|^2_{\hat{h}}.$$
Recall that $S$ is an analytic subset of $M$. By the argument above, for any compact set $K\subset\subset M$, there exists $K_1\subset M\backslash S$ such that for any $j\ge 0$, we have
\begin{equation}\label{l:converge formula 1}
\sup_{z\in K}\frac{|F_j(z)|^2_{\hat{h}}}{dV_M}\le C_1 \sup_{z\in K_1}\frac{|F_j(z)|^2_{\hat{h}}}{dV_M},
\end{equation}
where $C_1>0$ (depends on $K$ and $\hat{h}$) is a real number. Then there exists a compact subset $K_2\subset M\backslash S$ such that $K_1\subset K_2$ and for any $z\in K_1$ and $j\ge 0$,
\begin{equation}\nonumber
\frac{|F_j(z)|^2_{\hat{h}}}{dV_M}\le C_2 \int_{K_2}|F_j(z)|^2_{\hat{h}}\le \frac{C_2}{C_{K_2}} \int_{K_2}|F_j(z)|^2_{h}g_j\le \frac{C_2}{C_{K_2}} C<+\infty.
\end{equation}
Hence we know that
$\sup_{K_1}\frac{|F_j(z)|^2_{\hat{h}}}{dV_M}$ is uniformly bounded with respect to $j$. Then it follows from inequality \eqref{l:converge formula 1}, Montel theorem and diagonal method that we have a subsequence of $\{F_j\}$ (still denoted by $\{F_j\}$) uniformly converges to an $E$-valued holomorphic $(n,0)$ form $F$ on any compact subset of $M$. It follows from Fatou's Lemma and $\liminf_{j\rightarrow+\infty}\int_{M}|F_j|^2_hg_j\leq C$ that we have
\begin{equation}\nonumber
\int_{M}|F|^2_hg\le \liminf_{j\rightarrow+\infty}\int_{M}|F_j|^2_hg_j\leq C.
\end{equation}
Lemma \ref{l:converge} has been proved.
\end{proof}

Since the properties of $J_{z_0}$
 is local, we assume that $D$ is a pseudoconvex domain in $\mathbb{C}^n$ containing the origin $o\in \mathbb{C}^n$. Let $F$ be a holomorphic function on $D$.
Let $f=(f_1,f_2,\ldots,f_r)$ be a holomorphic section of $D\times \mathbb{C}^r$. Let $\psi$ be a plurisubharmonic function on $D$. Let $h$ be a \textit{measurable metric} on $D\times \mathbb{C}^r$. Denote $\tilde{h}:=he^{-\psi}$. Let $(D,D\times \mathbb{C}^r,\Sigma,D_j,\tilde{h},\tilde{h}_{j,s})$ be a singular metric on $E:=D\times \mathbb{C}^r$ which satisfies $\Theta_{\tilde{h}}(E)\ge^s_{Nak} 0$. Let $T\in [-\infty,+\infty)$. Denote $$\Psi:=\min\{\psi-2\log|F|,-T\}.$$
If $F(z)=0$ for some $z \in M$, we set $\Psi(z)=-T$. Let $T_1>T$ be a real number.

Denote
$$\tilde{M}:=\max\{\psi+T,2\log|F|\},$$
$$\varphi_1:=2\max\{\psi+T_1,2\log|F|\},$$
and
$$\Psi_1:=\min\{\psi-2\log|F|,-T_1\}.$$
If $F(z)=0$ for some $z \in M$, we set $\Psi_1(z)=-T_1$. We also note that by definition $I(h,\Psi_1)_o=I(h,\Psi)_o$.

Let $c(t)$ be a positive measurable function on $(T,+\infty)$ such that $c(t)\in\tilde{P}_{T,D,\Psi,h}$.

Let $dV_D$ be a continuous volume form on $D$.
Denote that $H_o:=\{f_o\in J(E,\Psi)_o:\int_{\{\Psi<-t\}\cap V_0}|f|^2_hc(-\Psi)dV_D<+\infty \text{ for some }t>T \text{ and } V_0 \text{ is an open neighborhood of o}\}$ and
$\mathcal{H}_o:=\{(F,o)\in \mathcal{O}^r_{\mathbb{C}^n,o}:\int_{U_0}|F|^2_he^{-\varphi_1}c(-\Psi_1)dV_D<+\infty \text{ for some open neighborhood} \ U_0 \text{ of } o\}$.

As  $c(t)\in\tilde{P}_{T,D,\Psi,h}$, $c(t)e^{-t}$ is decreasing with respect to $t$ and we have $I(h,\Psi_1)_o=I(h,\Psi)_o\subset H_o$.
We also note that $\mathcal{H}_o$ is an submodule of $\mathcal{O}^r_{\mathbb{C}^n,o}$.

\begin{Lemma}\label{construction of morphism} For any $f_o\in H_o$, there exist a pseudoconvex domain  $D_0\subset D$ containing $o$ and a holomorphic section $\tilde{F}$ of $D\times \mathbb{C}^r$  on $D_0$ such that $(\tilde{F},o)\in \mathcal{H}_o$ and
$$\int_{\{\Psi_1<-t_1\}\cap D_0}|\tilde{F}-fF^2|^2_he^{-\varphi_1-\Psi_1}<+\infty,$$
for some $t_1>T_1$.
\end{Lemma}
\begin{proof}It follows from $f_o\in H_o$ that there exist $t_0>T_1>T$ and a pseudoconvex domain  $D_0\Subset D$ containing $o$ such that
\begin{equation}\label{construction of morphism formula 1}
\int_{\{\Psi<-t_0\}\cap D_0}|f|^2_hc(-\Psi)<+\infty.
\end{equation}
Then it follows from Lemma \ref{L2 method for c(t)} that there exists a holomorphic section $\tilde F$ of $D\times \mathbb{C}^r$ on $D_0$ such that
\begin{equation}\label{construction of morphism formula 1+}
  \begin{split}
      & \int_{D_0}|\tilde{F}-(1-b_{t_0}(\Psi_1))fF^{2}|^2_{\tilde{h}}e^{v_{t_0}(\Psi_1)- \tilde{M}}c(-v_{t_0}(\Psi_1)) \\
      \le & \left(c(T_1)e^{-T_1}+\int_{T_1}^{t_0+1}c(s)e^{-s}ds\right)
       \int_{D_0}\mathbb{I}_{\{-t_0-1<\Psi_1<-t_0\}}|fF|^2_{\tilde{h}},
  \end{split}
\end{equation}
where $b_{t_0}(t)=\int^{t}_{-\infty} \mathbb{I}_{\{-t_0-1< s < -t_0\}}ds$,
$v_{t_0}(t)=\int^{t}_{-t_0}b_{t_0}(s)ds-t_0$. Note that $\tilde{h}=he^{-\psi}$, $\psi+ \tilde{M}=\varphi_1+\Psi_1$ and $\Psi=\Psi_1=\psi-2\log|F|$ on $\{\Psi<-t_0\}$. Hence, by \eqref{construction of morphism formula 1+}, we have
\begin{equation*}
  \begin{split}
      & \int_{D_0}|\tilde{F}-(1-b_{t_0}(\Psi_1))fF^{2}|^2_{h}e^{-\varphi_1-\Psi_1+v_{t_0}(\Psi_1)}c(-v_{t_0}(\Psi_1)) \\
      \le & \left(c(T_1)e^{-T_1}+\int_{T_1}^{t_0+1}c(s)e^{-s}ds\right)
       \int_{D_0}\mathbb{I}_{\{-t_0-1<\Psi_1<-t_0\}}|f|^2_{h}e^{-\Psi_1}.
  \end{split}
\end{equation*}
Denote $C:=c(T_1)e^{-T_1}+\int_{T_1}^{t_0+B}c(s)e^{-s}ds$, we note that $C$ is a positive number.

 As $v_{t_0}(t)>-t_0-1$, we have $e^{v_{t_0}(\Psi)}c(-v_{t_0}(\Psi))\ge c(t_0+1)e^{-(t_0+1)}>0$. As $b_{t_0}(t)\equiv 0$ on $(-\infty,-t_0-1)$, we have
\begin{equation}\label{construction of morphism formula 2}
\begin{split}
   &\int_{D_0\cap\{\Psi_1<-t_0-1\}}|\tilde{F}-fF^2|^2_he^{-\varphi_1-\Psi_1} \\
   \le & \frac{1}{c(t_0+1)e^{-(t_0+1)}}
   \int_{D_0}|\tilde{F}-(1-b_{t_0}(\Psi_1))fF^2|^2_he^{-\varphi_1-\Psi_1+v_{t_0}(\Psi_1)}c(-v_{t_0}(\Psi_1))\\
   \le &\frac{C}{c(t_0+1)e^{-(t_0+1)}}
   \int_{D_0}\mathbb{I}_{\{-t_0-1<\Psi_1<-t_0\}}|f|^2_he^{-\Psi_1}<+\infty.
\end{split}
\end{equation}
Note that on $\{\Psi_1<-t_0\}$, $|F|^4e^{-\varphi_1}=1$. As $v_{t_0}(\Psi_1)\ge \Psi_1$, we have $c(-v_{t_0}(\Psi_1))e^{v_{t_0}(\Psi_1)}\ge c(-\Psi_1)e^{-\Psi_1}$. Hence we have
\begin{equation}\nonumber
\begin{split}
   &\int_{D_0}|\tilde{F}|^2_he^{-\varphi_1}c(-\Psi_1) \\
   \le & 2\int_{D_0}|\tilde{F}-(1-b_{t_0}(\Psi_1))fF^2|^2_he^{-\varphi_1}c(-\Psi_1)\\
   +&2\int_{D_0}|(1-b_{t_0}(\Psi_1))fF^2|^2_he^{-\varphi_1}c(-\Psi_1)\\
   \le&
   2\int_{D_0}|\tilde{F}-(1-b_{t_0}(\Psi_1))fF^2|^2_he^{-\varphi_1-\Psi_1+v_{t_0}(\Psi_1)}c(-v_{t_0}(\Psi_1))\\
   +&2\int_{D_0\cap\{\Psi<-t_0\}}|f|^2_hc(-\Psi)\\
   < &+\infty.
\end{split}
\end{equation}
Hence we know that $(\tilde{F},o)\in \mathcal{H}_o$.
\end{proof}

For any $(\tilde{F},o)\in\mathcal{H}_o$ and $(\tilde{F}_1,o)\in\mathcal{H}_o$ such that $\int_{D_1\cap\{\Psi_1<-t_1\}}|\tilde{F}-fF^2|^2_he^{-\varphi_1-\Psi_1}<+\infty$ and
$\int_{D_1\cap\{\Psi_1<-t_1\}}|\tilde{F}_1-fF^2|^2_he^{-\varphi_1-\Psi_1}<+\infty$, for some open neighborhood $D_1$ of $o$ and $t_1> T_1$, we have
$$\int_{D_1\cap\{\Psi_1<-t_1\}}|\tilde{F}_1-\tilde{F}|^2_he^{-\varphi_1-\Psi_1}<+\infty.$$
As $(\tilde{F},o)\in\mathcal{H}_o$ and $(\tilde{F}_1,o)\in\mathcal{H}_o$, there exists a neighborhood $D_2$ of $o$ such that
\begin{equation}\label{construction of morphism formula 3}
\int_{D_2}|\tilde{F}_1-\tilde{F}|^2_he^{-\varphi_1}c(-\Psi_1)<+\infty.
\end{equation}
Note that we have $c(-\Psi_1)e^{\Psi_1}\ge c(t_1)e^{-t_1}$ on $\{\Psi\ge-t_1\}$. It follows from inequality \eqref{construction of morphism formula 3} that we have
$$\int_{D_2\cap \{\Psi\ge-t_1\}}|\tilde{F}_1-\tilde{F}|^2_he^{-\varphi_1-\Psi_1}<+\infty.$$
Hence we have $(\tilde{F}-\tilde{F}_1,o)\in \mathcal{E}(he^{-\varphi_1-\Psi_1})_o
$.

Thus it follows from Lemma \ref{construction of morphism} that there exists a map $\tilde{P}:H_o\to \mathcal{H}_o/\mathcal{E}(he^{-\varphi_1-\Psi_1})_o$ given by
$$\tilde{P}(f_o)=[(\tilde{F},o)]$$
for any $f_o\in H_o$, where $(\tilde{F},o)$ satisfies $(\tilde{F},o)\in \mathcal{H}_o$ and
$\int_{D_1\cap\{\Psi_1<-t_1\}}|\tilde{F}-fF^2|^2_he^{-\varphi_1-\Psi_1}<+\infty,$
for some $t_1>T_1$ and some open neighborhood $D_1$ of $o$, and $[(\tilde{F},o)]$ is the equivalence class of $(\tilde{F},o)$ in $\mathcal{H}_o/\mathcal{E}(he^{-\varphi_1-\Psi_1})_o$.

\begin{Proposition}\label{proposition of morphism}
$\tilde{P}$ is an $\mathcal{O}_{\mathbb{C}^n,o}$-module homomorphism and $Ker(\tilde{P})=I(h,\Psi_1)_o$.
\end{Proposition}
\begin{proof}For any $f_o,g_o\in H_o$. Denote that $\tilde{P}(f_o)=[(\tilde{F},o)]$, $\tilde{P}(g_o)=[(\tilde{G},o)]$ and $\tilde{P}(f_o+g_o)=[(\tilde{H},o)]$.

Note that there exist an open neighborhood $D_1$ of $o$ and $t> T_1$ such that $\int_{D_1\cap\{\Psi_1<-t\}}|\tilde{F}-fF^2|^2_he^{-\varphi_1-\Psi_1}<+\infty$,
$\int_{D_1\cap\{\Psi_1<-t\}}|\tilde{G}-gF^2|^2_he^{-\varphi_1-\Psi_1}<+\infty$, and
$\int_{D_1\cap\{\Psi_1<-t\}}|\tilde{H}-(f+g)F^2|^2_he^{-\varphi_1-\Psi_1}<+\infty$. Hence we have
$$\int_{D_1\cap\{\Psi_1<-t\}}|\tilde{H}-(\tilde{F}+\tilde{G})|^2_he^{-\varphi_1-\Psi_1}<+\infty.$$
As $(\tilde{F},o),(\tilde{G},o)$ and $(\tilde{H},o)$ belong to $ \mathcal{H}_o$, there exists an open neighborhood $\tilde{D}_1\subset D_1$ of $o$ such that
$\int_{\tilde{D}_1}|\tilde{H}-(\tilde{F}+\tilde{G})|^2_he^{-\varphi_1}c(-\Psi_1)<+\infty$.
As $c(t)e^{-t}$ is decreasing with respect to $t$, we have $c(-\Psi_1)e^{\Psi_1}\ge c(t)e^{-t}$ on $\{\Psi_1\ge -t\}$. Hence we have
$$\int_{\tilde{D}_1\cap\{\Psi_1\ge -t\}}|\tilde{H}-(\tilde{F}+\tilde{G})|^2_he^{-\varphi_1-\Psi_1}
\le\frac{1}{c(t)e^{-t}}\int_{\tilde{D}_1\cap\{\Psi_1\ge -t\}}|\tilde{H}-(\tilde{F}+\tilde{G})|^2_he^{-\varphi_1}c(-\Psi_1)<+\infty.$$
Thus we have $\int_{\tilde{D}_1}|\tilde{H}-(\tilde{F}+\tilde{G})|^2_he^{-\varphi_1-\Psi_1}<+\infty$, which implies that $\tilde{P}(f_o+g_o)=\tilde{P}(f_o)+\tilde{P}(g_o)$.

For any $(q,o) \in \mathcal{O}_{\mathbb{C}^n,o}$. Denote $\tilde{P}((qf)_o)=[(\tilde{F}_q,o)]$. Note that there exist an open neighborhood $D_2$ of $o$ and $t> T_1$ such that $\int_{D_2\cap\{\Psi_1<-t\}}|\tilde{F}_q-(qf)F^2|^2_he^{-\varphi_1-\Psi_1}<+\infty$. It follows from $\int_{D_2\cap\{\Psi_1<-t\}}|\tilde{F}-fF^2|^2_he^{-\varphi_1-\Psi_1}<+\infty$ and $q$ is holomorphic on $\overline{D_2}$ (shrink $D_2$ if necessary) that $\int_{D_2\cap\{\Psi_1<-t\}}|q\tilde{F}-qfF^2|^2_he^{-\varphi_1-\Psi_1}<+\infty$. Then we have
$$\int_{D_2\cap\{\Psi_1<-t\}}|\tilde{F}_q-q\tilde{F}|^2_he^{-\varphi_1-\Psi_1}<+\infty.$$
Note that $(q\tilde{F},o) $ and $(\tilde{F}_q,o)$ belong to $ \mathcal{H}_o$, we have
$\int_{D_2}|\tilde{F}_q-q\tilde{F}|^2_he^{-\varphi_1}c(-\Psi_1)<+\infty$.
As $c(t)e^{-t}$ is decreasing with respect to $t$, we have $c(-\Psi_1)e^{\Psi_1}\ge c(t)e^{-t}$ on $\{\Psi_1\ge -t\}$. Hence we have
$$\int_{D_2\cap\{\Psi_1\ge -t\}}|\tilde{F}_q-q\tilde{F}|^2_he^{-\varphi_1-\Psi_1}
\le\frac{1}{c(t)e^{-t}}\int_{D_2\cap\{\Psi_1\ge -t\}}|\tilde{F}_q-q\tilde{F}|^2_he^{-\varphi_1}c(-\Psi_1)<+\infty.$$
Thus we have $\int_{D_2}|\tilde{F}_q-q\tilde{F}|^2_he^{-\varphi_1-\Psi_1}<+\infty$, which implies that $\tilde{P}(qf_o)=(q,o)\tilde{P}(f_o)$.
We have proved that $\tilde{P}$ is an $\mathcal{O}_{\mathbb{C}^n,o}$-module homomorphism.

Next, we prove $Ker(\tilde{P})=I(h,\Psi_1)_o$.

If $f_o\in I(h,\Psi_1)_o$. Denote $\tilde{P}(f_o)=[(\tilde{F},o)]$. It follows from Lemma \ref{construction of morphism} that $(\tilde{F},o)\in \mathcal{H}_o$ and there exist an open neighborhood $D_3$ of $o$ and a real number $t_1>T_1$ such that
$$\int_{\{\Psi_1<-t_1\}\cap D_3}|\tilde{F}-fF^2|^2_he^{-\varphi_1-\Psi_1}<+\infty.$$
As $f_o\in I(h,\Psi_1)_o$, shrink $D_3$ and $t_1$ if necessary, we have
\begin{equation}\label{proposition of morphism formula 1}
\begin{split}
&\int_{\{\Psi_1<-t_1\}\cap D_3}|\tilde{F}|^2_he^{-\varphi_1-\Psi_1}\\
\le &2\int_{\{\Psi_1<-t_1\}\cap D_3}|\tilde{F}-fF^2|^2_he^{-\varphi_1-\Psi_1}
+2\int_{\{\Psi_1<-t_1\}\cap D_3}|fF^2|^2_he^{-\varphi_1-\Psi_1}\\
\le &2\int_{\{\Psi_1<-t_1\}\cap D_3}|\tilde{F}-fF^2|^2_he^{-\varphi_1-\Psi_1}
+2\int_{\{\Psi_1<-t_1\}\cap D_3}|f|^2_he^{-\Psi_1}\\
<&+\infty.
\end{split}
\end{equation}
As $c(t)e^{-t}$ is decreasing with respect to $t$, $c(-\Psi_1)e^{\Psi_1}\ge C_0>0$ for some positive number $C_0$ on $\{\Psi_1\ge-t_1\}$. Then we have
\begin{equation}\label{proposition of morphism formula 1'}
\begin{split}
\int_{\{\Psi_1\ge-t_1\}\cap D_3}|\tilde{F}|^2_he^{-\varphi_1-\Psi_1}
\le\frac{1}{C_0}\int_{\{\Psi_1\ge-t_1\}\cap D_3}|\tilde{F}|^2_he^{-\varphi_1}c(-\Psi_1)<+\infty.
\end{split}
\end{equation}
Combining inequality \eqref{proposition of morphism formula 1} and inequality  \eqref{proposition of morphism formula 1'}, we know that $\tilde{F}\in \mathcal{E}(he^{-\varphi_1-\Psi_1})_o$, which means $\tilde{P}(f_o)=0$ in $\mathcal{H}_o/\mathcal{E}(he^{-\varphi_1-\Psi_1})_o$. Hence we know $I(h,\Psi_1)_o\subset Ker(\tilde{P})$.

If $f_o\in Ker(\tilde{P})$, we know $\tilde{F}\in \mathcal{E}(he^{-\varphi_1-\Psi_1})_o$.
We can assume that $\tilde{F}$ satisfies $\int_{D_4}|\tilde{F}|^2_he^{-\varphi_1-\Psi_1}<+\infty$ for some open neighborhood $D_4$ of $o$. Then we have
\begin{equation}\label{proposition of morphism formula 2'}
\begin{split}
&\int_{ \{\Psi_1<-t_1\}\cap D_4}|f|^2_he^{-\Psi_1}\\
=&\int_{\{\Psi_1<-t_1\}\cap D_4}|fF^2|^2_he^{-\varphi_1-\Psi_1}\\
\le & \int_{\{\Psi_1<-t_1\}\cap D_4}|\tilde{F}|^2_he^{-\varphi_1-\Psi_1}+\int_{\{\Psi_1<-t_1\}\cap D_4}|\tilde{F}-fF^2|^2_he^{-\varphi_1-\Psi_1}\\
< &+\infty.
\end{split}
\end{equation}
By definition, we know $f_o\in I(h,\Psi_1)_o$. Hence $ Ker(\tilde{P})\subset I(h,\Psi_1)_o$.

$ Ker(\tilde{P})= I(h,\Psi_1)_o$ is proved.
\end{proof}

Now we can define an $\mathcal{O}_{\mathbb{C}^n,o}$-module homomorphism $P:H_o/I(h,\Psi_1)_o\to \mathcal{H}_o/\mathcal{E}(he^{-\varphi_1-\Psi_1})_o$ as follows,
$$P([f_o])=\tilde{P}(f_o)$$
for any $[f_o]\in H_o/I(h,\Psi_1)_o$, where $f_o\in H_o$ is any representative of $[f_o]$. It follows from Proposition \ref{proposition of morphism} that $P([f_o])$ is independent of the choices of the representatives of $[f_o]$.

Let $(\tilde{F},o)\in \mathcal{H}_o$, i.e. $\int_{U}|\tilde{F}|^2_he^{-\varphi_1}c(-\Psi_1)<+\infty$ for some neighborhood $U$ of $o$. Note that $|F|^4e^{-\varphi_1}\equiv 1$ on $\{\Psi_1<-T\}$. Hence we have $\int_{U\cap \{\Psi_1<-t\}}|\frac{\tilde{F}}{F^2}|^2_hc(-\Psi_1)<+\infty$ for some $t>T$, i.e. $(\frac{\tilde{F}}{F^2})_o\in H_o$. And if $(\tilde{F},o)\in \mathcal{E}(he^{-\varphi_1-\Psi_1})_o$, it is easy to verify that $(\frac{\tilde{F}}{F^2})_o\in I(h,\Psi_1)_o$. Hence we have an $\mathcal{O}_{\mathbb{C}^n,o}$-module homomorphism $Q:\mathcal{H}_o/\mathcal{E}(he^{-\varphi_1-\Psi_1})_o\to H_o/I(h,\Psi_1)_o$ defined as follows,
$$Q([(\tilde{F},o)])=[(\frac{\tilde{F}}{F^2})_o].$$

The above discussion shows that $Q$ is independent of the choices of the representatives of $[(\tilde{F},o)]$ and hence $Q$ is well defined.

\begin{Proposition}\label{module isomorphism}$P:H_o/I(h,\Psi_1)_o\to \mathcal{H}_o/\mathcal{E}(he^{-\varphi_1-\Psi_1})_o$ is an $\mathcal{O}_{\mathbb{C}^n,o}$-module isomorphism and $P^{-1}=Q$.
\end{Proposition}
\begin{proof} It follows from Proposition \ref{proposition of morphism} that we know $P$ is injective.

Now we prove $P$ is surjective.

For any $[(\tilde{F},o)]$ in $\mathcal{H}_o/\mathcal{E}(he^{-\varphi_1-\Psi_1})_o$. Let $(\tilde{F},o)$ be any representatives of $[(\tilde{F},o)]$ in $\mathcal{H}_o$. Denote that $[(f_1)_o]:=[(\frac{\tilde{F}}{F^2})_o]=Q([(\tilde{F},o)])$. Let $(f_1)_o:=(\frac{\tilde{F}}{F^2})_o\in H_o$ be the representative of $[(f_1)_o]$. Denote $[(\tilde{F}_1,o)]:=\tilde{P}((f_1)_o)=P([(f_1)_o])$. By the construction of $\tilde{P}$, we know that $(\tilde{F}_1,o)\in \mathcal{H}_o$ and
$$\int_{D_1\cap\{\Psi_1<-t\}}|\tilde{F}_1-f_1F^2|^2_he^{-\varphi_1-\Psi_1}<+\infty,$$
where $t>T$ and $D_1$ is some neighborhood of $o$. Note that $(f_1)_o:=(\frac{\tilde{F}}{F^2})_o$. Hence  we have
$$\int_{D_1\cap\{\Psi_1<-t\}}|\tilde{F}_1-\tilde{F}|^2_he^{-\varphi_1-\Psi_1}<+\infty.$$
It follows from $(\tilde{F},o)\in \mathcal{H}_o$ and $(\tilde{F}_1,o)\in \mathcal{H}_o$ that there exists a neighborhood $D_2\subset D_1$ of $o$ such that
$$\int_{D_2}|\tilde{F}-\tilde{F}_1|^2_he^{-\varphi_1}c(-\Psi_1)<+\infty.$$
Note that on $\{\Psi_1\ge -t\}$, we have $c(-\Psi_1)e^{\Psi_1}\ge c(t)e^{-t}>0$. Hence we have
$$\int_{D_2\cap \{\Psi_1\ge-t\}}|\tilde{F}-\tilde{F}_1|^2_he^{-\varphi_1-\Psi_1}<+\infty.$$
Thus we know that $(\tilde{F}_1-\tilde{F},o) \in \mathcal{E}(he^{-\varphi_1-\Psi_1})_o$, i.e. $[(\tilde{F},o)]=[(\tilde{F}_1,o)]$ in $ \mathcal{H}_o/\mathcal{E}(he^{-\varphi_1-\Psi_1})_o$. Hence we have $P\circ Q([(\tilde{F},o)])=[(\tilde{F},o)]$, which implies that $P$ is surjective.

We have proved that $P:H_o/I(h,\Psi_1)_o\to \mathcal{H}_o/\mathcal{E}(he^{-\varphi_1-\Psi_1})_o$ is an $\mathcal{O}_{\mathbb{C}^n,o}$-module isomorphism and $P^{-1}=Q$.
\end{proof}

The following lemma shows the closedness of submodules of $H_o$.

Recall that $D$ is a pseudoconvex domain in $\mathbb{C}^n$ containing the origin $o\in \mathbb{C}^n$, $F$ is a holomorphic function on $D$ and
$f=(f_1,f_2,\ldots,f_r)$ be a holomorphic section of $E:=D\times \mathbb{C}^r$. Let $\psi$ be a plurisubharmonic function on $D$. Let $h$ be a \textit{measurable metric} on $D\times \mathbb{C}^r$ and $\tilde{h}:=he^{-\psi}$. Let $(D,D\times \mathbb{C}^r,\Sigma,D_j,\tilde{h},\tilde{h}_{j,s})$ be a singular metric on $E:=D\times \mathbb{C}^r$ which satisfies $\Theta_{\tilde{h}}(E)\ge^s_{Nak} 0$. Let $c(t)\in \tilde{P}_{T,D,\Psi,h}$.
\begin{Lemma}\label{closedness of module}
Let $U_0\Subset D$ be a Stein neighborhood of $o$.
Let $J_o$ be an $\mathcal{O}_{\mathbb{C}^n,o}$-submodule of $H_o$ such that $I(h,\Psi)_o\subset J_o$. Assume that $f_o\in J(\Psi)_o$. Let $\{f_j\}_{j\ge 1}$ be a sequence of $E$-valued holomorphic $(n,0)$ forms on $U_0\cap \{\Psi<-t_j\}$ for any $j\ge 1$, where $t_j>T$. Assume that $t_0:=\lim_{j\to +\infty}t_j\in[T,+\infty)$,
\begin{equation}\label{convergence property of module}
\limsup\limits_{j\to+\infty}\int_{U_0\cap\{\Psi<-t_j\}}|f_j|^2_hc(-\Psi)\le C<+\infty,
\end{equation}
and $(f_j-f)_o\in J_o$. Then there exists a subsequence of $\{f_j\}_{j\ge 1}$ compactly convergent to an $E$-valued  holomorphic $(n,0)$ form $f_0$ on $\{\Psi<-t_0\}\cap U_0$ which satisfies
$$\int_{U_0\cap\{\Psi<-t_0\}}|f_0|^2_hc(-\Psi)\le C,$$
and $(f_0-f)_o\in J_o$.
\end{Lemma}
\begin{proof}

It follows from $c(t)\in\tilde{P}_{T,D,\Psi,h}$ that there exists an analytic subset $Z$ of $D$ and for any compact subset $K\subset D\backslash Z$, $|e_x|_h^2c(-\psi)\ge C_K|e_x|^2_{\hat{h}}$ for any $ x\in K\cap\{\Psi<-t_0\}$, where $C_K>0$ is a constant and $e_x\in E_x$.

It follows from inequality \eqref{convergence property of module}, Lemma \ref{l:converge} and diagonal method that there exists a subsequence of $\{f_j\}_{j\ge 1}$ (also denoted by $\{f_j\}_{j\ge 1}$) compactly convergent to an $E$-valued holomorphic $(n,0)$ form $f_0$ on $\{\Psi<-t_0\}\cap U_0$. It follows from Fatou's Lemma that
$$\int_{U_0\cap\{\Psi<-t_0\}}|f_0|^2_hc(-\Psi)\le \liminf\limits_{j\to+\infty}\int_{U_0\cap\{\Psi<-t_j\}}|f_j|^2_hc(-\Psi)\le C.$$

Now we prove $(f_0-f)_o\in J_o$. We firstly recall some constructions in Lemma \ref{construction of morphism}.

As $t_0:=\lim_{j\to +\infty}t_j\in[T,+\infty)$. We can assume that $\{t_j\}_{j\ge 0}$ is upper bounded by some real number $T_1+1$. Denote $\Psi_1:=\min\{\psi-2\log|F|,-T_1\}$, and
if $F(z)=0$ for some $z \in M$, we set $\Psi_1(z)=-T_1$. We note that \begin{equation}\nonumber
\limsup\limits_{j\to+\infty}\int_{U_0\cap\{\Psi<-T_1-1\}}|f_j|^2_hc(-\Psi)\le C<+\infty.
\end{equation}

It follows from $c(t)\in\tilde{P}_{T,D,\Psi,h}$ and Lemma \ref{L2 method for c(t)} that there exists an $E$-valued holomorphic $(n,0)$ form $\tilde{F}_j$ on $U_0$ such that
\begin{equation}\label{convergence property of module formula 1}
  \begin{split}
      & \int_{U_0}|\tilde{F}_j-(1-b_{1}(\Psi_1))f_jF^{2}|^2_he^{-\varphi_1+v_{1}(\Psi_1)-\Psi_1}c(-v_{1}(\Psi_1)) \\
      \le & \left(c(T_1)e^{-T_1}+\int_{T_1}^{T_1+2}c(s)e^{-s}ds\right)
       \int_{U_0}\mathbb{I}_{\{-T_1-2<\Psi_1<-T_1-1\}}|f_j|^2_he^{-\Psi_1},
  \end{split}
\end{equation}
where $b_{1}(t)=\int^{t}_{-\infty} \mathbb{I}_{\{-T_1-2< s < -T_1-1\}}ds$,
$v_{1}(t)=\int^{t}_{-T_1-1}b_{1}(s)ds-(T_1+1)$. Denote $C_1:=c(T_1)e^{-T_1}+\int_{T_1}^{T_1+1}c(s)e^{-s}ds$.

Note that $v_{1}(t)>-T_1-2$. We have $e^{v_{1}(\Psi_1)}c(-v_{1}(\Psi))\ge c(T_1+2)e^{-(T_1+2)}>0$. As $b_{1}(t)\equiv 0$ on $(-\infty,-T_1-2)$, we have
\begin{equation}\label{convergence property of module formula 2}
\begin{split}
   &\int_{U_0\cap\{\Psi<-T_1-2\}}|\tilde{F}_j-f_jF^2|^2_he^{-\varphi_1-\Psi_1} \\
   \le & \frac{1}{c(T_1+2)e^{-(T_1+2)}}
   \int_{U_0}|\tilde{F}_j-(1-b_{1}(\Psi_1))f_jF^2|^2_he^{-\varphi_1-\Psi_1+v_{1}(\Psi_1)}c(-v_{t_j}(\Psi_1))\\
   \le &\frac{C_1}{c(T_1+2)e^{-(T_1+2)}}
   \int_{U_0}\mathbb{I}_{\{-T_1-2<\Psi_1<-T_1-1\}}|f_j|^2_he^{-\Psi_1}<+\infty.
\end{split}
\end{equation}
Note that $|F^2|^2e^{-\varphi_1}=1$ on $\{\Psi_1<-T_1-1\}$. As $v_{t_j}(\Psi_1)\ge \Psi_1$, we have $c(-v_{t_j}(\Psi_1))e^{v_{t_j}(\Psi_1)}\ge c(-\Psi_1)e^{-\Psi_1}$. Hence we have
\begin{equation}\label{convergence property of module formula 3}
\begin{split}
   &\int_{U_0}|\tilde{F}_j|^2_he^{-\varphi_1}c(-\Psi_1) \\
   \le & 2\int_{U_0}|\tilde{F}_j-(1-b_{1}(\Psi_1))f_jF^2|^2_he^{-\varphi_1}c(-\Psi_1)\\
   +&2\int_{U_0}|(1-b_{1}(\Psi_1))f_jF^2|^2_he^{-\varphi_1}c(-\Psi_1)\\
   \le&
   2\int_{U_0}|\tilde{F}_j-(1-b_{1}(\Psi_1))f_jF^2|^2_he^{-\varphi_1-\Psi_1+v_{1}(\Psi_1)}c(-v_{1}(\Psi_1))\\
   +&2\int_{U_0\cap\{\Psi_1<-T_1-1\}}|f_j|^2_hc(-\Psi_1)\\
   < &+\infty.
\end{split}
\end{equation}
Hence we know that $(\tilde{F}_j,o)\in \mathcal{H}_o$.

It follows from inequality \eqref{convergence property of module}, $\sup_{j\ge1}\left(\int_{U_0}\mathbb{I}_{\{-T_1-2<\Psi<-T_1-1\}}|f_j|^2_he^{-\Psi}\right)<+\infty$ and inequality \eqref{convergence property of module formula 3} that we actually have \begin{equation}\label{closedness of module formula 1}
\sup_j\left(\int_{U_0}|\tilde{F}_j|^2_he^{-\varphi_1}c(-\Psi_1)\right)<+\infty.
\end{equation}
Note that
$c(t)e^{-t}$ is decreasing with respect to $t$ and there exists an analytic subset $S$ of $D$ and for any compact subset $K\subset D\backslash S$,  $|e_x|_h^2c(-\psi)\ge C_K|e_x|^2_{\hat{h}}$ for any $ x\in K\cap\{\Psi<-t_0\}$, where $C_K>0$ is a constant and $e_x\in E_x$.

 Let $K\subset U_0\backslash S\subset D\backslash S$ be any compact set, then for any $f$ being an $E$-valued holomorphic $(n,0)$ form, we have
 $$|f_x|^2_h e^{-\varphi_1}c(-\Psi_1)\ge \tilde{C}_K|f_x|^2_{\hat{h}}$$ for any $x\in {K\cap\{\Psi_1<-T_1\}}$ and
  $$|f_x|^2_h e^{-\varphi_1}c(-\Psi_1)\ge C_1|f_x|^2_he^{-\varphi_1-\Psi_1}=|f_x|^2_{\tilde{h}}e^{-\delta \tilde{M}}\ge C_1C_2|f_x|^2_{\hat{h}}$$ for any $x\in {K\cap\{\Psi_1\ge-T_1\}}$, where $C_K,C_1,C_2>0$ are constants. Hence

\begin{equation}\label{closedness of module formula 2}
  \begin{split}
    |f_x|^2_h e^{-\varphi_1}c(-\Psi_1)
     \ge \min\{\tilde{C}_K,C_1C_2\}|f_x|^2_{\hat{h}},
  \end{split}
\end{equation}
for any $x\in {K\cap\{\Psi_1\ge-T_1\}}$.
It follows from inequality \eqref{closedness of module formula 1}, inequality \eqref{closedness of module formula 2} and Lemma \ref{l:converge} that there exists a subsequence of $\{\tilde{F}_j\}_{j\ge 1}$ (also denoted by $\{\tilde{F}_j\}_{j\ge 1}$) compactly convergent to an $E$-valued holomorphic $(n,0)$ form $\tilde{F}_0$ on $U_0$ and
\begin{equation}\label{convergence property of module formula 4}
\int_{U_0}|\tilde{F}_0|^2_he^{-\varphi_1}c(-\Psi_1)\le
\liminf_{j\to +\infty}\int_{U_0}|\tilde{F}_j|^2_he^{-\varphi_1}c(-\Psi_1)<+\infty.
\end{equation}
As $f_j$ converges to $f_0$, it follows from Fatou's Lemma and inequality \eqref{convergence property of module formula 1} that
\begin{equation}\nonumber
  \begin{split}
  &\int_{U_0}|\tilde{F}_0-(1-b_{1}(\Psi))f_0F^{2}|^2_he^{-\varphi_1+v_{1}(\Psi_1)-\Psi_1}c(-v_{1}(\Psi_1)) \\
     \le & \liminf_{j\to+\infty} \int_{U_0}|\tilde{F}_j-(1-b_{1}(\Psi))f_jF^{2}|^2_he^{-\varphi_1+v_{1}(\Psi)-\Psi}c(-v_{1}(\Psi_1)) \\
     < &+\infty,
  \end{split}
\end{equation}
which implies that
\begin{equation}\label{convergence property of module formula 5}
  \begin{split}
\int_{U_0\cap\{\Psi<-T_1-2\}}|\tilde{F}_0-f_0F^2|^2_he^{-\varphi_1-\Psi_1}<+\infty.
  \end{split}
\end{equation}
It follows from inequality \eqref{convergence property of module formula 2}, inequality \eqref{convergence property of module formula 3}, inequality \eqref{convergence property of module formula 4}, inequality \eqref{convergence property of module formula 5} and definition of $P:H_o/I(h,\Psi_1)_o\to \mathcal{H}_o/\mathcal{E}(he^{-\varphi_1-\Psi_1})_o$ that for any $j\ge 0$, we have
$$P([(f_j)_o])=[(\tilde{F}_j,o)].$$

Note that $I(h,\Psi_1)_o=I(h,\Psi)_o\subset J_o$. As $(f_j-f)_o\in J_o$ for any $j\ge 1$, we have $(f_j-f_1)_o\in J_o$ for any $j\ge 1$.
It follows from Proposition \ref{module isomorphism} that there exists a submodule $\tilde{J}$ of $\mathcal{O}^r_{\mathbb{C}^n,o}$ such that $\mathcal{E}(he^{-\varphi_1-\Psi_1})_o\subset \tilde{J}\subset \mathcal{H}_o$ and $\tilde{J}/\mathcal{E}(he^{-\varphi_1-\Psi_1})_o=\text{Im}(P|_{J_o/I(h,\Psi_1)_o})$. It follows from $(f_j-f_1)_o\in J_o$ and $P([(f_j)_o])=[(F_j,o)]$ for any $j\ge 1$ that we have
$$(\tilde{F}_j-\tilde{F}_1)\in \tilde{J},$$
for any $j\ge 1$.

As $\tilde{F}_j$ compactly converges to $\tilde{F}_0$, using Lemma \ref{closedness}, we obtain that $(\tilde{F}_0-\tilde{F}_1,o)\in\tilde{J}$. Note that $P$ is an  $\mathcal{O}_{\mathbb{C}^n,o}$-module isomorphism and $\tilde{J}/\mathcal{E}(he^{-\varphi_1-\Psi_1})_o=\text{Im}(P|_{J_o/I(h,\Psi_1)_o})$. We have $(f_0-f_1)_o\in J_o$, which implies that $(f_0-f)_o\in J_o$.

Lemma \ref{closedness of module} is proved.
\end{proof}

 Let $c\equiv1$,  and note that $H_o=I(h, 0\Psi)_o$ and   $\mathcal{H}_o=\mathcal{E}(he
^{-\varphi_1})_o$.
It is clear that $I(h,a\Psi)_o\subset I(h,a'\Psi)_o$ for any $0\le a'<a<+\infty$. Denote that $I_+(h,a\Psi)_o:=\cup_{p>a}I(h,p\Psi)_o$ is an $\mathcal{O}_{\mathbb{C}^n,o}$-submodule of $H_o$, where $a\ge0$.
\begin{Lemma}
	\label{l:m5} There exists $a'>a$ such that $I(h,a'\Psi)_o=I_+(h,a\Psi)_o$ for any $a\ge0$.
\end{Lemma}
\begin{proof}
The definition of $I_+(h,a\Psi)_o$ shows $I(h,p\Psi)_o\subset I_+(h,a\Psi)_o$ for any $p>a$. It suffices to prove that  there exists $a'>a$ such that $I_+(h,a\Psi)_o\subset I(h,a'\Psi)_o$.

 Denote that $\tilde\varphi_1:=k\varphi_1=2\max\{k\psi+kT,2\log|F^k|\}$ and $\tilde \Psi:=k\Psi=\min\{k\psi-2\log|F^k|,-kT\}$, where  $k>a$ is an integer. As $he^{-\psi}\ge_{Nak}^s0$ and $\psi$ is plurisubharmonic on $M$, it follows from Remark \ref{example of singular metric} that $he^{-k\psi}\ge_{Nak}^s0$. Proposition \ref{module isomorphism} shows that there exists an $\mathcal{O}_{\mathbb{C}^n,o}$-module isomorphism $P$ from $I(h,0\Psi)_o/I(h,\tilde \Psi)_o\rightarrow \mathcal{E}(he^{-\varphi_1})_o/\mathcal{E}(he^{-\tilde\varphi_1-\tilde \Psi})_o$, which implies that for any $p\in(0,k)$, there exists an $\mathcal{O}_{\mathbb{C}^n,o}$-submodule $K_p$ of $\mathcal{O}^r_{\mathbb{C}^n,o}$ such that
 $$P(I(h,p\Psi)_o/I(h,\tilde \Psi)_o)=K_p/\mathcal{E}(he^{-\tilde\varphi_1-\tilde \Psi})_o.$$
	Denote that
	$$L:=\cup_{a<p<k}K_p$$
	 be an $\mathcal{O}_{\mathbb{C}^n,o}$-submodule $K_p$ of $\mathcal{O}^r_{\mathbb{C}^n,o}$. Hence $P|_{I_+(h,a\Psi)_o/I(h,\tilde \Psi)_o}$ is an  $\mathcal{O}_{\mathbb{C}^n,o}$-module isomorphism from $I_+(h,a\Psi)_o/I(h,\tilde \Psi)_o$ to $L/\mathcal{E}(he^{-\tilde\varphi_1-\tilde \Psi})_o$. As $\mathcal{O}_{\mathbb{C}^n,o}$ is a Noetherian ring (see \cite{hormander}), we know that $\mathcal{O}^r_{\mathbb{C}^n,o}$ is a Noetherian $\mathcal{O}_{\mathbb{C}^n,o}$-module, which implies that $L$ is finitely generated. Thus,  we have a finite set $\{(f_1)_o,\ldots,(f_m)_o\}\subset I_+(h,a\Psi)_o$, which satisfies that for any $f_o\in I_+(h,a\Psi)_o$, there exists  $(h_j,o)\in\mathcal{O}_{\mathbb{C}^n,o}$ for $1\le j\le m$ such that
	 $$f_o-\sum_{j=1}^{m}(h_j,o)\cdot (f_j)_o\in I(h,\tilde \Psi)_o.$$
	  By the definition of $I_+(h,a\Psi)_o$, there exists $a'\in(a,k)$ such that $\{(f_1)_o,\ldots,(f_m)_o\}\subset I(h,a'\Psi)_o$. Note that  $I(h,\tilde\Psi)_o=I(h,k\Psi)_o\subset I(h,a'\Psi)_o$. Then we obtain that $I_+(h,a\Psi)_o\subset I(h,a'\Psi)_o$.
	
	  Thus, Lemma \ref{l:m5} holds.
\end{proof}

\section{Properties of $G(t)$}
Following the notations in Section \ref{sec:Main result}, we present some properties of the function $G(t)$ in this section.

For any $t\ge T$, denote
\begin{equation}\nonumber
\begin{split}
\mathcal{H}^2(t;c,f,H):=\Bigg\{\tilde{f}:\int_{ \{ \Psi<-t\}}|\tilde{f}|^2_hc(-\Psi)<+\infty,\  \tilde{f}\in
H^0(\{\Psi<-t\},\mathcal{O} (K_M \otimes E) ) \\
\& (\tilde{f}-f)_{z_0}\in
\mathcal{O} (K_M)_{z_0} \otimes (J_{z_0}\cap H_{z_0}),\ \text{for any }  z_0\in Z_0  \Bigg\},
\end{split}
\end{equation}
where $f$ is an $E$-valued holomorphic $(n,0)$ form on $\{\Psi<-t_0\}\cap V$ for some $V\supset Z_0$ is an open subset of $M$ and some $t_0\ge T$, $c(t)$ is a positive measurable function on $(T,+\infty)$ and $H_{z_0}=\{f_o\in J(\Psi)_o:\int_{\{\Psi<-t\}\cap V_0}|f|^2e^{-\varphi}c(-\Psi)<+\infty \text{ for some }t>T_0 \text{ and } V_0 \text{ is an open neighborhood of}\  z_0\}$ (the definition of $H_{z_0}$ can be referred to Section \ref{sec:properties of module}).

If $G(t_1;c,\Psi,\varphi,J,f)<+\infty$, then there exists an $E$-valued holomorphic $(n,0)$ form $\tilde{f}_0$ on $\{\Psi<-t_1\}$ such that $(\tilde{f}_0-f)_{z_0}\in
\mathcal{O} (K_M)_{z_0} \otimes J_{z_0},\text{for any }  z_0\in Z_0$ and $$\int_{ \{ \Psi<-t_1\}}|\tilde{f}_0|^2_hc(-\Psi)<+\infty.$$

\begin{Lemma}
\label{module in def of G(t)}If $G(t_1;c,\Psi,\varphi,J,f)<+\infty$ for some $t_1\ge T$, we have $\mathcal{H}^2(t;c,f)=\mathcal{H}^2(t;c,\tilde{f}_0)=\mathcal{H}^2(t;c,\tilde{f}_0,H)$ for any $t\ge T$.
\end{Lemma}
\begin{proof}

 As $(\tilde{f}_0-f)_{z_0}\in
\mathcal{O} (K_M)_{z_0} \otimes J_{z_0},\text{for any }  z_0\in Z_0$, we have $\mathcal{H}^2(t;c,f)=\mathcal{H}^2(t;c,\tilde{f}_0)$ for any $t\ge T$.

 Now we prove $\mathcal{H}^2(t;c,\tilde{f}_0)=\mathcal{H}^2(t;c,\tilde{f}_0,H)$ for any $t\ge T$. It is obviously that $\mathcal{H}^2(t;c,\tilde{f}_0)\supset\mathcal{H}^2(t;c,\tilde{f}_0,H)$. We only need to show $\mathcal{H}^2(t;c,\tilde{f}_0)\subset\mathcal{H}^2(t;c,\tilde{f}_0,H)$.

Let $\tilde{f}_1\in \mathcal{H}^2(t_2;c,\tilde{f}_0)$ for some $t_2\ge T$. As $\int_{ \{ \Psi<-t_2\}}|\tilde{f}_1|^2_hc(-\Psi)<+\infty$, denote $t=\max\{t_1,t_2\}$, we know that
$$\int_{ \{ \Psi<-t\}}|\tilde{f}_1-\tilde{f}_0|^2_hc(-\Psi)<+\infty,$$
which implies that $(\tilde{f}_1-\tilde{f}_0)_{z_0}\in
\mathcal{O} (K_M)_{z_0} \otimes  H_{z_0},\ \text{for any }  z_0\in Z_0$. Hence $(\tilde{f}_1-\tilde{f}_0)_{z_0}\in
\mathcal{O} (K_M)_{z_0} \otimes  (J_{z_0}\cap H_{z_0}),\ \text{for any }  z_0\in Z_0$, which implies that $\tilde{f}_1\in \mathcal{H}^2(t;c,\tilde{f}_0, H)$. Hence $\mathcal{H}^2(t;c,\tilde{f}_0)=\mathcal{H}^2(t;c,\tilde{f}_0,H)$.
\end{proof}
\begin{Remark}
\label{module equivalence in def of G}
If $G(t_1;c,\Psi,\varphi,J,f)<+\infty$ for some $t_1\ge T$, we can always assume that $J_{z_0}$ is an $\mathcal{O}_{M,z_0}$-submodule of $H_{z_0}$ such that $I\big(h,\Psi\big)_{z_0}\subset J_{z_0}$, for any $z_0\in Z_0$ in the definition of $G(t;c,\Psi,h,J,f)$, where $t\in[T,+\infty)$.
\end{Remark}
\begin{proof} If $G(t_1;c,\Psi,\varphi,J,f)<+\infty$ for some $t_1\ge T$, it follows from Lemma \ref{module in def of G(t)} that $\mathcal{H}^2(t;c,f)=\mathcal{H}^2(t;c,\tilde{f}_0)=\mathcal{H}^2(t;c,\tilde{f}_0,H)$ for any $t\ge T$. By definition, we have $G(t;c,\Psi,h,J,f)=G(t;c,\Psi,h,J,\tilde{f}_0)=G(t;c,\Psi,h,J\cap H,\tilde{f}_0)$.

Hence we can always assume that $J_{z_0}$ is an $\mathcal{O}_{M,z_0}$-submodule of $H_{z_0}$ such that $I\big(h,\Psi\big)_{z_0}\subset J_{z_0}$, for any $z_0\in Z_0$.
\end{proof}

In the following discussion, we assume that $J_{z_0}$ is an $\mathcal{O}_{M,z_0}$-submodule of $H_{z_0}$ such that $I\big(h,\Psi\big)_{z_0}\subset J_{z_0}$, for any $z_0\in Z_0$.

Let $c(t)\in \tilde{P}_{T,M,\Psi,h}$.
The following lemma will be used to discuss the convergence property of $E$-valued holomorphic forms on $\{\Psi<-t\}$.

\begin{Lemma}\label{global convergence property of module}
 Let $f$ be an $E$-valued holomorphic $(n,0)$ form on $\{\Psi<-\hat{t}_0\}\cap V$, where $V\supset Z_0$ is an open subset of $M$ and $\hat{t}_0>T$
is a real number. For any $z_0\in Z_0$, let $J_{z_0}$ be an $\mathcal{O}_{M,z_0}$-submodule of $H_{z_0}$ such that $I\big(h,\Psi\big)_{z_0}\subset J_{z_0}$.

Let $\{f_j\}_{j\ge 1}$ be a sequence of $E$-valued holomorphic $(n,0)$ forms on $\{\Psi<-t_j\}$. Assume that $t_0:=\lim_{j\to +\infty}t_j\in[T,+\infty)$,
\begin{equation}\label{global convergence property of module 1}
\limsup\limits_{j\to+\infty}\int_{\{\Psi<-t_j\}}|f_j|^2_hc(-\Psi)\le C<+\infty,
\end{equation}
and $(f_j-f)_{z_0}\in \mathcal{O} (K_M)_{z_0}\otimes J_{z_0}$ for any $z_0\in Z_0$. Then there exists a subsequence of $\{f_j\}_{j\in \mathbb{N}^+}$ compactly convergent to an $E$-valued  holomorphic $(n,0)$ form $f_0$ on $\{\Psi<-t_0\}$ which satisfies
$$\int_{\{\Psi<-t_0\}}|f_0|^2_hc(-\Psi)\le C,$$
and $(f_0-f)_{z_0}\in \mathcal{O} (K_M)_{z_0}\otimes  J_{z_0}$ for any $z_0\in Z_0$.
\end{Lemma}
\begin{proof}

It follows from $c(t)\in\tilde{P}_{T,M,\Psi,h}$ that there exists an analytic subset $Z$ of $D$ and for any compact subset $K\subset D\backslash Z$, $|e|_h^2c(-\psi)\ge C_K|e|^2_{\hat{h}}$ on $K\cap\{\Psi<-t_0\}$, where $C_K>0$ is a constant and $e$ is any $E$-valued holomorphic $(n,0)$ form.
It follows from inequality \eqref{global convergence property of module 1}, Lemma \ref{l:converge} and diagonal method that there exists a subsequence of $\{f_j\}_{j\ge 1}$ (also denoted by $\{f_j\}_{j\ge 1}$) compactly convergent to an $E$-valued holomorphic $(n,0)$ form $f_0$ on $\{\Psi<-t_0\}$. It follows from Fatou's Lemma that
$$\int_{\{\Psi<-t_0\}}|f_0|^2_hc(-\Psi)\le \liminf\limits_{j\to+\infty}\int_{\{\Psi<-t_j\}}|f_j|^2_hc(-\Psi)\le C.$$

Next we prove $(f_0-f)_{z_0}\in \mathcal{O} (K_M)_{z_0}\otimes  J_{z_0}$ for any $z_0\in Z_0$.

Let $z_0\in Z_0$ be a point. As $\limsup\limits_{j\to+\infty}\int_{\{\Psi<-t_j\}}|f_j|^2_hc(-\Psi)\le C<+\infty$,
 there exists an open Stein neighborhood $U_{z_0}\Subset M$ of $z_0$ such that
$$\limsup\limits_{j\to+\infty}\int_{U_{z_0}\cap\{\Psi<-t_j\}}|f_j|^2_hc(-\Psi)\le C<+\infty.$$
Note that we also have $(f_j-f)_{z_0}\in J_{z_0}$.
 It follows from Lemma \ref{closedness of module} and the uniqueness of limit function that  $(f_0-f)_{z_0}\in
\mathcal{O} (K_M)_{z_0} \otimes J_{z_0}$ for any $z_0\in Z_0$.

Lemma \ref{global convergence property of module} is proved.
\end{proof}

\begin{Lemma}
\label{characterization of g(t)=0} Let $t_0>T$.
The following two statements are equivalent,\\
(1) $G(t_0)=0$;\\
(2) $f_{z_0}\in
\mathcal{O} (K_M)_{z_0} \otimes J_{z_0}$, for any  $ z_0\in Z_0$.
\end{Lemma}
\begin{proof}If $f_{z_0}\in
\mathcal{O} (K_M)_{z_0} \otimes J_{z_0}$, for any  $ z_0\in Z_0$, then take $\tilde{f}\equiv 0$ in the definition of $G(t)$ and we get $G(t_0)\equiv 0$.

If $G(t_0)=0$, by definition, there exists a sequence of $E$-valued holomorphic $(n,0)$ forms $\{f_j\}_{j\in\mathbb{Z}^+}$ on $\{\Psi<-t_0\}$ such that
\begin{equation}\label{estimate in G(t)=0}
\lim_{j\to+\infty}\int_{\{\Psi<-t_0\}}|f_j|^2_hc(-\Psi)=0,
\end{equation}
 and $(f_j-f)_{z_0}\in
\mathcal{O} (K_M)_{z_0} \otimes J_{z_0}$, for any  $ z_0\in Z_0$ and $j\ge 1$. It follows from Lemma \ref{global convergence property of module} that there
exists a subsequence of $\{f_j\}_{j\in \mathbb{N}^+}$ compactly convergent to an $E$-valued holomorphic $(n,0)$ form $f_0$ on $\{\Psi<-t_0\}$ which satisfies
$$\int_{\{\Psi<-t_0\}}|f_0|^2_hc(-\Psi)=0$$
and
$(f_0-f)_{z_0}\in \mathcal{O} (K_M)_{z_0}\otimes J_{z_0}$ for any $z_0\in Z_0$. It follows from $\int_{\{\Psi<-t_0\}}|f_0|^2_hc(-\Psi)=0$ that we know $f_0\equiv 0$. Hence we have $f_{z_0}\in \mathcal{O} (K_M)_{z_0}\otimes J_{z_0}$ for any $z_0\in Z_0$. Statement (2) is proved.
\end{proof}

The following lemma shows the existence and uniqueness of the $E$-valued holomorphic $(n,0)$ form related to $G(t)$.
\begin{Lemma}
\label{existence of F}
Assume that $G(t)<+\infty$ for some $t\in [T,+\infty)$. Then there exists a unique $E$-valued
holomorphic $(n,0)$ form $F_t$ on $\{\Psi<-t\}$ satisfying
$$\ \int_{\{\Psi<-t\}}|F_t|^2_hc(-\Psi)=G(t)$$  and
$\ (F_t-f)\in
\mathcal{O} (K_M)_{z_0} \otimes J_{z_0}$, for any  $ z_0\in Z_0$.
\par
Furthermore, for any $E$-valued holomorphic $(n,0)$ form $\hat{F}$ on $\{\Psi<-t\}$ satisfying
$$\int_{\{\Psi<-t\}}|\hat{F}|^2_hc(-\Psi)<+\infty$$ and $\ (\hat{F}-f)\in
\mathcal{O} (K_M)_{z_0} \otimes J_{z_0}$, for any  $ z_0\in Z_0$. We have the following equality
\begin{equation}
\begin{split}
&\int_{\{\Psi<-t\}}|F_t|^2_hc(-\Psi)+
\int_{\{\Psi<-t\}}|\hat{F}-F_t|^2_hc(-\Psi)\\
=&\int_{\{\Psi<-t\}}|\hat{F}|^2_hc(-\Psi).
\label{orhnormal F}
\end{split}
\end{equation}
\end{Lemma}

\begin{proof} We firstly show the existence of $F_t$. As $G(t)<+\infty$, then there exists a sequence of $E$-valued holomorphic $(n,0)$ forms $\{f_j\}_{j\in \mathbb{N}^+}$ on $\{\Psi<-t\}$ such that $$\lim\limits_{j \to +\infty}\int_{\{\Psi<-t\}}|f_j|^2_hc(-\Psi)=G(t)$$ and $(f_j-f)\in
\mathcal{O} (K_M)_{z_0} \otimes J_{z_0}$, for any  $ z_0\in Z_0$ and any $j\ge 1$.
It follows from Lemma \ref{global convergence property of module} that there
exists a subsequence of $\{f_j\}_{j\in \mathbb{N}^+}$ compactly convergent to an $E$-valued holomorphic $(n,0)$ form $F$ on $\{\Psi<-t\}$ which satisfies
$$\int_{\{\Psi<-t\}}|F|^2_hc(-\Psi)\le G(t)$$
and
$(F-f)_{z_0}\in \mathcal{O} (K_M)_{z_0}\otimes J_{z_0}$ for any $z_0\in Z_0$. By the definition of $G(t)$, we have $\int_{\{\Psi<-t\}}|F|^2_hc(-\Psi)= G(t)$. Then we obtain the existence of $F_t(=F)$.

We prove the uniqueness of $F_t$ by contradiction: if not, there exist
two different holomorphic $(n,0)$ forms $f_1$ and $f_2$ on $\{\Psi<-t\}$
satisfying $\int_{\{\Psi<-t\}}|f_1|^2_hc(-\Psi)=
\int_{\{\Psi<-t\}}|f_2|^2_hc(-\Psi)=G(t)$, $(f_1-f)_{z_0}\in \mathcal{O} (K_M)_{z_0}\otimes J_{z_0}$ for any $z_0\in Z_0$ and $(f_2-f)_{z_0}\in \mathcal{O} (K_M)_{z_0}\otimes J_{z_0}$ for any $z_0\in Z_0$. Note that
\begin{equation}\nonumber
\begin{split}
\int_{\{\Psi<-t\}}|\frac{f_1+f_2}{2}|^2_hc(-\Psi)+
\int_{\{\Psi<-t\}}|\frac{f_1-f_2}{2}|^2_hc(-\Psi)\\
=\frac{1}{2}(\int_{\{\Psi<-t\}}|f_1|^2_hc(-\Psi)+
\int_{\{\Psi<-t\}}|f_1|^2_hc(-\Psi))=G(t),
\end{split}
\end{equation}
then we obtain that
\begin{equation}\nonumber
\begin{split}
\int_{\{\Psi<-t\}}|\frac{f_1+f_2}{2}|^2_hc(-\Psi)
< G(t)
\end{split}
\end{equation}
and $(\frac{f_1+f_2}{2}-f)_{z_0}\in \mathcal{O} (K_M)_{z_0}\otimes J_{z_0}$ for any $z_0\in Z_0$, which contradicts to the definition of $G(t)$.

Now we prove equality \eqref{orhnormal F}. Let $q$ be an $E$-valued holomorphic $(n,0)$ form on $\{\Psi<-t\}$
such that $\int_{\{\Psi<-t\}}|q|^2_hc(-\Psi)<+\infty$ and $q \in \mathcal{O} (K_M)_{z_0}\otimes J_{z_0}$ for any $z_0\in Z_0$.  It is clear that for any complex
number $\alpha$, $F_t+\alpha q$ satisfying $((F_t+\alpha q)-f) \in \mathcal{O} (K_M)_{z_0}\otimes J_{z_0}$ for any $z_0\in Z_0$ and
$\int_{\{\Psi<-t\}}|F_t|^2_hc(-\Psi) \leq \int_{\{\Psi<-t\}}|F_t+\alpha
q|^2_hc(-\Psi)$. Note that
\begin{equation}\nonumber
\begin{split}
\int_{\{\Psi<-t\}}|F_t+\alpha
q|^2_hc(-\Psi)-\int_{\{\psi<-t\}}|F_t|^2_hc(-\Psi)\geq 0
\end{split}
\end{equation}
(By considering $\alpha \to 0$) implies
\begin{equation}\nonumber
\begin{split}
\mathfrak{R} \int_{\{\Psi<-t\}}\langle F_t,\bar{q}\rangle_hc(-\Psi)=0,
\end{split}
\end{equation}
then we have
\begin{equation}\nonumber
\begin{split}
\int_{\{\Psi<-t\}}|F_t+q|^2_hc(-\Psi)=
\int_{\{\Psi<-t\}}(|F_t|^2_h+|q|^2_h)c(-\Psi).
\end{split}
\end{equation}
\par
Letting $q=\hat{F}-F_t$, we obtain equality \eqref{orhnormal F}.
\end{proof}

The following lemma shows the  lower semicontinuity property of $G(t)$.
\begin{Lemma}$G(t)$ is decreasing with respect to $t\in
[T,+\infty)$, such that $\lim \limits_{t \to t_0+0}G(t)=G(t_0)$ for any $t_0\in
[T,+\infty)$, and if $G(t)<+\infty$ for some $t>T$, then $\lim \limits_{t \to +\infty}G(t)=0$. Especially, $G(t)$ is lower semicontinuous on $[T,+\infty)$.
 \label{semicontinuous}
\end{Lemma}
\begin{proof}By the definition of $G(t)$, it is clear that $G(t)$ is decreasing on
$[T,+\infty)$. If $G(t)<+\infty$ for some $t>T$, by the dominated convergence theorem, we know $\lim\limits_{t\to +\infty}G(t)=0$. It suffices
to prove $\lim \limits_{t \to t_0+0}G(t)=G(t_0)$ . We prove it by
contradiction: if not, then $\lim \limits_{t \to t_0+0}G(t)<
G(t_0)$.

By using Lemma \ref{existence of F}, for any $t>t_0$, there exists a unique $E$-valued holomorphic $(n,0)$ form
$F_t$ on $\{\Psi<-t\}$ satisfying
$\int_{\{\Psi<-t\}}|F_t|^2_hc(-\Psi)=G(t)$ and $(F_t-f) \in \mathcal{O} (K_M)_{z_0}\otimes J_{z_0}$ for any $z_0\in Z_0$. Note that $G(t)$ is decreasing with respect to $t$. We have $\int_{\{\Psi<-t\}}|F_t|^2_hc(-\Psi)\leq \lim
\limits_{t \to t_0+0}G(t)$ for any $t>t_0$. If $\lim\limits_{t \to t_0+0}G(t)=+\infty$, the equality $\lim \limits_{t \to t_0+0}G(t)=G(t_0)$ obviously holds, thus it suffices to prove the case $\lim\limits_{t \to t_0+0}G(t)<+\infty$. It follows from $\int_{\{\Psi<-t\}}|F_t|^2_hc(-\Psi)\le \lim\limits_{t \to t_0+0}G(t)<+\infty$  holds for any $t\in (t_0,t_1]$ (where $t_1>t_0$ is a fixed number) and Lemma \ref{global convergence property of module} that there exists a subsequence of $\{F_t\}$ (denoted by $\{F_{t_j}\}$) compactly convergent to an $E$-valued holomorphic $(n,0)$ form $\hat{F}_{t_0}$ on $\{\Psi<-t_0\}$ satisfying
$$\int_{\{\Psi<-t_0\}}|\hat{F}_{t_0}|^2_hc(-\Psi)\le \lim\limits_{t \to t_0+0}G(t)<+\infty$$
and $(\hat{F}_{t_0}-f)_{z_0} \in \mathcal{O} (K_M)_{z_0}\otimes J_{z_0}$ for any $z_0\in Z_0$.

Then we obtain that $G(t_0)\leq
\int_{\{\Psi<-t_0\}}|\hat{F}_{t_0}|^2_hc(-\Psi)
\leq \lim \limits_{t\to t_0+0} G(t)$,
which contradicts $\lim \limits_{t\to t_0+0} G(t) <G(t_0)$. Thus we have $\lim \limits_{t \to t_0+0}G(t)=G(t_0)$.
\end{proof}

We consider the derivatives of $G(t)$ in the following lemma.

\begin{Lemma}
\label{derivatives of G}
Assume that $G(t_1)<+\infty$, where $t_1\in (T,+\infty)$. Then for any $t_0>t_1$, we have
\begin{equation}\nonumber
\begin{split}
\frac{G(t_1)-G(t_0)}{\int^{t_0}_{t_1} c(t)e^{-t}dt}\leq
\liminf\limits_{B \to
0+0}\frac{G(t_0)-G(t_0+B)}{\int_{t_0}^{t_0+B}c(t)e^{-t}dt},
\end{split}
\end{equation}
i.e.
\begin{equation}\nonumber
\frac{G(t_0)-G(t_1)}{\int_{T_1}^{t_0}
c(t)e^{-t}dt-\int_{T_1}^{t_1} c(t)e^{-t}dt} \geq
\limsup \limits_{B \to 0+0}
\frac{G(t_0+B)-G(t_0)}{\int_{T_1}^{t_0+B}
c(t)e^{-t}dt-\int_{T_1}^{t_0} c(t)e^{-t}dt}.
\end{equation}
\end{Lemma}

\begin{proof}
It follows from Lemma \ref{semicontinuous} that $G(t)<+\infty$ for any $t>t_1$. By Lemma \ref{existence of F}, there exists an $E$-valued holomorphic $(n,0)$ form $F_{t_0}$ on $\{\Psi<-t_0\}$, such that $(F_{t_0}-f)_{z_0}\in \mathcal{O} (K_M)_{z_0}\otimes J_{z_0}$ for any $z_0\in Z_0$ and $G(t_0)=\int_{\{\Psi<-t_0\}}|F_{t_0}|^2_hc(-\Psi)$.

It suffices to consider that $\liminf\limits_{B\to 0+0} \frac{G(t_0)-G(t_0+B)}{\int_{t_0}^{t_0+B}c(t)e^{-t}dt}\in [0,+\infty)$ because of the decreasing property of $G(t)$. Then there exists $1\ge B_j\to 0+0$ (as $j\to+\infty$) such that
\begin{equation}
	\label{derivatives of G c(t)1}
\lim\limits_{j\to +\infty} \frac{G(t_0)-G(t_0+B_j)}{\int_{t_0}^{t_0+B_j}c(t)e^{-t}dt}=\liminf\limits_{B\to 0+0} \frac{G(t_0)-G(t_0+B)}{\int_{t_0}^{t_0+B}c(t)e^{-t}dt}
\end{equation}
and $\{\frac{G(t_0)-G(t_0+B_j)}{\int_{t_0}^{t_0+B_j}c(t)e^{-t}dt}\}_{j\in\mathbb{N}^{+}}$ is bounded. As $c(t)e^{-t}$ is decreasing and positive on $(t,+\infty)$, then
\begin{equation}\label{derivatives of G c(t)2}
\begin{split}
\lim\limits_{j\to +\infty} \frac{G(t_0)-G(t_0+B_j)}{\int_{t_0}^{t_0+B_j}c(t)e^{-t}dt}
=&\left(\lim\limits_{j\to +\infty} \frac{G(t_0)-G(t_0+B_j)}{B_j}\right)\left(\frac{1}{\lim\limits_{t\to t_0+0}c(t)e^{-t}}\right)\\
=&\left(\lim\limits_{j\to +\infty} \frac{G(t_0)-G(t_0+B_j)}{B_j}\right)\left(\frac{e^{t_0}}{\lim\limits_{t\to t_0+0}c(t)}\right).
\end{split}
\end{equation}
Hence $\{\frac{G(t_0)-G(t_0+B_j)}{B_j}\}_{j\in\mathbb{N}^+}$ is uniformly bounded with respect to $j$.

As $t \leq v_{t_0,j}(t)$, the decreasing property of $c(t)e^{-t}$ shows that
\begin{equation}\nonumber
e^{-\Psi+v_{t_0,B_j}(\Psi)}c(-v_{t_0,B_j}(\Psi))\geq c(-\Psi).
\end{equation}
\par
It follows from Lemma \ref{L2 method in JM concavity} that, for any $B_j$, there exists an $E$-valued holomorphic $(n,0)$ form $\tilde{F}_j$ on $\{\Psi<-t_1\}$ such that
\begin{flalign}
&\int_{\{\Psi<-t_1\}}|\tilde{F}_j-(1-b_{t_0,B_j}(\Psi))F_{t_0}|^2_hc(-\Psi)\nonumber\\
\leq &
\int_{\{\Psi<-t_1\}}|\tilde{F}_j-(1-b_{t_0,B_j}(\Psi))F_{t_0}|^2_he^{-\Psi+v_{t_0,B_j}(\Psi)}c(-v_{t_0,B_j}(\Psi))\nonumber\\
\leq &
\int^{t_0+B_j}_{t_1}c(t)e^{-t}dt\int_{\{\Psi<-t_1\}}\frac{1}{B_j}
\mathbb{I}_{\{-t_0-B_j<\Psi<-t_0\}}|F_{t_0}|^2_he^{-\Psi}\nonumber\\
\leq &
\frac{e^{t_0+B_j}\int^{t_0+B_j}_{t_1}c(t)e^{-t}dt}{\inf
\limits_{t\in(t_0,t_0+B_j)}c(t)}\int_{\{\Psi<-t_1\}}\frac{1}{B_j}
\mathbb{I}_{\{-t_0-B_j<\Psi<-t_0\}}|F_{t_0}|^2_hc(-\Psi)\nonumber\\
= &
\frac{e^{t_0+B_j}\int^{t_0+B_j}_{t_1}c(t)e^{-t}dt}{\inf
\limits_{t\in(t_0,t_0+B_j)}c(t)}\times
\bigg(\int_{\{\Psi<-t_1\}}\frac{1}{B_j}\mathbb{I}_{\{\Psi<-t_0\}}|F_{t_0}|^2_hc(-\Psi)\nonumber\\
&-\int_{\{\Psi<-t_1\}}\frac{1}{B_j}\mathbb{I}_{\{\Psi<-t_0-B_j\}}|F_{t_0}|^2_hc(-\Psi)\bigg)\nonumber\\
\leq &
\frac{e^{t_0+B_j}\int^{t_0+B_j}_{t_1}c(t)e^{-t}dt}{\inf
\limits_{t\in(t_0,t_0+B_j)}c(t)} \times
\frac{G(t_0)-G(t_0+B_j)}{B_j}<+\infty.
\label{derivative of G 1}
\end{flalign}

Note that $b_{t_0,B_j}(t)=0$ for $t\le-t_0-B_j$, $b_{t_0,B_j}(t)=1$ for $t\ge t_0$, $v_{t_0,B_j}(t)>-t_0-B_j$ and $c(t)e^{-t}$ is decreasing with respect to $t$. It follows from inequality \eqref{derivative of G 1} that $(F_j-F_{t_0})_{z_0}\in \mathcal{O} (K_M)_{z_0}\otimes I(h,\Psi)_{z_0} \subset \mathcal{O} (K_M)_{z_0}\otimes J_{z_0}$ for any $z_0\in Z_0$.

Note that
\begin{equation}\label{derivative of G 2}
\begin{split}
&\int_{\{\Psi<-t_1\}}|\tilde{F}_j|^2_hc(-\Psi)\\
\le&2\int_{\{\Psi<-t_1\}}|\tilde{F}_j-(1-b_{t_0,B_j}(\Psi))F_{t_0}|^2_hc(-\Psi)
+2\int_{\{\Psi<-t_1\}}|(1-b_{t_0,B_j}(\Psi))F_{t_0}|^2_hc(-\Psi)\\
\le&2
\frac{e^{t_0+B_j}\int^{t_0+B_j}_{t_1}c(t)e^{-t}dt}{\inf
\limits_{t\in(t_0,t_0+B_j)}c(t)} \times
\frac{G(t_0)-G(t_0+B_j)}{B_j}
+2\int_{\{\Psi<-t_0\}}|F_{t_0}|^2_hc(-\Psi).
\end{split}
\end{equation}

We also note that $B_j\le 1$, $\frac{G(t_0)-G(t_0+B_j)}{B_j}$ is uniformly bounded with respect to $j$ and $G(t_0)=\int_{\{\Psi<-t_0\}}|F_{t_0}|^2_hc(-\Psi)$. It follows from inequality \eqref{derivative of G 2} that we know $\int_{\{\Psi<-t_1\}}|\tilde{F}_j|^2e^{-\varphi}c(-\Psi)$ is uniformly bounded with respect to $j$.

It follows from Lemma \ref{global convergence property of module} that there exists a subsequence of $\{\tilde{F}_j\}_{j\in \mathbb{N}^+}$ compactly convergent to an $E$-valued holomorphic $(n,0)$ form $\tilde{F}_{t_1}$ on $\{\Psi<-t_1\}$ which satisfies
$$\int_{\{\Psi<-t_1\}}|\tilde{F}_{t_1}|^2_hc(-\Psi)\le \liminf_{j\to+\infty} \int_{\{\Psi<-t_1\}}|\tilde{F}_j|^2_hc(-\Psi)<+\infty,$$
and $(\tilde{F}_{t_1}-F_{t_0})_{z_0}\in \mathcal{O} (K_M)_{z_0}\otimes  J_{z_0}$ for any $z_0\in Z_0$.

Note that
$\lim_{j\to+\infty}b_{t_0,B_j}(t)=\mathbb{I}_{\{t\ge -t_0\}}$ and
\begin{equation}\nonumber
v_{t_0}(t):=\lim_{j\to+\infty}v_{t_0,B_j}(t)=\left\{
\begin{aligned}
&-t_0  &\text{ if } & x<-t_0, \\
&\ t  &\text{ if }  & x\ge t_0 .
\end{aligned}
\right.
\end{equation}

It follows from inequality \eqref{derivative of G 1} and Fatou's lemma that

\begin{flalign}
\label{derivative of G 3}
&\int_{\{\Psi<-t_0\}}|\tilde{F}_{t_1}-F_{t_0}|^2_hc(-\Psi)
+\int_{\{-t_0\le\Psi<-t_1\}}|\tilde{F}_{t_1}|^2_hc(-\Psi)\nonumber\\
\leq &
\int_{\{\Psi<-t_1\}}|\tilde{F}_{t_1}-\mathbb{I}_{\{\Psi< -t_0\}}F_{t_0}|^2_he^{-\Psi+v_{t_0}(\Psi)}c(-v_{t_0}(\Psi))\nonumber\\
\le&\liminf_{j\to+\infty}\int_{\{\Psi<-t_1\}}|\tilde{F}_j-(1-b_{t_0,B_j}(\Psi))F_{t_0}|^2_hc(-\Psi)\nonumber\\
\leq &\liminf_{j\to+\infty}
\bigg(\frac{e^{t_0+B_j}\int^{t_0+B_j}_{t_1}c(t)e^{-t}dt}{\inf
\limits_{t\in(t_0,t_0+B_j)}c(t)} \times
\frac{G(t_0)-G(t_0+B_j)}{B_j}\bigg).
\end{flalign}

It follows from Lemma \ref{existence of F}, equality \eqref{derivatives of G c(t)1}, equality \eqref{derivatives of G c(t)2} and inequality \eqref{derivative of G 3} that we have

\begin{equation}
\label{derivative of G 4}
\begin{split}
&\int_{\{\Psi<-t_1\}}|\tilde{F}_{t_1}|^2_hc(-\Psi)
-\int_{\{\Psi<-t_0\}}|F_{t_0}|^2_hc(-\Psi)\\
\le&\int_{\{\Psi<-t_0\}}|\tilde{F}_{t_1}-F_{t_0}|^2_hc(-\Psi)
+\int_{\{-t_0\le\Psi<-t_1\}}|\tilde{F}_{t_1}|^2_hc(-\Psi)\\
\leq &
\int_{\{\Psi<-t_1\}}|\tilde{F}_{t_1}-\mathbb{I}_{\{\Psi< -t_0\}}F_{t_0}|^2_he^{-\Psi+v_{t_0}(\Psi)}c(-v_{t_0}(\Psi))\\
\le&\liminf_{j\to+\infty}\int_{\{\Psi<-t_1\}}|\tilde{F}_j-(1-b_{t_0,B_j}(\Psi))F_{t_0}|^2_hc(-\Psi)\\
\leq &\liminf_{j\to+\infty}
\big(\frac{e^{t_0+B_j}\int^{t_0+B_j}_{t_1}c(t)e^{-t}dt}{\inf
\limits_{t\in(t_0,t_0+B_j)}c(t)} \times
\frac{G(t_0)-G(t_0+B_j)}{B_j}\big)\\
\le &\bigg(\int^{t_0}_{t_1}c(t)e^{-t}dt\bigg)\liminf\limits_{B\to 0+0} \frac{G(t_0)-G(t_0+B)}{\int_{t_0}^{t_0+B}c(t)e^{-t}dt}.
\end{split}
\end{equation}
Note that $(\tilde{F}_{t_1}-F_{t_0})_{z_0}\in \mathcal{O} (K_M)_{z_0}\otimes  J_{z_0}$ for any $z_0\in Z_0$. It follows from the definition of $G(t)$ and inequality \eqref{derivative of G 4} that we have

\begin{equation}
\label{derivative of G 5}
\begin{split}
&G(t_1)-G(t_0)\\
\le&\int_{\{\Psi<-t_1\}}|\tilde{F}_{t_1}|^2_hc(-\Psi)
-\int_{\{\Psi<-t_0\}}|F_{t_0}|^2_hc(-\Psi)\\
\le&\int_{\{\Psi<-t_1\}}|\tilde{F}_{t_1}-\mathbb{I}_{\{\Psi< -t_0\}}F_{t_0}|^2_hc(-\Psi)\\
\leq &
\int_{\{\Psi<-t_1\}}|\tilde{F}_{t_1}-\mathbb{I}_{\{\Psi< -t_0\}}F_{t_0}|^2_he^{-\Psi+v_{t_0}(\Psi)}c(-v_{t_0}(\Psi))\\\
\le &\big(\int^{t_0}_{t_1}c(t)e^{-t}dt\big)\liminf\limits_{B\to 0+0} \frac{G(t_0)-G(t_0+B)}{\int_{t_0}^{t_0+B}c(t)e^{-t}dt}.
\end{split}
\end{equation}

Lemma \ref{derivatives of G} is proved.
\end{proof}

The following property of concave
functions will be used in the proof of Theorem \ref{main theorem}.
\begin{Lemma}[see \cite{G16}]
Let $H(r)$ be a lower semicontinuous function on $(0,R]$. Then $H(r)$ is concave
if and only if
\begin{equation}\nonumber
\begin{split}
\frac{H(r_1)-H(r_2)}{r_1-r_2} \leq
\liminf\limits_{r_3 \to r_2-0}
\frac{H(r_3)-H(r_2)}{r_3-r_2}
\end{split}
\end{equation}
holds for any $0<r_2<r_1 \leq R$.
\label{characterization of concave function}
\end{Lemma}
\section{Proof of Theorem \ref{main theorem}, Remark \ref{infty2}, Corollary
\ref{necessary condition for linear of G} and  Remark \ref{rem:linear}}
We firstly prove Theorem \ref{main theorem}.

\begin{proof} We firstly show that if $G(t_0)<+\infty$ for some $t_0> T$, then $G(t_1)<+\infty$ for any $T< t_1<t_0$. As $G(t_0)<+\infty$, it follows from Lemma \ref{existence of F} that there exists an unique $E$-valued
holomorphic $(n,0)$ form $F_{t_0}$ on $\{\Psi<-t\}$ satisfying
$$\ \int_{\{\Psi<-t_0\}}|F_{t_0}|^2_hc(-\Psi)=G(t_0)<+\infty$$  and
$\ (F_{t_0}-f)_{z_0}\in
\mathcal{O} (K_M)_{z_0} \otimes J_{z_0}$, for any  $ z_0\in Z_0$.

It follows from Lemma \ref{L2 method in JM concavity} that there exists an $E$-valued holomorphic $(n,0)$ form $\tilde{F}_1$ on $\{\Psi<-t_1\}$ such that

 \begin{equation}\label{main theorem 1}
  \begin{split}
      & \int_{\{\Psi<-t_1\}}|\tilde{F}_1-(1-b_{t_0,B}(\Psi))F_{t_0}|^2_he^{v_{t_0,B}(\Psi)-\Psi}c(-v_{t_0,B}(\Psi)) \\
      \le & (\int_{t_1}^{t_0+B}c(s)e^{-s}ds)
       \int_{M}\frac{1}{B}\mathbb{I}_{\{-t_0-B<\Psi<-t_0\}}|F_{t_0}|^2_he^{-\Psi}<+\infty.
  \end{split}
\end{equation}
Note that $b_{t_0,B}(t)=0$ on $\{\Psi<-t_0-B\}$ and $v_{t_0,B}(\Psi)> -t_0-B$. We have $e^{v_{t_0,B}(\Psi)}c(-v_{t_0,B}(\Psi))$ has a positive lower bound. It follows from inequality \eqref{main theorem 1} that
we have  $\ (\tilde{F}_{1}-F_{t_0})_{z_0}
\in \mathcal{O} (K_M)_{z_0}\otimes I(h,\Psi)_{z_0} \subset \mathcal{O} (K_M)_{z_0}\otimes J_{z_0}$ for any $z_0\in Z_0$, which implies that
$(\tilde{F}_{1}-f)_{z_0}\in
\mathcal{O} (K_M)_{z_0} \otimes J_{z_0}$, for any  $ z_0\in Z_0$. As $v_{t_0,B}(\Psi)\ge\Psi$ and $c(t)e^{-t}$ is decreasing with respect to $t$, it follows from inequality \eqref{main theorem 1} that we have
 \begin{equation}\label{main theorem 2}
  \begin{split}
      & \int_{\{\Psi<-t_1\}}|\tilde{F}_1-(1-b_{t_0,B}(\Psi))F_{t_0}|^2_hc(-\Psi)
      \\
      \le&\int_{\{\Psi<-t_1\}}|\tilde{F}_1-(1-b_{t_0,B}(\Psi))F_{t_0}|^2_he^{v_{t_0,B}(\Psi)-\Psi}c(-v_{t_0,B}(\Psi)) \\
      \le & (\int_{t_1}^{t_0+B}c(s)e^{-s}ds)
       \int_{M}\frac{1}{B}\mathbb{I}_{\{-t_0-B<\Psi<-t_0\}}|F_{t_0}|^2_he^{-\Psi}<+\infty.
  \end{split}
\end{equation}
Then we have
 \begin{equation}\label{main theorem 3}
  \begin{split}
  &\int_{\{\Psi<-t_1\}}|\tilde{F}_1|^2_hc(-\Psi)\\
     \le & 2\int_{\{\Psi<-t_1\}}|\tilde{F}_1-(1-b_{t_0,B}(\Psi))F_{t_0}|^2_hc(-\Psi)
      +2\int_{\{\Psi<-t_1\}}|(1-b_{t_0,B}(\Psi))F_{t_0}|^2_hc(-\Psi)\\
      \le & 2(\int_{t_1}^{t_0+B}c(s)e^{-s}ds)
       \int_{M}\frac{1}{B}\mathbb{I}_{\{-t_0-B<\Psi<-t_0\}}|F_{t_0}|^2_he^{-\Psi}
       +2\int_{\{\Psi<-t_0\}}|F_{t_0}|^2_hc(-\Psi)\\
       <&+\infty.
  \end{split}
\end{equation}
Hence we have $G(t_1)\le \int_{\{\Psi<-t_1\}}|\tilde{F}_1|^2_hc(-\Psi)<+\infty$.

Now, it follows from Lemma \ref{semicontinuous}, Lemma \ref{derivatives of G} and Lemma \ref{characterization of concave function} that we know $G(h^{-1}(r))$ is concave with respect to $r$. It follows from Lemma \ref{semicontinuous} that $\lim\limits_{t\to T+0}G(t)=G(T)$ and $\lim\limits_{t\to+\infty}G(t)=0$.

Theorem \ref{main theorem} is proved.
\end{proof}

Now we prove Remark \ref{infty2}.
\begin{proof}Note that if there exists a positive decreasing concave function $g(t)$ on $(a,b)\subset\mathbb{R}$ and $g(t)$ is not a constant function, then $b<+\infty$.

Assume that $G(t_0)<+\infty$ for some $t_0\geq T$. As $f_{z_0}\notin
\mathcal{O} (K_M)_{z_0} \otimes J_{z_0}$ for some  $ z_0\in Z_0$, Lemma \ref{characterization of g(t)=0} shows that $G(t_0)\in(0,+\infty)$. Following from Theorem \ref{main theorem} we know $G({h}^{-1}(r))$ is concave with respect to $r\in(\int_{T_1}^{T}c(t)e^{-t}dt,\int_{T_1}^{+\infty}c(t)e^{-t}dt)$ and $G({h}^{-1}(r))$ is not a constant function, therefore we obtain $\int_{T_1}^{+\infty}c(t)e^{-t}dt<+\infty$, which contradicts to $\int_{T_1}^{+\infty}c(t)e^{-t}dt=+\infty$. Thus we have $G(t)\equiv+\infty$.

When $G(t_2)\in(0,+\infty)$ for some $t_2\in[T,+\infty)$, Lemma \ref{characterization of g(t)=0} shows that $f_{z_0}\notin
\mathcal{O} (K_M)_{z_0} \otimes J_{z_0}$, for any  $ z_0\in Z_0$. Combining the above discussion, we know $\int_{T_1}^{+\infty}c(t)e^{-t}dt<+\infty$. Using Theorem \ref{main theorem}, we obtain that $G(\hat{h}^{-1}(r))$ is concave with respect to  $r\in (0,\int_{T}^{+\infty}c(t)e^{-t}dt)$, where $\hat{h}(t)=\int_{t}^{+\infty}c(l)e^{-l}dl$.

Thus, Remark \ref{infty2} holds.
\end{proof}

Now we prove Corollary \ref{necessary condition for linear of G}.

\begin{proof} As $G(h^{-1}(r))$ is linear with respect to $r\in[0,\int_T^{+\infty}c(s)e^{-s}ds)$, we have $G(t)=\frac{G(T_1)}{\int_{T_1}^{+\infty}c(s)e^{-s}ds}\int_{t}^{+\infty}c(s)e^{-s}ds$ for any $t\in[T,+\infty)$ and $T_1 \in (T,+\infty)$.

We follow the notation and the construction in Lemma \ref{derivatives of G}. Let $t_0>t_1> T$ be given. It follows from $G(h^{-1}(r))$ is linear with respect to $r\in[0,\int_T^{+\infty}c(s)e^{-s}ds)$ that we know that all inequalities in \eqref{derivative of G 5} should be equalities, i.e., we have

\begin{equation}
\label{necessary condition for linear of G 1}
\begin{split}
&G(t_1)-G(t_0)\\
=&\int_{\{\Psi<-t_1\}}|\tilde{F}_{t_1}|^2_hc(-\Psi)
-\int_{\{\Psi<-t_0\}}|F_{t_0}|^2_hc(-\Psi)\\
=&\int_{\{\Psi<-t_1\}}|\tilde{F}_{t_1}-\mathbb{I}_{\{\Psi< -t_0\}}F_{t_0}|^2_hc(-\Psi)\\
= &
\int_{\{\Psi<-t_1\}}|\tilde{F}_{t_1}-\mathbb{I}_{\{\Psi< -t_0\}}F_{t_0}|^2_he^{-\Psi+v_{t_0}(\Psi)}c(-v_{t_0}(\Psi))\\
= &\big(\int^{t_0}_{t_1}c(t)e^{-t}dt\big)\liminf\limits_{B\to 0+0} \frac{G(t_0)-G(t_0+B)}{\int_{t_0}^{t_0+B}c(t)e^{-t}dt}.
\end{split}
\end{equation}
Note that $G(t_0)=\int_{\{\Psi<-t_0\}}|F_{t_0}|^2_hc(-\Psi)$. Equality \eqref{necessary condition for linear of G 1} shows that $G(t_1)=\int_{\{\Psi<-t_1\}}|\tilde{F}_{t_1}|^2_hc(-\Psi)$.

Note that on $\{\Psi\ge -t_0\}$, we have $e^{-\Psi+v_{t_0}(\Psi)}c(-v_{t_0}(\Psi))=c(-\Psi)$. It follows from
\begin{equation}\nonumber
\begin{split}
&\int_{\{\Psi<-t_1\}}|\tilde{F}_{t_1}-\mathbb{I}_{\{\Psi< -t_0\}}F_{t_0}|^2_hc(-\Psi)\\
= &
\int_{\{\Psi<-t_1\}}|\tilde{F}_{t_1}-\mathbb{I}_{\{\Psi< -t_0\}}F_{t_0}|^2_he^{-\Psi+v_{t_0}(\Psi)}c(-v_{t_0}(\Psi))\\
\end{split}
\end{equation}
that we have (note that $v_{t_0}(\Psi)=-t_0$ on  $\{\Psi< -t_0\}$)
\begin{equation}
\begin{split}
\label{necessary condition for linear of G 2}
&\int_{\{\Psi<-t_0\}}|\tilde{F}_{t_1}-F_{t_0}|^2_hc(-\Psi)\\
= &
\int_{\{\Psi<-t_0\}}|\tilde{F}_{t_1}-F_{t_0}|^2_he^{-\Psi-t_0}c(t_0).\\
\end{split}
\end{equation}
As $\int_{T}^{+\infty}c(t)e^{-t}dt<+\infty$ and $c(t)e^{-t}$ is decreasing with respect to $t$, we know that there exists $t_2>t_0$ such that $c(t)e^{-t}<c(t_0)e^{-t_0}-\epsilon$ for any $t\ge t_2$, where $\epsilon>0$ is a constant. Then equality \eqref{necessary condition for linear of G 2} implies that
\begin{equation}
\begin{split}
\label{necessary condition for linear of G 3}
&\epsilon\int_{\{\Psi<-t_2\}}|\tilde{F}_{t_1}-F_{t_0}|^2_he^{-\Psi}\\
\le &
\int_{\{\Psi<-t_2\}}|\tilde{F}_{t_1}-F_{t_0}|^2_h(e^{-\Psi-t_0}c(t_0)-c(-\Psi))\\
\le &
\int_{\{\Psi<-t_0\}}|\tilde{F}_{t_1}-F_{t_0}|^2_h(e^{-\Psi-t_0}c(t_0)-c(-\Psi))\\
= &0.
\end{split}
\end{equation}
Note that for any relatively compact subset  $K\subset \{\Psi<-t_2\}$, $|\tilde{F}_{t_1}-F_{t_0}|^2_he^{-\Psi}=|(\tilde{F}_{t_1}-F_{t_0})F|^2_he^{-\psi}
=|\tilde{F}_{t_1}F-F_{t_0}F|^2_{\tilde{h}}\ge |\tilde{F}_{t_1}F-F_{t_0}F|^2_{\tilde{h}_{K,1}}$ on $K$, and the integrand in \eqref{necessary condition for linear of G 3} is nonnegative, we must have $\tilde{F}_{t_1}|_{\{\Psi<-t_0\}}=F_{t_0}$.

It follows from Lemma \ref{existence of F} that for any $t>T$, there exists an unique $E$-valued holomorphic $(n,0)$ form $F_t$ on $\{\Psi<-t\}$ satisfying
$$\ \int_{\{\Psi<-t\}}|F_t|^2_hc(-\Psi)=G(t)$$  and
$\ (F_t-f)_{z_0}\in
\mathcal{O} (K_M)_{z_0} \otimes J_{z_0}$, for any  $ z_0\in Z_0$.  By the above discussion, we know $F_{t}=F_{t'}$ on $\{\Psi<-\max{\{t,t'\}}\}$ for any $t\in(T,+\infty)$ and $t'\in(T,+\infty)$. Hence combining $\lim_{t\rightarrow T+0}G(t)=G(T)$, we obtain that there  exists an unique $E$-valued holomorphic $(n,0)$ form $\tilde{F}$ on $\{\Psi<-T\}$ satisfying $(\tilde{F}-f)_{z_0}\in
\mathcal{O} (K_M)_{z_0} \otimes J_{z_0}$ for any  $z_0\in Z_0$ and $G(t)=\int_{\{\Psi<-t\}}|\tilde{F}|^2_hc(-\Psi)$ for any $t\ge T$.

Secondly, we prove equality \eqref{other a also linear}.

As $a(t)$ is a nonnegative measurable function on $(T,+\infty)$, then there exists a sequence of functions $\{\sum\limits_{j=1}^{n_i}a_{ij}\mathbb{I}_{E_{ij}}\}_{i\in\mathbb{N}^+}$ $(n_i<+\infty$ for any $i\in\mathbb{N}^+)$ satisfying that $\sum\limits_{j=1}^{n_i}a_{ij}\mathbb{I}_{E_{ij}}$ is increasing with respect to $i$ and $\lim\limits_{i\to +\infty}\sum\limits_{j=1}^{n_i}a_{ij}\mathbb{I}_{E_{ij}}=a(t)$ for any $t\in(T,+\infty)$, where $E_{ij}$ is a Lebesgue measurable subset of $(T,+\infty)$ and $a_{ij}\ge 0$ is a constant for any $i,j$.  It follows from Levi's Theorem that it suffices to prove the case that $a(t)=\mathbb{I}_{E}(t)$, where $E\subset\subset (T,+\infty)$ is a Lebesgue measurable set.

Note that $G(t)=\int_{\{\Psi<-t\}}|\tilde{F}|^2_hc(-\Psi)=\frac{G(T_1)}{\int_{T_1}^{+\infty}c(s)e^{-s}ds}
\int_{t}^{+\infty}c(s)e^{-s}ds$ where $T_1 \in (T,+\infty)$, then
  \begin{equation}\label{linear 3.4}
\int_{\{-t_1\le\Psi<-t_2\}}|\tilde{F}|^2_hc(-\Psi)=\frac{G(T_1)}{\int_{T_1}^{+\infty}c(s)e^{-s}ds}
\int_{t_2}^{t_1}c(s)e^{-s}ds
  \end{equation}
  holds for any $T\le t_2<t_1<+\infty$. It follows from the dominated convergence theorem and equality \eqref{linear 3.4} that
    \begin{equation}\label{linear 3.5}
\int_{\{z\in M:-\Psi(z)\in N\}}|\tilde{F}|^2_h=0
  \end{equation}
  holds for any $N\subset\subset (T,+\infty)$ such that $\mu(N)=0$, where $\mu$ is the Lebesgue measure on $\mathbb{R}$.

  As $c(t)e^{-t}$ is decreasing on $(T,+\infty)$, there are at most countable points denoted by $\{s_j\}_{j\in \mathbb{N}^+}$ such that $c(t)$ is not continuous at $s_j$. Then there is a decreasing sequence of open sets $\{U_k\}$,  such that
$\{s_j\}_{j\in \mathbb{N}^+}\subset U_k\subset (T,+\infty)$ for any $k$, and $\lim\limits_{k \to +\infty}\mu(U_k)=0$. Choosing any closed interval $[t'_2,t'_1]\subset (T,+\infty)$, then we have
\begin{equation}\label{linear 3.6}
\begin{split}
&\int_{\{-t'_1\le\Psi<-t'_2\}}|\tilde{F}|^2_h\\
=&\int_{\{z\in M:-\Psi(z)\in(t'_2,t'_1]\backslash U_k\}}|\tilde{F}|^2_h+
\int_{\{z\in M:-\Psi(z)\in[t'_2,t'_1]\cap U_k\}}|\tilde{F}|^2_h\\
=&\lim_{n\to+\infty}\sum_{i=0}^{n-1}\int_{\{z\in M:-\Psi(z)\in I_{i,n}\backslash U_k\}}|\tilde{F}|^2_h+
\int_{\{z\in M:-\Psi(z)\in[t'_2,t'_1]\cap U_k\}}|\tilde{F}|^2_h,
\end{split}
\end{equation}
where $I_{i,n}=(t'_1-(i+1)\alpha_n,t'_1-i\alpha_n]$ and $\alpha_n=\frac{t'_1-t'_2}{n}$. Note that
\begin{equation}\label{linear 3.7}
\begin{split}
&\lim_{n\to+\infty}\sum_{i=0}^{n-1}\int_{\{z\in M:-\Psi(z)\in I_{i,n}\backslash U_k\}}|\tilde{F}|^2_h\\
\le&\limsup_{n\to+\infty}\sum_{i=0}^{n-1}\frac{1}{\inf_{I_{i,n}\backslash U_k}c(t)}\int_{\{z\in M:-\Psi(z)\in I_{i,n}\backslash U_k\}}|\tilde F|^2_hc(-\Psi).
\end{split}
\end{equation}

It follows from equality \eqref{linear 3.4} that inequality \eqref{linear 3.7} becomes
\begin{equation}\label{linear 3.8}
\begin{split}
&\lim_{n\to+\infty}\sum_{i=0}^{n-1}\int_{\{z\in M:-\Psi(z)\in I_{i,n}\backslash U_k\}}|\tilde F|^2_h\\
\le&\frac{G(T_1)}{\int_{T_1}^{+\infty}c(s)e^{-s}ds}
\limsup_{n\to+\infty}\sum_{i=0}^{n-1}\frac{1}{\inf_{I_{i,n}\backslash U_k}c(t)}\int_{I_{i,n}\backslash U_k}c(s)e^{-s}ds.
\end{split}
\end{equation}
It is clear that $c(t)$ is uniformly continuous and has positive lower bound and upper bound on $[t'_2,t'_1]\backslash U_k$. Then we have
\begin{equation}\label{linear 3.9}
\begin{split}
&\limsup_{n\to+\infty}\sum_{i=0}^{n-1}\frac{1}{\inf_{I_{i,n}\backslash U_k}c(t)}\int_{I_{i,n}\backslash U_k}c(s)e^{-s}ds \\
\le&\limsup_{n\to+\infty}\sum_{i=0}^{n-1}\frac{\sup_{I_{i,n}\backslash U_k}c(t)}{\inf_{I_{i,n}\backslash U_k}c(t)}\int_{I_{i,n}\backslash U_k}e^{-s}ds\\
=&\int_{(t'_2,t'_1]\backslash U_k}e^{-s}ds.
\end{split}
\end{equation}

Combining inequality \eqref{linear 3.6}, \eqref{linear 3.8} and \eqref{linear 3.9}, we have
\begin{equation}\label{linear 3.10}
\begin{split}
&\int_{\{-t'_1\le\Psi<-t'_2\}}|\tilde{F}|^2_h\\
=&\int_{\{z\in M:-\Psi(z)\in(t'_2,t'_1]\backslash U_k\}}|\tilde{F}|^2_h+
\int_{\{z\in M:-\Psi(z)\in[t'_2,t'_1]\cap U_k\}}|\tilde{F}|^2_h\\
\le&\frac{G(T_1)}{\int_{T_1}^{+\infty}c(s)e^{-s}ds}\int_{(t'_2,t'_1]\backslash U_k}e^{-s}ds+
\int_{\{z\in M:-\Psi(z)\in[t'_2,t'_1]\cap U_k\}}|\tilde{F}|^2_h.
\end{split}
\end{equation}
Let $k\to +\infty$, following from equality \eqref{linear 3.5} and inequality \eqref{linear 3.10}, then we obtain that
\begin{equation}\label{linear 3.11}
\begin{split}
\int_{\{-t'_1\le\Psi<-t'_2\}}|\tilde{F}|^2_h
\le\frac{G(T_1)}{\int_{T_1}^{+\infty}c(s)e^{-s}ds}\int_{t'_2}^{t'_1}e^{-s}ds.
\end{split}
\end{equation}
Following from a similar discussion we can obtain that
\begin{equation}\nonumber
\begin{split}
\int_{\{-t'_1\le\Psi<-t'_2\}}|\tilde{F}|^2_h
\ge\frac{G(T_1)}{\int_{T_1}^{+\infty}c(s)e^{-s}ds}\int_{t'_2}^{t'_1}e^{-s}ds.
\end{split}
\end{equation}
Then combining inequality \eqref{linear 3.11}, we know
\begin{equation}\label{linear 3.12}
\begin{split}
\int_{\{-t'_1\le\Psi<-t'_2\}}|\tilde{F}|^2_h
=\frac{G(T_1)}{\int_{T_1}^{+\infty}c(s)e^{-s}ds}\int_{t'_2}^{t'_1}e^{-s}ds.
\end{split}
\end{equation}
Then it is clear that for any open set $U\subset (T,+\infty)$ and compact set $V\subset (T,+\infty)$,
$$
\int_{\{z\in M;-\Psi(z)\in U\}}|\tilde{F}|^2_h
=\frac{G(T_1)}{\int_{T_1}^{+\infty}c(s)e^{-s}ds}\int_{U}e^{-s}ds,
$$
and
$$
\int_{\{z\in M;-\Psi(z)\in V\}}|\tilde{F}|^2_h
=\frac{G(T_1)}{\int_{T_1}^{+\infty}c(s)e^{-s}ds}\int_{V}e^{-s}ds.
$$
As $E\subset\subset (T,+\infty)$, then $E\cap(t_2,t_1]$ is a Lebesgue measurable subset of $(T+\frac{1}{n},n)$ for some large $n$, where $T\le t_2<t_1\le+\infty$. Then there exists a sequence of compact sets $\{V_j\}$ and a sequence of open subsets $\{V'_j\}$ satisfying $V_1\subset \ldots \subset V_j\subset V_{j+1}\subset\ldots \subset E\cap(t_2,t_1]\subset \ldots \subset V'_{j+1}\subset V'_j\subset \ldots\subset V'_1\subset\subset (T,+\infty)$ and $\lim\limits_{j\to +\infty}\mu(V'_j-V_j)=0$, where $\mu$ is the Lebesgue measure on $\mathbb{R}$. Then we have
\begin{equation}\nonumber
\begin{split}
\int_{\{-t'_1\le\Psi<-t'_2\}}|\tilde{F}|^2_h\mathbb{I}_E(-\Psi)
=&\int_{z\in M:-\Psi(z)\in E\cap (t_2,t_1]}|\tilde{F}|^2_h\\
\le&\liminf_{j\to+\infty}\int_{\{z\in M:-\Psi(z)\in V'_j\}}|\tilde{F}|^2_h\\
\le&\liminf_{j\to+\infty}\frac{G(T_1)}{\int_{T_1}^{+\infty}c(s)e^{-s}ds}\int_{V'_j}e^{-s}ds\\
\le&\frac{G(T_1)}{\int_{T_1}^{+\infty}c(s)e^{-s}ds}\int_{E\cap(t_2,t_1]}e^{-s}ds\\
=&\frac{G(T_1)}{\int_{T_1}^{+\infty}c(s)e^{-s}ds}\int_{t_2}^{t_1}e^{-s}\mathbb{I}_E(s)ds,
\end{split}
\end{equation}
and
\begin{equation}\nonumber
\begin{split}
\int_{\{-t'_1\le\Psi<-t'_2\}}|\tilde{F}|^2_h\mathbb{I}_E(-\Psi)
\ge&\liminf_{j\to+\infty}\int_{\{z\in M:-\Psi(z)\in V_j\}}|\tilde{F}|^2_h\\
\ge&\liminf_{j\to+\infty}\frac{G(T_1)}{\int_{T_1}^{+\infty}c(s)e^{-s}ds}\int_{V_j}e^{-s}ds\\
=&\frac{G(T_1)}{\int_{T_1}^{+\infty}c(s)e^{-s}ds}\int_{t_2}^{t_1}e^{-s}\mathbb{I}_E(s)ds,
\end{split}
\end{equation}
which implies that
$$\int_{\{-t'_1\le\Psi<-t'_2\}}|\tilde{F}|^2_h\mathbb{I}_E(-\Psi)=
\frac{G(T_1)}{\int_{T_1}^{+\infty}c(s)e^{-s}ds}\int_{t_2}^{t_1}e^{-s}\mathbb{I}_E(s)ds.$$
Hence we know that equality \eqref{other a also linear} holds.

Corollary \ref{necessary condition for linear of G} is proved.
\end{proof}

Now we prove Remark \ref{rem:linear}.

\begin{proof}[Proof of Remark \ref{rem:linear}]
By the definition of $G(t;\tilde{c})$, we have $G(t_0;\tilde{c})\le\int_{\{\Psi<-t_0\}}|\tilde{F}|^2_h\tilde{c}(-\Psi)$, where $\tilde{F}$ is the holomorphic $(n,0)$ form on $\{\Psi<-T\}$ such that $G(t)=\int_{\{\Psi<-t\}}|\tilde{F}|^2_hc(-\Psi)$ for any $t\ge T$.  Hence we only consider the case $G(t_0;\tilde{c})<+\infty$.

By the definition of $G(t;\tilde{c})$, we can choose an $E$-valued holomorphic $(n,0)$ form $F_{t_0,\tilde{c}}$ on $\{\Psi<-t_0\}$ satisfying $\ (F_{t_0,\tilde{c}}-f)_{z_0}\in
\mathcal{O} (K_M)_{z_0} \otimes J_{z_0}$, for any  $ z_0\in Z_0$ and $\int_{ \{ \Psi<-t_0\}}|F_{t_0,\tilde{c}}|^2_h\tilde{c}(-\Psi)<+\infty$. As $\mathcal{H}^2(\tilde{c},t_0)\subset \mathcal{H}^2(c,t_0)$, we have $\int_{ \{ \Psi<-t_0\}}|F_{t_0,\tilde{c}}|^2_hc(-\Psi)<+\infty$. Using
Lemma \ref{existence of F}, we obtain that
\begin{equation}\nonumber
\begin{split}
\int_{ \{ \Psi<-t\}}|F_{t_0,\tilde{c}}|^2_hc(-\Psi)
=&\int_{ \{ \Psi<-t\}}|\tilde{F}|^2_hc(-\Psi)\\
+&\int_{ \{ \Psi<-t\}}|F_{t_0,\tilde{c}}-\tilde{F}|^2_hc(-\Psi)
\end{split}
\end{equation}
for any $t\ge t_0,$ then
\begin{equation}\label{linear 3.13}
\begin{split}
\int_{ \{-t_3\le \Psi<-t_4\}}|F_{t_0,\tilde{c}}|^2_hc(-\Psi)
=&\int_{ \{-t_3\le \Psi<-t_4\}}|\tilde{F}|^2_hc(-\Psi)\\
+&\int_{\{-t_3\le \Psi<-t_4\}}|F_{t_0,\tilde{c}}-\tilde{F}|^2_hc(-\Psi)
\end{split}
\end{equation}
holds for any $t_3>t_4\ge t_0$. It follows from the dominated convergence theorem, equality \eqref{linear 3.13}, \eqref{linear 3.5} and $c(t)>0$ for any $t>T$, that
\begin{equation}\label{linear 3.14}
\begin{split}
\int_{ \{z\in M:-\Psi(z)=t\}}|F_{t_0,\tilde{c}}|^2_h
=\int_{\{z\in M:-\Psi(z)=t\}}|F_{t_0,\tilde{c}}-\tilde{F}|^2_h
\end{split}
\end{equation}
holds for any $t>t_0$.

Choosing any closed interval $[t'_4,t'_3]\subset (t_0,+\infty)\subset (T,+\infty)$. Note that $c(t)$ is uniformly continuous and have positive lower bound and upper bound on $[t'_4,t'_3]\backslash U_k$, where $\{U_k\}$ is the decreasing sequence of open subsets of $(T,+\infty)$, such that $c$ is continuous on $(T,+\infty)\backslash U_k$ and $\lim\limits_{k \to +\infty}\mu(U_k)=0$. Take $N=\cap_{k=1}^{+\infty}U_k.$ Note that
\begin{equation}\label{linear 3.15}
\begin{split}
&\int_{ \{-t'_3\le\Psi<-t'_4\}}|F_{t_0,\tilde{c}}|^2_h\\
=&\lim_{n\to+\infty}\sum_{i=0}^{n-1}\int_{\{z\in M:-\Psi(z)\in S_{i,n}\backslash U_k\}}|F_{t_0,\tilde{c}}|^2_h
+\int_{\{z\in M:-\Psi(z)\in(t'_4,t'_3]\cap U_k\}}|F_{t_0,\tilde{c}}|^2_h\\
\le&\limsup_{n\to+\infty}\sum_{i=0}^{n-1}\frac{1}{\inf_{S_{i,n}}c(t)}\int_{\{z\in M:-\Psi(z)\in S_{i,n}\backslash U_k\}}|F_{t_0,\tilde{c}}|^2_hc(-\Psi)\\
&+\int_{\{z\in M:-\Psi(z)\in(t'_4,t'_3]\cap U_k\}}|F_{t_0,\tilde{c}}|^2_h,
\end{split}
\end{equation}
where $S_{i,n}=(t'_4-(i+1)\alpha_n,t'_3-i\alpha_n]$ and $\alpha_n=\frac{t'_3-t'_4}{n}$.
It follows from equality \eqref{linear 3.13},\eqref{linear 3.14}, \eqref{linear 3.5} and the dominated convergence theorem that
\begin{equation}\label{linear 3.16}
\begin{split}
&\int_{\{z\in M:-\Psi(z)\in S_{i,n}\backslash U_k\}}|F_{t_0,\tilde{c}}|^2_hc(-\Psi)\\
=&\int_{\{z\in M:-\Psi(z)\in S_{i,n}\backslash U_k\}}|\tilde F|^2_hc(-\Psi)
+\int_{\{z\in M:-\Psi(z)\in S_{i,n}\backslash U_k\}}|F_{t_0,\tilde{c}}-\tilde F|^2_hc(-\Psi).
\end{split}
\end{equation}
As $c(t)$ is uniformly continuous and have positive lower bound and upper bound on $[t'_3,t'_4]\backslash U_k$, combining equality \eqref{linear 3.16}, we have
\begin{equation}\label{linear 3.17}
\begin{split}
&\limsup_{n\to+\infty}\sum_{i=0}^{n-1}\frac{1}{\inf_{S_{i,n}\backslash U_k}c(t)}\int_{\{z\in M:-\Psi(z)\in S_{i,n}\backslash U_k\}}|F_{t_0,\tilde{c}}|^2_hc(-\Psi)\\
=&\limsup_{n\to+\infty}\sum_{i=0}^{n-1}\frac{1}{\inf_{S_{i,n}\backslash U_k}c(t)}(\int_{\{z\in M:-\Psi(z)\in S_{i,n}\backslash U_k\}}|\tilde F|^2_hc(-\Psi)\\
&+\int_{\{z\in M:-\Psi(z)\in S_{i,n}\backslash U_k\}}|F_{t_0,\tilde{c}}-\tilde F|^2_hc(-\Psi))\\
\le & \limsup_{n\to+\infty}\sum_{i=0}^{n-1}\frac{\sup_{S_{i,n}\backslash U_k}c(t)}{\inf_{S_{i,n}\backslash U_k}c(t)}(\int_{\{z\in M:-\Psi(z)\in S_{i,n}\backslash U_k\}}|\tilde F|^2_h\\
&+\int_{\{z\in M:-\Psi(z)\in S_{i,n}\backslash U_k\}}|F_{t_0,\tilde{c}}-\tilde F|^2_h)\\
=&\int_{\{z\in M:-\Psi(z)\in (t'_4,t'_3]\backslash U_k\}}|\tilde F|^2_h
+\int_{\{z\in M:-\Psi(z)\in (t'_4,t'_3]\backslash U_k\}}|F_{t_0,\tilde{c}}-\tilde F|^2_h.
\end{split}
\end{equation}
If follows from inequality \eqref{linear 3.15} and \eqref{linear 3.17} that
\begin{equation}\label{linear 3.18}
\begin{split}
&\int_{ \{-t'_3\le\Psi<-t'_4\}}|F_{t_0,\tilde{c}}|^2_h\\
\le & \int_{\{z\in M:-\Psi(z)\in (t'_4,t'_3]\backslash U_k\}}|\tilde F|^2_h
+\int_{\{z\in M:-\Psi(z)\in (t'_4,t'_3]\backslash U_k\}}|F_{t_0,\tilde{c}}-\tilde F|^2_h\\
&+\int_{ \{z\in M: -\Psi(z)\in(t'_4,t'_3]\cap U_k\}}|F_{t_0,\tilde{c}}|^2_h.
\end{split}
\end{equation}
It follows from $F_{t_0,\tilde{c}}\in\mathcal{H}^2(c,t_0)$ that $\int_{ \{-t'_3\le\Psi<-t'_4\}}|F_{t_0,\tilde{c}}|^2_h<+\infty$. Let $k\to+\infty,$ by equality \eqref{linear 3.5}, inequality \eqref{linear 3.18} and the dominated theorem, we have
\begin{equation}\label{linear 3.19}
\begin{split}
&\int_{ \{-t'_3\le\Psi<-t'_4\}}|F_{t_0,\tilde{c}}|^2e^{-\varphi}\\
\le & \int_{\{z\in M:-\Psi(z)\in (t'_4,t'_3]\}}|\tilde F|^2_h
+\int_{\{z\in M:-\Psi(z)\in (t'_4,t'_3]\backslash N\}}|F_{t_0,\tilde{c}}-\tilde F|^2_h\\
&+\int_{ \{z\in M: -\Psi(z)\in(t'_4,t'_3]\cap N\}}|F_{t_0,\tilde{c}}|^2_h.
\end{split}
\end{equation}
By similar discussion, we also have that
\begin{equation}\nonumber
\begin{split}
&\int_{ \{-t'_3\le\Psi<-t'_4\}}|F_{t_0,\tilde{c}}|^2_h\\
\ge & \int_{\{z\in M:-\Psi(z)\in (t'_4,t'_3]\}}|\tilde F|^2_h
+\int_{\{z\in M:-\Psi(z)\in (t'_4,t'_3]\backslash N\}}|F_{t_0,\tilde{c}}-\tilde F|^2_h\\
&+\int_{ \{z\in M: -\Psi(z)\in(t'_4,t'_3]\cap N\}}|F_{t_0,\tilde{c}}|^2_h.
\end{split}
\end{equation}
then combining inequality \eqref{linear 3.19}, we have
\begin{equation}\label{linear 3.20}
\begin{split}
&\int_{ \{-t'_3\le\Psi<-t'_4\}}|F_{t_0,\tilde{c}}|^2_h\\
= & \int_{\{z\in M:-\Psi(z)\in (t'_4,t'_3]\}}|\tilde F|^2_h
+\int_{\{z\in M:-\Psi(z)\in (t'_4,t'_3]\backslash N\}}|F_{t_0,\tilde{c}}-\tilde F|^2_h\\
&+\int_{ \{z\in M: -\Psi(z)\in(t'_4,t'_3]\cap N\}}|F_{t_0,\tilde{c}}|^2_h.
\end{split}
\end{equation}
Using equality \eqref{linear 3.5}, \eqref{linear 3.14} and Levi's Theorem, we have
\begin{equation}\label{linear 3.21}
\begin{split}
&\int_{ \{z\in M:-\Psi(z)\in U\}}|F_{t_0,\tilde{c}}|^2_h\\
= & \int_{\{z\in M:-\Psi(z)\in U\}}|\tilde F|^2_h
+\int_{\{z\in M:-\Psi(z)\in U\backslash N\}}|F_{t_0,\tilde{c}}-\tilde F|^2_h\\
&+\int_{ \{z\in M: -\Psi(z)\in U\cap N\}}|F_{t_0,\tilde{c}}|^2_h
\end{split}
\end{equation}
holds for any open set $U\subset\subset (t_0,+\infty)$, and
\begin{equation}\label{linear 3.22}
\begin{split}
&\int_{ \{z\in M:-\Psi(z)\in V\}}|F_{t_0,\tilde{c}}|^2_h\\
= & \int_{\{z\in M:-\Psi(z)\in V\}}|\tilde F|^2_h
+\int_{\{z\in M:-\Psi(z)\in V\backslash N\}}|F_{t_0,\tilde{c}}-\tilde F|^2_h\\
&+\int_{ \{z\in M: -\Psi(z)\in V\cap N\}}|F_{t_0,\tilde{c}}|^2_h
\end{split}
\end{equation}
holds for any compact set $V\subset (t_0,+\infty)$. For any measurable set $E\subset\subset (t_0,+\infty)$, there exists a sequence of compact set $\{V_l\}$, such that $V_l\subset V_{l+1}\subset E$ for any $l$ and $\lim\limits_{l\to +\infty}\mu(V_l)=\mu(E)$, hence by equality \eqref{linear 3.22}, we have
\begin{equation}\label{linear 3.23}
\begin{split}
\int_{ \{\Psi<-t_0\}}|F_{t_0,\tilde{c}}|^2_h\mathbb{I}_E(-\Psi)
\ge&\lim_{l \to +\infty} \int_{ \{\Psi<-t_0\}}|F_{t_0,\tilde{c}}|^2_h\mathbb{I}_{V_j}(-\Psi)\\
\ge&\lim_{l \to +\infty} \int_{ \{\Psi<-t_0\}}|\tilde F|^2_h\mathbb{I}_{V_j}(-\Psi)\\
=& \int_{ \{\Psi<-t_0\}}|\tilde F|^2_h\mathbb{I}_{V_j}(-\Psi).
\end{split}
\end{equation}
It is clear that for any $t>t_0$, there exists a sequence of functions $\{\sum_{j=1}^{n_i}\mathbb{I}_{E_{i,j}}\}_{i=1}^{+\infty}$ defined on $(t,+\infty)$, satisfying $E_{i,j}\subset\subset (t,+\infty)$, $\sum_{j=1}^{n_{i+1}}\mathbb{I}_{E_{i+1,j}}(s)\ge \sum_{j=1}^{n_{i}}\mathbb{I}_{E_{i,j}}(s)$ and $\lim\limits_{i\to+\infty}\sum_{j=1}^{n_i}\mathbb{I}_{E_{i,j}}(s)=\tilde{c}(s)$ for any $s>t$. Combining Levi's Theorem and inequality \eqref{linear 3.23}, we have
\begin{equation}\label{linear 3.24}
\begin{split}
\int_{ \{\Psi<-t_0\}}|F_{t_0,\tilde{c}}|^2_h\tilde{c}(-\Psi)
\ge\int_{ \{\Psi<-t_0\}}|\tilde F|^2_h\tilde{c}(-\Psi).
\end{split}
\end{equation}
By the definition of $G(t_0,\tilde{c})$, we have $G(t_0,\tilde{c})=\int_{ \{\Psi<-t_0\}}|\tilde F|^2_h\tilde{c}(-\Psi).$ Equality \eqref{other c also linear} is proved.

\end{proof}
\section{Proofs of Theorem \ref{p:DK} and Corollary \ref{thm:soc}}

In this section, we prove Theorem \ref{p:DK} and Corollary \ref{thm:soc}.

\subsection{Proof of Theorem \ref{p:DK}}

	Lemma \ref{l:m5} tells us that there exists $p_0>2a_{z_0}^f(\Psi;h)$ such that $I(h,p_0\Psi)_{z_0}=I_+(h,2a_{z_0}^f(\Psi;h)\Psi)_{z_0}$. Following from the definition of $a_{z_0}^{f}(\Psi;h)$ and Lemma \ref{characterization of g(t)=0}, we obtain that
\begin{equation}
	\label{eq:0222d}G(0;c\equiv1,\Psi,h,I_+(h,2a_{z_0}^f(\Psi;h)\Psi)_{z_0},f)>0.
\end{equation}
Without loss of generality, assume that there exists $t>t_0$ such that $\int_{\{\Psi<-t\}}|f|^2_h<+\infty$.
Denote that $t_1:=\inf\{t\ge t_0:\int_{\{\Psi<-t\}}|f|^2_h<+\infty\}.$
 Denote
\begin{displaymath}\begin{split}
 	\inf\Bigg\{\int_{\{p\Psi<-t\}}|\tilde f|^2_h:\tilde f\in H^0(\{p\Psi<-t\},&\mathcal{O}(K_M\otimes E))
 	\\ &\&\,(\tilde f-f)_{z_0}\in \mathcal{O}(K_M)_{z_0}\otimes I(h,p\Psi)_{z_0}\Bigg\}
 	\end{split}
 \end{displaymath}
by $G_{p}(t)$, where $t\in[0,+\infty)$ and $p>2a_{z_0}^f(\Psi;h)$. Then we know that $G_{p}(0)\ge G(0;c\equiv1,\Psi,h,I_+(h,2a_{z_0}^f(\Psi;h)\Psi)_{z_0},f)$ for any $p>2a_{z_0}^f(\Psi;h)$.
Note that
$$p\Psi=\min\{p\psi+(2\lceil p\rceil-2p)\log|F|-2\log|F^{\lceil p\rceil}|,0\},$$
where $\lceil p\rceil=\min\{n\in \mathbb{Z}:n\geq p\}$,
 and
 $$G_{p}(pt)\le\int_{\{\Psi<-t\}}|f|^2_h<+\infty$$
 for any $t>t_1$. Note that $\Theta_{\tilde h}(E)\ge_{Nak}^s 0$ and $(p-2a_{z_0}^f(\Psi;h))\psi+(2\lceil p\rceil-2p)\log|F|$ is plurisubharmonic on $M$. Remark \ref{example of singular metric} implies that $\Theta_{he^{-(p\psi+(2\lceil p\rceil-2p)\log|F|)}}(E)\ge _{Nak}^s0$. Note that $h$ has a positive locally lower  bound.
  Theorem \ref{main theorem} tells us  that $G_{p}(-\log r)$ is concave with respect to $r\in(0,1]$ and $\lim_{t\rightarrow+\infty}G_{p}(t)=0$, which implies that
 \begin{equation}
 	\label{eq:0222b}
 	\begin{split}
 	\frac{1}{r_1^2}\int_{\{p\Psi<2\log r_1\}}|f|^2_h&\ge \frac{1}{r_1^2}G_{p}(-2\log r_1)\\
 	&\ge  G_{p}(0)\\
 	&\ge G(0;c\equiv1,\Psi,h,I_+(h,2a_{z_0}^f(\Psi;h)\Psi)_{z_0},f),
 	\end{split}
 \end{equation}
where  $0<r_1\le e^{-\frac{pt_0}{2}}$.

We prove $a_{z_0}^f(\Psi;h)>0$ by contradiction: if $a_{z_0}^f(\Psi;h)=0$, as $\int_{\{\Psi<-t_1-1\}}|f|^2_h<+\infty$, it follows from the dominated convergence theorem and inequality \eqref{eq:0222b} that
\begin{equation}
	\label{eq:0222c}\begin{split}
	\frac{1}{r_1^2}\int_{\{\Psi=-\infty\}}|f|^2_h&=\lim_{p\rightarrow0+0}\frac{1}{r_1^2}\int_{\{p\Psi<2\log r_1\}}|f|^2_h\\
	&\ge  G(0;c\equiv1,\Psi,h,I_+(h,2a_{z_0}^f(\Psi;h)\Psi)_{z_0},f).\end{split}
\end{equation}
Note that $\mu(\{\Psi=-\infty\})=\mu(\{\psi=-\infty\})=0$, where $\mu$ is the Lebesgue measure on $M$. Inequality \eqref{eq:0222c} implies that $G(0;c\equiv1,\Psi,h,I_+(h,2a_{z_0}^f(\Psi;h)\Psi)_{z_0},f)=0$, which contradicts inequality \eqref{eq:0222d}. Thus, we get that $a_{z_0}^f(\Psi;h)>0$.

For any $r_2\in (0,e^{-a_{z_0}^{f}(\Psi;h)t_1})$, note that $\frac{2\log r_2}{p}<-t_1$ for any $p\in(2a_0^f(\Psi;h),-\frac{2\log r_2}{t_1})$, which implies that
$\int_{\{p\Psi<2\log r_2\}}|f|^2_h<+\infty$ for any $p\in(2a_0^f(\Psi;h),-\frac{2\log r_2}{t_1}).$
 Then it follows from the dominated convergence theorem and inequality \eqref{eq:0222b} that
\begin{equation}
	\label{eq:0222e}\begin{split}
	\frac{1}{r_2^2}\int_{\{2a_0^f(\Psi;h)\Psi\le2\log r_2\}}|f|^2_h&=\lim_{p\rightarrow2a_0^f(\Psi;h)+0}\frac{1}{r_2^2}\int_{\{p\Psi<2\log r_2\}}|f|^2_h\\
	&\ge G(0;c\equiv1,\Psi,h,I_+(h,2a_{z_0}^f(\Psi;h)\Psi)_{z_0},f).	
	\end{split}
\end{equation}
For any $r\in(0,e^{-a_{z_0}^f(\Psi;h)t_0}]$, if $r>e^{-a_{z_0}^{f}(\Psi;h)t_1}$, we have $\int_{\{a_0^f(\Psi;h)\Psi<\log r\}}|f|^2_h=+\infty>G(0;c\equiv1,\Psi,h,I_+(h,2a_{z_0}^f(\Psi;h)\Psi)_{z_0},f)$, and if $r\in(0,e^{-a_{z_0}^f(\Psi;h)t_1}]$, it follows from  $\{a_{z_0}^f(\Psi;h)\Psi<\log r\}=\cup_{0<r_2<r}\{a_{z_0}^f(\Psi;h)\Psi<\log r_2\}$ and inequality \eqref{eq:0222e}  that
\begin{displaymath}
	\begin{split}
\int_{\{a_0^f(\Psi;h)\Psi<\log r\}}|f|^2_h&=
\sup_{r_2\in(0,r)}\int_{\{2a_0^f(\Psi;h)\Psi\le2\log r_2\}}|f|^2_h
\\&\ge \sup_{r_2\in(0,r)}r_2^2G(0;c\equiv1,\Psi,h,I_+(h,2a_{z_0}^f(\Psi;h)\Psi)_{z_0},f)\\
&=r^2G(0;c\equiv1,\Psi,h,I_+(h,2a_{z_0}^f(\Psi;h)\Psi)_{z_0},f).
	\end{split}
\end{displaymath}

Thus, Theorem \ref{p:DK} holds.

\subsection{Proof of Corollary \ref{thm:soc}}

	It is clear that $I_+(h,a\Psi)_{z_0}\subset I(h,a\Psi)_{z_0}$, hence it suffices to prove that   $I(h,a\Psi)_{z_0}\subset I_+(h,a\Psi)_{z_0}$.

If there exists $f_{z_0}\in I(h,a\Psi)_{z_0}$ such that $f_{z_0}\not\in I_+(h,a\Psi)_{z_0}$, then $a_{z_0}^f(\Psi;h)_{z_0}=\frac{a}{2}<+\infty$.  Theorem \ref{p:DK} shows that $a>0$. Without loss of generality, assume that $M=D$ is a domain in $\mathbb{C}^n$ and  $f\in H^0(\{\Psi<-t_0\}\cap D,\mathcal{O}(E))$, where $t_0>0$. For any neighborhood $U\subset D$ of $z_0$, it follows from Proposition \ref{p:DK} that  there exists $C_U>0$ such that
\begin{equation}
	\label{eq:0222f}\frac{1}{r^2}\int_{\{a\Psi<2\log r\}\cap U}|f|^2_h\ge C_U
\end{equation}
for any $r\in(0,e^{-\frac{at_0}{2}}]$. For any $t>at_0$, it follows from Fubini's Theorem and inequality \eqref{eq:0222f} that
\begin{displaymath}
	\begin{split}
		\int_{\{a\Psi<-t\}\cap U}|f|^2_he^{-a\Psi}
		=&\int_{\{a\Psi<-t\}\cap U}\left(|f|^2_h\int_0^{e^{-a\Psi}}dl\right)\\
		=&\int_0^{+\infty}\left(\int_{\{l<e^{-a\Psi}\}\cap\{a\Psi<-t\}\cap U}|f|^2_h\right)dl\\
		\ge&\int_{e^t}^{+\infty}\left(\int_{\{a\Psi<-\log l\}\cap U}|f|^2_h\right)dl\\
		\ge&C_U\int_{e^t}^{+\infty}\frac{1}{l}dl\\
		=&+\infty,
	\end{split}
\end{displaymath}
which contradicts  $f_{z_0}\in I(h,a\Psi)_{z_0}$. Thus, we have $I(h,a\Psi)_{z_0}\backslash I_+(h,a\Psi)_{z_0}=\emptyset$ for any $a\ge 0$, which shows that $I(h,a\Psi)_{z_0}=I_+(h,a\Psi)_{z_0}$ for any $a\ge0$.

\section{Proof of Theorem \ref{thm:effe}}
In this section, we prove Theorem \ref{thm:effe} by using Theorem \ref{main theorem}.
	
	For any Lebesgue measurable function $c$ on $(0,+\infty)$ and any $q'> 2a_{z_0}^f(\Psi;h)\ge1$, denote that\begin{displaymath}
		\begin{split}
			G_{c,q'}(t):=\inf\Bigg\{\int_{\{q'\Psi<-t\}}|\tilde f|^2_hc(-q'\Psi):&(\tilde f-f)_{z_0}\in\mathcal{O}(K_M)_{z_0}\otimes I(h,q'\Psi)_{z_0}\\
			&\&\,\tilde f\in H^0(\{q'\Psi<-t\},\mathcal{O}(K_M\otimes E)) \Bigg\}.
		\end{split}
	\end{displaymath}
Note that there exist a plurisubharmonic function $\psi_1=q'\psi+(2k-2q')\log|F|$ and a holomorphic function $F_1=F^k$ on $M$ such that
 $$\psi_1-2\log|F_1|= q'(\psi-2\log|F|)$$
 on $M$, where $k>q'$ is a integer. Denote that $\Psi_1:=\min\{\psi_1-2\log|F_1|,0\}=q'\Psi$ on $M$.
		
	Firstly, we prove  inequality
	\begin{equation}
		\label{eq:0307c}\int_{\{\Psi<-\frac{l}{a}\}}|f|^2_he^{-(1-a)\Psi}\ge e^{-\frac{a-1+q'}{a}l}\frac{1}{K_{\Psi,f,h,a}(z_0)}
	\end{equation}
	in two case $a\in(0,1]$ and $a>1$, where  $l\ge0$ and $q'> 2a_{z_0}^f(\Psi;h)$.
	
	We prove inequality \eqref{eq:0307c} for the case $a\in(0,1]$. Let $c_1(t)=e^{\frac{1-a}{q'}t}$ on $(0,+\infty)$, hence $c_1(t)e^{-t}$ is decreasing on $(0,+\infty)$ and $c_1(-q'\Psi)=e^{-(1-a)\Psi}\ge1$ on $M$. Note that $h$ has a positive locally lower bound. As
	$\Theta_{\tilde h}\ge_{Nak}^s0$ and $\psi$ is plurisubharmonic, where $\tilde h=he^{-2a_{z_0}^f(\Psi,h)\psi}$, then we have   $$\Theta_{he^{-\psi_1}}\ge_{Nak}^s0.$$ Theorem \ref{main theorem} (replace $\Psi$ and $c$ by $\Psi_1$ and $c_1$, respectively) shows that $G_{c_1,q'}(h^{-1}(r))$ is concave with respect to $r$, where $h(t)=\int_t^{+\infty}c_1(s)e^{-s}ds$. Note that
	$$G_{c_1,q'}(0)\ge\frac{1}{K_{\Psi,f,h,a}(z_0)}$$ for any $q'> 2a_{z_0}^f(\Psi;h)$. Hence we have
	\begin{equation*}
	\begin{split}
			\int_{\{\Psi<-\frac{l}{a}\}}|f|^2_he^{-(1-a)\Psi}\ge& G_{c_1,q'}\left(\frac{q'l}{a}\right)\\
			\ge&\frac{\int_{\frac{q'l}{a}}^{+\infty}c_1(s)e^{-s}ds}{\int_0^{+\infty}c_1(s)e^{-s}ds}G_{c_1,q'}(0)\\
			\ge&e^{-\frac{a-1+q'}{a}l}\frac{1}{K_{\Psi,f,h,a}(z_0)}.		\end{split}
	\end{equation*}
	
 We prove inequality \eqref{eq:0307c} for the case $a>1$. Take $\tilde c_m(t)=e^{\frac{1-a}{q'}t}$ on $(0,m)$ and $\tilde c_m(t)=e^{\frac{1-a}{q'}m}$ on $(m,+\infty)$, then $\tilde c_{m}(t)$ is a continuous function on $(0,+\infty)$ and $c_1(t)e^{-t}$ is decreasing on $(0,+\infty)$, where $m$ is any positive integer. Note that $c(t)\ge e^{\frac{1-a}{q'}m}$ on $(0,+\infty)$ and $h$ has a positive locally lower bound. Theorem \ref{main theorem} (replace $\Psi$ and $c$ by $\Psi_1$ and $\tilde c_m$, respectively)  shows that $G_{\tilde c_m,q'}(h_m^{-1}(r))$ is concave with respect to $r$, where $h_m(t)=\int_t^{+\infty}\tilde c_m(s)e^{-s}ds$. Note that $$G_{\tilde c_m,q'}(0)\ge\frac{1}{K_{\Psi,f,h,a}(z_0)}$$ for any $q'> 2a_{z_0}^f(\Psi;h)$. Hence we have
	\begin{equation}\label{eq:0307d}
	\begin{split}
			\int_{\{\Psi<-\frac{l}{a}\}}|f|^2_h\tilde c_m(-q'\Psi)\ge& G_{\tilde c_m,q'}\left(\frac{q'l}{a}\right)\\
			\ge&\frac{\int_{\frac{q'l}{a}}^{+\infty}\tilde c_m(s)e^{-s}ds}{\int_0^{+\infty}\tilde c_m(s)e^{-s}ds}G_{\tilde c_m,q'}(0)\\
			\ge&\frac{\int_{\frac{q'l}{a}}^{+\infty}\tilde c_m(s)e^{-s}ds}{\int_0^{+\infty}\tilde c_m(s)e^{-s}ds}\frac{1}{K_{\Psi,f,h,a}(z_0)}.		\end{split}
	\end{equation}
As $\int_{\{\Psi<0\}}|f|^2_he^{-\Psi}\le C_1<+\infty$, it follows from $\tilde c_m(-q'\Psi)\le e^{-\Psi}$, the dominated convergence theorem and inequality \eqref{eq:0307d} that
\begin{displaymath}
	\begin{split}
		\int_{\{\Psi<-\frac{l}{a}\}}|f|^2_he^{-(1-a)\Psi}
		=&\lim_{m\rightarrow+\infty}\int_{\{\Psi<-\frac{l}{a}\}}|f|^2_h\tilde c_m(-q'\Psi)\\
		\ge&\lim_{m\rightarrow+\infty}\frac{\int_{\frac{q'l}{a}}^{+\infty}\tilde c_m(s)e^{-s}ds}{\int_0^{+\infty}\tilde c_m(s)e^{-s}ds}\frac{1}{K_{\Psi,f,h,a}(z_0)}\\
		=&e^{-\frac{a-1+q'}{a}l}\frac{1}{K_{\Psi,f,h,a}(z_0)}.
	\end{split}
\end{displaymath}

Next, we complete the proof. Following from Fubini's Theorem, we have
	\begin{equation*}
\begin{split}
			&\int_{\{\Psi<0\}}|f|^2_he^{-\Psi}\\
			=&\int_{\{\Psi<0\}}\left(|f|^2_he^{-\Psi+a\Psi}\int_0^{e^{-a\Psi}}ds\right)\\
			=&\int_0^{+\infty}\left(\int_{\{\Psi<0\}\cap\{s<e^{-a\Psi}\}}|f|^2_he^{-\Psi+a\Psi} \right)ds\\
			=&\int_{-\infty}^{+\infty}\left(\int_{\{\Psi<-\frac{l}{a}\}\cap\{\Psi<0\}}|f|^2_he^{-\Psi+a\Psi} \right)e^{l}dl\\
			=&\int_{\{\Psi<0\}}|f|^2_he^{-\Psi+a\Psi} +\int_{0}^{+\infty}\left(\int_{\{\Psi<-\frac{l}{a}\}}|f|^2_he^{-\Psi+a\Psi} \right)e^{l}dl.
		\end{split}
	\end{equation*}
Using inequality \eqref{eq:0307c} and the definition of $K_{\Psi,f,h,a}(z_0)$, we obtain that
	\begin{equation}
		\label{eq:0307e}\begin{split}
			&\int_{\{\Psi<0\}}|f|^2_he^{-\Psi}\\
			=&\int_{\{\Psi<0\}}|f|^2_he^{-\Psi+a\Psi} +\int_{0}^{+\infty}\left(\int_{\{\Psi<-\frac{l}{a}\}}|f|^2_he^{-\Psi+a\Psi} \right)e^{l}dl\\
			\ge&\left(1+\int_0^{+\infty}e^{-\frac{-1+q'}{a}l}dl\right)\frac{1}{K_{\Psi,f,h,a}(z_0)}\\
			=&\frac{a+q'-1}{q'-1}\cdot\frac{1}{K_{\Psi,f,h,a}(z_0)}
		\end{split}
	\end{equation}
	for any $q'>2a_{z_0}^f(\Psi;h)$.
	Let $q'\rightarrow2a_{z_0}^f(\Psi;h)$, we get that inequality \eqref{eq:0307e} also holds when $q'=2a_{z_0}^f(\Psi;h)$. Thus, if $q>1$ satisfies
	$$\frac{q+a-1}{q-1}>\frac{C_1}{C_2}\geq K_{\Psi,f,h,a}(z_0)\int_{\{\Psi<0\}}|f|^2_he^{-\Psi},$$
	we have $p<2a_{z_0}^f(\Psi;h)$, i.e. $f_{z_0}\in \mathcal{O}(K_M)_{z_0}\otimes I(h,p\Psi)_{z_0}$.

\section{Proof of Theorem \ref{p:soc-twist}}

In this section, we prove Theorem \ref{p:soc-twist} by using Remark \ref{infty2} and Theorem \ref{p:DK}.
Firstly, we recall two basic lemmas, which will be used in the proof of Theorem \ref{p:soc-twist}.

\begin{Lemma}[see \cite{GY-twisted}]\label{l:2}
Let $a(t)$ be a positive measurable function on $(-\infty,+\infty)$,
such that $a(t)e^{t}$ is increasing near $+\infty$,
and $a(t)$ is not integrable near $+\infty$.
Then there exists a positive measurable function $\tilde a(t)$ on $(-\infty,+\infty)$ satisfying the following statements:
		
		$(1)$ there exists $T<+\infty$ such that $\tilde{a}(t)\leq a(t)$ for any $t>T$;
		
		$(2)$  $\tilde a(t)e^{t}$ is strictly increasing and continuous near $+\infty$;
		
		$(3)$ $\tilde a(t)$ is not integrable near $+\infty$.
\end{Lemma}
\begin{Lemma}
	[see \cite{GZ-soc17}]\label{l:m} For any two measurable spaces $(X_i,\mu_i)$ and two measurable functions $g_i$ on $X_i$ respectively ($i\in\{1,2\}$), if $\mu_1(\{g_1\geq s^{-1}\})\geq\mu_2(\{g_2\geq s^{-1}\})$ for any $s\in(0,s_0]$, then $\int_{\{g_1\geq s_0^{-1}\}}g_1d\mu_1\geq\int_{\{g_2\geq s_0^{-1}\}}g_2d\mu_2$.
\end{Lemma}

\begin{proof}[Proof of Theorem \ref{p:soc-twist}]
	We prove Theorem \ref{p:soc-twist} in two cases, that $a(t)$ satisfies condition $(1)$ or condition $(2)$.

\

\emph{Case $(1)$. $a(t)$ is decreasing near $+\infty$.}

\

Firstly, we prove $(B)\Rightarrow(A)$.
Consider $F\equiv 1$, $f=(f_1,f_2,\ldots,f_r)=(1,1,\ldots,1)$, $h\equiv1$  and $\psi=\log|z_1|$ on the unit polydisc $\Delta^n\subset\mathbb{C}^n$. Note that $a_{o}^f(\log|z_1|;h)=1$ and
\begin{displaymath}
	\begin{split}	&\int_{\Delta_{s_0}^n}|f_h^2e^{-2a_{o}^f(\log|z_1|;h)\Psi}a(-2a_{o}^f(\log|z_1|;h)\Psi)\\
	=&
			r\int_{\Delta_{s_0}^n}a(-2\log|z_1|)\frac{1}{|z_1|^2}\\
			=&r(\pi s_0^2)^{n-1}\int_{\Delta_{s_0}}a(-2\log|z_1|)\frac{1}{|z_1|^2}\\
		=&r(\pi s_0^2)^{n-1}2\pi\int_0^{s_0}a(-2\log r)r^{-1}dr\\
		=&r(\pi s_0^2)^{n-1}\pi\int_{-2\log{s_0}}^{+\infty}a(t)dt
	\end{split}
\end{displaymath}
for $s_0\in(0,1)$,
hence we obtain $(B)\Rightarrow (A)$.

Then, we prove $(A)\Rightarrow(B)$.
Corollary \ref{thm:soc} shows that  $f_o\not\in I(h,2a_o^f(\Psi;h)\Psi)_o$ and $a_o^f(\Psi;h)>0$.
Now we assume that there exist $t_0>0$ and a pseudoconvex domain $D_0\subset D$ containing $o$ such that  $\int_{\{\Psi<-t_0\}\cap D_0}|f|_h^2e^{-2a_o^f(\Psi;h)\Psi}a(-2a_o^f(\Psi;h)\Psi)<+\infty$  to get a contradiction. As $f_o\in I(h,0\Psi)_o$, there exist $t_1>t_0$ and
a pseudoconvex domain $D_1\subset D_0$ containing $o$ such that
$\int_{D_1\cap\{\Psi<-t_1\}}|f|^2_h<+\infty$.
Set $c(t)=a(t)e^{t}+1$, then we have
\begin{equation}
	\label{eq:0305a}
	\int_{D_1\cap\{\Psi<-t_1\}}|f|^2_hc(-2a_o^f(\Psi;h)\Psi)<+\infty.
\end{equation}
Without loss of generality, assume that $a(t)$ is decreasing on $(2a_o^f(\Psi;h)t_1,+\infty)$. Note that
 $c(t)e^{-t}=a(t)+e^{-t}$ is decreasing on $(2a_o^f(\Psi;h)t_1,+\infty)$ and $\liminf_{t\rightarrow+\infty}c(t)>0$. As $a(t)$ is not integrable near $+\infty$, so is $c(t)e^{-t}$. Note that there exist a plurisubharmonic function $\psi_1=2a_o^f(\Psi;h)\psi+2(k-a_o^f(\Psi;h))\log|F|$ and a holomorphic function $F_1=F^k$ on $D_1$ such that
 $$\psi_1-2\log|F_1|= 2a_o^f(\Psi;h)(\psi-2\log|F|)$$
 on $D_1$, where $k>a_o^f(\Psi;h)$ is a integer. Denote that $\Psi_1:=\min\{\psi_1-2\log|F_1|,-2a_o^f(\Psi;h)t_1\}$ on $D_1$. Note that $h$ has a positive locally lower bound, $c(t)\ge1$ on $(2a_o^f(\Psi;h)t_1,+\infty)$ and $\Theta_{he^{-\psi_1}}\ge_{Nak}^s0$.
  Using Remark \ref{infty2} (replacing $M$, $\Psi$ and $T$ by $D_1$, $\Psi_1$ and $2a_o^f(\Psi;h)t_1$ respectively), as $f_o\not\in I(h,2a_o^f(\Psi;h)\Psi)_o=I(h,\Psi_1)_o$, then we have $G(2a_o^f(\Psi;h)t_1;c,\Psi_1,h,I(\Psi_1+\varphi)_o,f)=+\infty$,
which contradicts to inequality \eqref{eq:0305a}. Thus, we obtain $(A)\Rightarrow (B)$.

\

\emph{Case $(2)$. $a(t)e^{t}$ is increasing near $+\infty$.}

\

In this case, the proof of $(B)\Rightarrow(A)$ is the same as  the case $(1)$, hence it suffices to prove $(A)\Rightarrow(B)$.

Assume that statement $(A)$ holds.
 Lemma \ref{l:2} shows that there exists a positive function $\tilde a(t)$ on $(-\infty,+\infty)$ satisfying that: $\tilde a(t)\leq a(t)$ near $+\infty$; $\tilde a(t)e^{t}$ is strictly increasing and continuous near $+\infty$; $\tilde a(t)$ is not integrable near $+\infty$. Thus, it suffices to prove that for any $\Psi$, $h$ and $f_o\in I(h,0\Psi)_o$ satisfying $a_{o}^{f}(\Psi;h)<+\infty$, $|f|^2_he^{-2a_o^{f}(\Psi;h)\Psi}a(-2a_o^{f}(\Psi;h)\Psi)\not\in L^1(U\cap\{\Psi<-t\})$ for any neighborhood $U$  of $o$ and any  $t>0$.

Take any   $t_0\gg0$ and any  pseudoconvex domain $D_0\subset D$ containing the origin $o$ such that $f\in\mathcal{O}(D_0\cap\{\Psi<-t_0\})$.
 Let $\mu_1(X)=\int_{X}|f|^2_h$, where $X$ is a Lebesgue measurable subset of $D_0\cap\{\Psi<-t_0\}$,  and let $\mu_2$ be the Lebeague measure on $(0,1]$. Denote that $Y_s=\{-2a_{o}^f(\Psi;h)\Psi\geq -\log s\}$. Theorem \ref{p:DK} shows that there exists a positive constant $C$ such that $\mu_1(Y_s)\geq Cs$ holds for any $s\in(0,e^{-2a_o^f(\Psi;h)t_0}]$.

Let $g_1=\tilde a(-2a_o^{f}(\Psi;h)\Psi)\exp(-2a_o^{f}(\Psi;h)\Psi)$ and $g_2(x)=\tilde a(-\log x+\log C)Cx^{-1}$. As $\tilde a(t)e^{-t}$ is strictly increasing near $+\infty$, then $g_1\geq \tilde a(-\log s)s^{-1}$ on $Y_s$ implies that
\begin{equation}
	\label{eq:210820h}
	\mu_1(\{g_1\geq \tilde a(-\log s)s^{-1}\})\geq\mu_1(Y_s)\geq Cs
\end{equation}
holds for any  $s>0$ small enough.
As $\tilde a(t)e^{t}$ is strictly increasing near $+\infty$, then there exists $s_0\in(0,e^{-2a_o^f(\Psi)t_0})$ such that
\begin{equation}
	\label{eq:210820i}
	\mu_2(\{x\in(0,s_0]:g_2(x)\geq \tilde a(-\log s)s^{-1}\})=\mu_2(\{0<x\leq Cs\})=Cs
\end{equation}
for any $s\in(0,s_0]$.
As  $\tilde a(-\log s)s^{-1}$ converges to $+\infty$ (when $s\rightarrow0+0$) and $\tilde a(t)$ is continuous near $+\infty$, we obtain that
$$\mu_1(\{g_1\geq s^{-1}\})\geq\mu_2(\{x\in(0,s_0]:g_2(x)\geq s^{-1}\})$$
holds for any $s>0$ small enough. Following from Lemma \ref{l:m} and $\tilde a(t)$ is not integrable near $+\infty$, we obtain $|f|^2_he^{-2a_o^{f}(\Psi;h)\Psi}a(-2a_o^{f}(\Psi;h)\Psi)\not\in L^1(U\cap\{\Psi<-t\})$.

Thus, Theorem \ref{p:soc-twist} holds.
\end{proof}

\section{Proof of Proposition \ref{p:1}}

In this section, we prove Proposition \ref{p:1} by using Theorem \ref{main theorem}.

\begin{proof}
Let
\begin{displaymath}
	h(x)=\left\{ \begin{array}{cc}
	e^{-\frac{1}{1-(x-1)^2}} & \textrm{if $|x-1|<1$}\\
	0 & \textrm{if $|x-1|\geq1$}
    \end{array} \right.
\end{displaymath}
be a real function defined on $\mathbb R$, and let $g_n(x)=\frac{n}{(n+1)d}\int_0^{nx}h(s)ds$, where $d=\int_{\mathbb R}h(s)ds$.
    Note that $h(x)\in C_0^\infty(\mathbb R)$ and $h(x)\geq0$ for any $x\in \mathbb R$. Then  we get that $g_n(x)$ is increasing with respect to $x$, $g_n(x)\leq g_{n+1}(x)$ for any $n\in \mathbb{N}$ and $x\in \mathbb R$, and $\lim_{n\rightarrow+\infty}g_n(x)=\mathbb I_{\{s\in \mathbb R: s>0\}}(x)$ for any $x\in \mathbb R$.
Setting $c_t^n(x)=1-g_n(x-t)$, where $t$ is the given positive number in Proposition \ref{p:1}, it follows from the properties of $\{g_n(x)\}_{n\in \mathbb{N}}$ that $c_t^n(x)$ is decreasing with respect to $x$, $c_t^n(x)\geq c_t^{n+1}(x)$ for any $n\in \mathbb{N}$ and $x\in \mathbb R$, and $\lim_{n\rightarrow+\infty}c_t^n(x)=\mathbb I_{\{s\in \mathbb R: s\leq t\}}(x)$ for any $x\in \mathbb R$.

Denote \begin{equation*}
\begin{split}
\inf\Bigg\{\int_{\{ \Psi<-t\}}|\tilde{f}|^2_hc_t^n(-\Psi): &\tilde{f}\in
H^0(\{\Psi<-t\},\mathcal{O} (K_M\otimes E)  ) \\
&\&\, (\tilde{f}-f)_{z_0}\in
\mathcal{O} (K_M)_{z_0} \otimes I(h,\Psi)_{z_0}\,\text{for any }  z_0\in Z_0\Bigg\}
\end{split}
\end{equation*}
 by $G_{t,n}(s)$. Note that $\Theta_{he^{-\psi}}\ge_{Nak}^s0$, $h$ has a positive locally lower bound and $$c_t^n(x)\in[\frac{1}{n+1},1]$$
 on $(0,+\infty)$.
    By using Theorem \ref{main theorem}, we have
    \begin{equation}
    	\label{eq:proof1}
    	\int_{\{\Psi<-l\}}|f|^2_h
c_t^n(-\Psi)\geq G_{t,n}(l)\geq \frac{\int_l^{+\infty}c_t^n(s)e^{-s}ds}{\int_0^{+\infty}c_t^n(s)e^{-s}ds}G_{t,n}(0)
	    \end{equation}
	    for any $l>0$.
	    Following from $\int_{\{\Psi<-l\}}|f|_h^2<+\infty$ for any $l>0$, the properties of $\{c_t^n\}_{n\in \mathbb{N}}$ and the dominated convergence theorem, we obtain that
	    \begin{equation}
	    	\label{eq:proof2}
	    	\lim_{n\rightarrow+\infty}\int_{\{\Psi<-l\}}|f|_h^2
c_t^n(-\Psi)=\int_{\{-t\leq\Psi<-l\}}|f|_h^2.	    \end{equation}
As $c_t^n(x)\geq\mathbb I_{\{s\in\mathbb R:s\leq t\}}(x)$ for any $x>0$ and $n\in \mathbb{N}$, then it follows from the definitions of $G_{t,n}(0)$ and $C_{\Psi,f,h,t}(Z_0)$ that
\begin{equation}
	\label{eq:proof3}
	G_{t,n}(0)\geq C_{\Psi,f,h,t}(Z_0).
\end{equation}
Combining inequality \eqref{eq:proof1}, equality \eqref{eq:proof2}, and inequality \eqref{eq:proof3}, we obtain that
\begin{displaymath}
	\begin{split}
		\int_{\{-t\leq\Psi<-l\}}|f|^2_h &=	\lim_{n\rightarrow+\infty}\int_{\{\Psi<-l\}}|f|^2_h
c_t^n(-\Psi)\\
&\geq \lim_{n\rightarrow+\infty} \frac{\int_l^{+\infty}c_t^n(s)e^{-s}ds}{\int_0^{+\infty}c_t^n(s)e^{-s}ds}C_{\Psi,f,h,t}(Z_0)\\
&=\frac{e^{-l}-e^{-t}}{1-e^{-t}}C_{\Psi,f,h,t}(Z_0)
	\end{split}	
\end{displaymath}
for any $l\in (0,t)$. Following from the definition of $C_{\Psi,f,h,t}(Z_0)$, we have $\int_{\{-t\leq\Psi<0\}}|f|^2_h\ge C_{\Psi,f,h,t}(Z_0)$. Thus, we have
    \begin{equation}
    \label{eq:proof4}
    	\int_{\{-t\leq\Psi<-l\}}|f|^2_h\geq\frac{e^{-l}-e^{-t}}{1-e^{-t}}C_{\Psi,f,h,t}(Z_0)
    	 \end{equation}
    for any $l\in [0,t)$.
Following from Fubini's Theorem and inequality \eqref{eq:proof4}, we obtain that
\begin{displaymath}
	\begin{split}
		\int_{M_t}|f|^2_h
e^{-\Psi}&=\int_{M_t}\left(|f|^2_h
\int_0^{e^{-\Psi}}dr\right)\\
&=\int_0^{+\infty}\left(\int_{M_t\cap\{r<e^{-\Psi}\}}|f|^2_h \right)dr\\
&=\int_{-\infty}^t\left(\int_{\{-t\leq\Psi<\min\{-l,0\}\}}|f|^2_h\right)e^ldl\\
&=\int_{-\infty}^0\left(\int_{\{-t\leq\Psi<\min\{-l,0\}\}}|f|^2_h\right)e^ldl+\int_{0}^t\left(\int_{\{-t\leq\Psi<-l\}}|f|^2_h\right)e^ldl\\
&\geq C_{\Psi,f,h,t}(Z_0)\left(\int_{-\infty}^0e^ldl+\int_0^t\frac{1-e^{l-t}}{1-e^{-t}}dl\right)\\
&=\frac{t}{1-e^{-t}}C_{\Psi,f,h,t}(Z_0).	\end{split}
\end{displaymath}
Then Proposition \ref{p:1} has thus been proved.	
\end{proof}

\section{Appendix: Proof of Lemma \ref{L2 method}}
In this section, we prove Lemma \ref{L2 method}.
\subsection{Some results used in the proof of Lemma \ref{L2 method}}

In this section, we do some preparations for the proof of Lemma \ref{L2 method}.

Let $M$ be a complex manifold. Let $\omega$ be a continuous hermitian metric on $M$. Let $dV_M$ be a continuous volume form on $M$. We denote by $L^2_{p,q}(M,\omega,dV_M)$ the spaces of $L^2$ integrable $(p,q)$ forms over $M$ with respect to $\omega$ and $dV_M$. It is known that $L^2_{p,q}(M,\omega,dV_M)$ is a Hilbert space.
\begin{Lemma}
\label{weakly convergence}
Let $\{u_n\}_{n=1}^{+\infty}$ be a sequence of $(p,q)$ forms in $L^2_{p,q}(M,\omega,dV_M)$ which is weakly convergent to $u$. Let $\{v_n\}_{n=1}^{+\infty}$ be a sequence of Lebesgue measurable real functions on $M$ which converges pointwisely to $v$. We assume that there exists a constant $C>0$ such that $|v_n|\le C$ for any $n$. Then $\{v_nu_n\}_{n=1}^{+\infty}$ weakly converges to $vu$ in $L^2_{p,q}(M,\omega,dV_M)$.
\end{Lemma}
\begin{proof}Let $g\in L^2_{p,q}(M,\omega,dV_M)$. Consider
\begin{equation*}
  \begin{split}
     I & =|\langle v_nu_n,g\rangle-\langle vu,g\rangle| \\
       &=|\int_{M}(v_nu_n,g)_{\omega}dV_M-\int_{M}(vu,g)_{\omega}dV_M|\\
       &\le|\int_{M}(v_nu_n-vu_n,g)_{\omega}dV_M|+|\int_{M}(vu_n-vu,g)_{\omega}dV_M|\\
       &=|\int_{M}(u_n,v_ng-vg)_{\omega}dV_M|+|\int_{M}(u_n-u,vg)_{\omega}dV_M|\\
       &\le||u_n||\cdot||v_ng-vg||+|\int_{M}(u_n-u,vg)_{\omega}dV_M|.
  \end{split}
\end{equation*}
Denote $I_1:=||u_n||\cdot||v_ng-vg||$ and $I_2:=|\int_{M}(u_n-u,vg)_{\omega}dV_M|$. It follows from $\{u_n\}_{n=1}^{+\infty}$ weakly converges to $u$ that $||u_n||$ is uniformly bounded with respect to $n$. Note that $|v_n|$ is uniformly bounded with respect to $n$. We know that $|v|<C$ and then $vg\in L^2_{p,q}(M,\omega,dV_M)$. Hence we have $I_2 \to 0$ as $n\to+\infty$. It follows from Lebesgue dominated convergence theorem that we have $\lim_{n\to+\infty}I_1 = 0$.

Hence $\lim_{n\to+\infty}I = 0$ and we know $\{v_nu_n\}_{n=1}^{+\infty}$ weakly converges to $vu$ in $L^2_{p,q}(M,\omega,dV_M)$.
\end{proof}

\begin{Lemma}[see \cite{Demailly00}]
\label{BKN Identity}
Let Q be a Hermitian vector bundle on a K\"ahler manifold M of dimension $n$ with a
K\"ahler metric $\omega$. Assume that $\eta , g >0$ are smooth functions on M. Then
for every form $v\in D(M,\wedge^{n,q}T^*M \otimes Q)$ with compact support we have
\begin{equation}
\begin{split}
\label{BKN Identiy formula}
&\int_M (\eta+g^{-1})|D^{''*}v|^2_QdV_M+\int_M \eta|D^{''}v|^2_QdV_M \\
\ge  &\int_M \langle[\eta \sqrt{-1}\Theta_Q-\sqrt{-1}\partial \bar{\partial}
\eta-\sqrt{-1}g
\partial\eta \wedge\bar{\partial}\eta, \Lambda_{\omega}]v,v\rangle_QdV_M.
\end{split}
\end{equation}
\end{Lemma}

The following approximation result can be referred to \cite{Demaillybook}.
Let $(X,\omega)$ be a hermitian manifold. Let $Q$ be a holomorphic vector bundle on $X$ and $h$ be a hermitian metric on $Q$. Denote $D(M,\wedge^{n,q}T^*M \otimes Q)$ be the space of $Q$-valued smooth $(n,q)$ forms with compact support for any $q\ge 0$. Let $D'':L^2(X,\wedge^{n,q}T^*M \otimes Q)\to L^2(X,\wedge^{n,q+1}T^*M \otimes Q)$ be the extension of $\bar{\partial}$-operator in the sense of distribution. Let $D''^*$ be the adjoint operator of $D''$ in the Von-Neumann sense.
\begin{Lemma}[see \cite{Demaillybook}]
\label{approximation on complete mfld}
Assume that $(X,\omega)$ is complete. Then $D(M,\wedge^{n,\bullet}T^*M \otimes Q)$ is dense in $\text{Dom}D''$, $\text{Dom}D''^*$ and $\text{Dom}D''\cap \text{Dom}D''^*$ respectively for the graph norms

\centerline{$u\to ||u||+||D''u||$,  $u\to ||u||+||D''^*u||$, $u\to ||u||+||D''u||+||D''^*u||$.}
\end{Lemma}

\begin{Lemma}[Lemma 4.2 in \cite{guan-zhou13ap}]
Let Q be a Hermitian vector bundle on a K\"ahler manifold M of dimension $n$ with a
K\"ahler metric $\omega$. Let $\theta$ be a continuous (1,0) form on M.
Then we have
\begin{equation}
[\sqrt{-1}\theta \wedge
\bar{\theta},\Lambda_\omega]\alpha=\bar{\theta}\wedge(\alpha\llcorner(\bar{\theta})^\sharp),
\end{equation}
for any (n,1) form $\alpha$ with value in Q. Moreover, for any positive (1,1) form
$\beta$, we have $[\beta,\Lambda_\omega]$ is semipositive.
\label{sempositive lemma}
\end{Lemma}

We need the following propositions of positive definite hermitian matrices.

Let $\mathcal{M}:=\{M\in M_n(\mathbb{C}): M \text{\ is a positive definite hermitian matrix}\}$. Note that $M_n(\mathbb{C})$ is an $2n^2$-dimensional real manifold. Then $\mathcal{M}$ is an $n^2$-dimensional real sub-manifold of $M_n(\mathbb{C})$.
Denote $F:M_n(\mathbb{C})\to M_n(\mathbb{C})$ by $F(X)=X^2$ for any $X\in M_n(\mathbb{C})$. Denote $F|_\mathcal{M}: \mathcal{M}\to \mathcal{M}$. We have the following property of $F|_\mathcal{M}$.
\begin{Lemma}\label{sqrt of positive definite matrices}
$F|_\mathcal{M}: \mathcal{M}\to \mathcal{M}$ is a smooth diffeomorphism.
\end{Lemma}
\begin{proof}It is easy to see that $F|_\mathcal{M}: \mathcal{M}\to \mathcal{M}$ is a smooth injection.

Let $M\in \mathcal{M}\subset M_n(\mathbb{C})$ be any positive definite hermitian matrix. Then $M$ can be viewed as a self-adjoint positive definite linear map on $\mathbb{C}^n$. Then we can find a unitary matrix $P$ such that $M=P^{-1}\tilde{M}P$, where $\tilde{M}=diag(\lambda_1,\lambda_2,\ldots,\lambda_n)$ is a diagonal matrix and all $\lambda_i\in\mathbb{R}_{>0}$. Denote $\tilde{N}:=diag(\sqrt{\lambda_1},\sqrt{\lambda_2},\ldots,\sqrt{\lambda_n})$. Then we have $M=P^{-1}\tilde{N}PP^{-1}\tilde{N}P=N^2$, where $N:=P^{-1}\tilde{N}P$ is a  positive definite hermitian matrix. Then we have $M=N^2$. By the theory of positive linear operator, we know that $N$ is unique.
Hence we know that $F|_\mathcal{M}: \mathcal{M}\to \mathcal{M}$ is surjective and the inverse mapping $(F|_\mathcal{M})^{-1}:\mathcal{M}\to \mathcal{M}$ of $F|_\mathcal{M}$ exists.

Assume that $X$ is a positive definite hermitian matrix. Let $dF_X$ be the tangent map induced by $F$ at point $X\in M_n(\mathbb{C})$. Then for any matrix $Y\in T_X(M_n(\mathbb{C}))\cong M_n(\mathbb{C})$, $dF(Y)=\lim_{t\to 0}\frac{F(X+tY)-F(X)}{t}=XY+YX$. As $X$ is a positive definite hermitian matrix. We can find a unitary matrix $Q$ such that $X=Q^{-1}\tilde{X}Q$, where $\tilde{X}$ is a diagonal matrix and denote $\tilde{Y}$ by the equation $Y=Q^{-1}\tilde{Y}Q$. Then $XY+YX=0$ if and only if $\tilde{X}\tilde{Y}+\tilde{Y}\tilde{X}=0$. As $\tilde{X}$ is a diagonal matrix, we know that $\tilde{X}\tilde{Y}+\tilde{Y}\tilde{X}=0$ if and only if $\tilde{Y}=0$ which implies that $XY+YX=0$ if and only if $Y=0$. Hence we know that $dF_X$ is non-degenerate at $X$ when $X$ is a positive definite hermitian matrix. Hence we know that $F^{-1}$ exists locally near $X$ and $F^{-1}$ is smooth.

 By the uniqueness of inverse map, we know that $(F|_\mathcal{M})^{-1}=F^{-1}|_{\mathcal{M}}$, hence $(F|_\mathcal{M})^{-1}$ is a smooth map from $\mathcal{M}\to \mathcal{M}$. We have proved that $F|_\mathcal{M}: \mathcal{M}\to \mathcal{M}$ is a smooth diffeomorphism.
\end{proof}

\begin{Remark}\label{convergence of sqrt of pos def matric}
Let $A_k=(a_{ij}^k)\in M_n(\mathbb{C})$ and $A=(a_{ij})\in M_n(\mathbb{C})$ be a family of $n\times n$ positive definite hermitian matrices such that $\lim_{k\to+\infty}A_k=A$ (which means for any $i,j\in\{1,2,\ldots.n\}$, $\lim_{k\to+\infty}a_{ij}^k=a_{ij}$). Then there exists a unique family of $n\times n$ positive definite hermitian matrices $B_k=(b_{ij}^k)$ and  $B=(b_{ij})$ such that   $B_k^2=A_k$ and $B^2=A$. More over, we have $\lim_{k\to+\infty}B_k=B$.
\end{Remark}
\begin{proof}
Denote $B_k:=(F|_\mathcal{M})^{-1}(A_k)$ and $B:=(F|_\mathcal{M})^{-1}(A)$. Then we have the
 existence and uniqueness of $B_k$ and $B$. As $\lim_{k\to+\infty}A_k=A$, by the smoothness of $(F|_\mathcal{M})^{-1}$, we know that $\lim_{k\to+\infty}B_k=B$.
Hence we have remark \ref{convergence of sqrt of pos def matric}.
\end{proof}

\begin{Lemma}\label{metric consturcution}
Let $A$ and $B$ be two $n\times n$ positive definite hermitian matrices. Then there exists a unique matrix $C$ with positive eigenvalue such that $A=CB\overline{C}^T$ and $CB=B\overline{C}^T$. The matrix $C$ depends smoothly on $A$ and $B$ in $\mathcal{M}\times\mathcal{M}$. Especially, if $\lim_{i\to+\infty}A_i=A_0$ and $\lim_{i\to+\infty}B_i=B_0$, then we have $\lim_{i\to+\infty}C_i=C_0$.
\end{Lemma}
\begin{proof}It follows from Remark \ref{convergence of sqrt of pos def matric} that there exists a unique positive definite hermitian matrix $b$ such that $B=b^2$ and the matrix $b$ depends smoothly on $B$ in $\mathcal{M}$. As $b=\overline{b}^T$, we know that $b^{-1}Ab^{-1}$ is a positive definite hermitian matrices. It follows from Remark \ref{convergence of sqrt of pos def matric} that there exists a unique positive definite hermitian matrix $a$ such that $b^{-1}Ab^{-1}=a^2$ and we note that the matrix $a$ depends smoothly on $A$ and $B$ in $\mathcal{M}\times\mathcal{M}$. Denote $C:=bab^{-1}$. Then $C$ depends smoothly on $A$ and $B$ in $\mathcal{M}\times\mathcal{M}$. We note that all eigenvalues of $C$ are positive and $\overline{C}^T=b^{-1}ab$. We have
$$CB\overline{C}^T=bab^{-1}b^2b^{-1}ab=ba^2b=A,$$ and
$$CB=bab^{-1}b^2=bab=b^2b^{-1}ab=B\overline{C}^T.$$

Now we prove the uniqueness of $C$. Assume that there exists another $\tilde{C}$ satisfies $\tilde{C}B\overline{\tilde{C}}^T=A$ and
$\tilde{C}B=B\overline{\tilde{C}}^T$. It follows from $\tilde{C}B=B\overline{\tilde{C}}^T$ and $B=b^2$ that we have $b^{-1}\tilde{C}b=b\overline{\tilde{C}}^Tb^{-1}$, which shows that $b^{-1}\tilde{C}b$ is a hermitian matrix. We note that $$(b^{-1}\tilde{C}b)^2=b^{-1}\tilde{C}bb\overline{\tilde{C}}^Tb^{-1}=
b^{-1}\tilde{C}B\overline{\tilde{C}}^Tb^{-1}=b^{-1}Ab^{-1}.$$
By the uniqueness of $a$ such that $b^{-1}Ab^{-1}=a^2$, we know that $b^{-1}\tilde{C}b=a$ and then we have $C=\tilde{C}=bab^{-1}$.

If $\{A_i\}_{i=0}^{+\infty}$ and $\{B_i\}_{i=0}^{+\infty}$ satisfy $\lim_{i\to+\infty}A_i=A_0$ and $\lim_{i\to+\infty}B_i=B_0$, then we have $C_i$ such that $A_i=C_iB\overline{C_i}^T$ and $C_iB=B\overline{C_i}^T$, for any $i\ge 0$. As $C_i$ depends smoothly on $A_i$ and $B_i$ in $\mathcal{M}\times\mathcal{M}$ for any $i\ge 0$, we know that $\lim_{i\to+\infty}C_i=C_0$.
\end{proof}

Let $X$ be an $n-$dimensional complex manifold and $\omega$ be a hermitian metric on $X$. Let $Q$ be a vector bundle on $X$ with rank $r$. Let $\{h_i\}_{i=1}^{+\infty}$ be a family of $C^2$ smooth hermitian metric on $Q$ and $h$ be a measurable metric on $Q$ such that $\lim_{i\to+\infty}h_i=h$ almost everywhere on $X$.  We assume that $\{h_i\}_{i=1}^{+\infty}$ and $h$ satisfy one of the following conditions,\\
$(A)$ $h_i$ is increasingly convergent to $h$ as $i\to+\infty$;\\
$(B)$ there exists a continuous metric $\hat{h}$ on $Q$ and a constant $C>0$ such that for any $i\ge 0$, $\frac{1}{C}\hat{h}\le h_i\le C\hat{h}$ and $\frac{1}{C}\hat{h}\le h\le C\hat{h}$.

Denote $\mathcal{H}_i:=L^2(X,K_X\otimes Q,h_i,dV_{\omega})$ and $\mathcal{H}:=L^2(X,K_X\otimes Q,h,dV_{\omega})$. Note that $\mathcal{H}\subset \mathcal{H}_i\subset \mathcal{H}_1$ for any $i\in\mathbb{Z}_{>0}$.
\begin{Lemma}\label{existence of linear isomorphism}
There exists a linear isomorphism $H_i:\mathcal{H}_i\to \mathcal{H}_1$ (and $H:\mathcal{H}\to \mathcal{H}_1$) which preserves inner product, i.e., for any $\alpha,\beta\in \mathcal{H}_i$ (or $\tilde{\alpha},\tilde{\beta}\in \mathcal{H}$),
$$\langle\alpha,\beta\rangle_{h_i}=\langle H_i(\alpha),H_i(\beta)\rangle_{h_1}\ \big(\text{and\ } \langle\tilde{\alpha},\tilde{\beta}\rangle_{h}=\langle H(\tilde{\alpha}),H(\tilde{\beta})\rangle_{h_1}\big),$$
and satisfies $H^{-1}_{i}:\mathcal{H}_1\to \mathcal{H}_i\subset \mathcal{H}_1$
(and $H^{-1}:\mathcal{H}_1\to \mathcal{H}\subset \mathcal{H}_1$) is self-adjoint.
Moreover, $H_i^{-1}(\gamma)$ converges to $H^{-1}(\gamma)$ point-wisely for any $\gamma\in\mathcal{H}_1$.
\end{Lemma}
\begin{proof}We firstly consider the local case.

Let $X=\Omega\subset \mathbb{C}^n$ be a bounded domain and $Q=\Omega\times \mathbb{C}^r$. In the local case, every metric $h_i$ (or $h$) on $Q$ can be viewed as a positive definite hermitian matrix $h_i=\big(h^{i}_{kl}(x)\big)$ (or $h=\big(h_{kl}(x)\big)$) where all $\{h^{i}_{kl}(x)\}_{k,l=1}^r$ are $C^2$ smooth functions on $\Omega$ (all $\{h_{kl}(x)\}_{k,l=1}^r$ are measurable functions on $\Omega$). It follows from Lemma \ref{metric consturcution} that there exists $C_i$ (or $C$) such that $h_i=C_ih_1\overline{C_i}^T$ and $C_ih_1=h_1\overline{C_i}^T$ ($h=Ch_1\overline{C}^T$ and $Ch_1=h_1\overline{C}^T$). By Lemma \ref{metric consturcution}, we also know that $C_i:=\big(C^i_{k,l}(x)\big)_{k,l=1}^r$ is $C^2$ smooth matrix functions, $C:=\big(C_{k,l}(x)\big)_{k,l=1}^r$ is measurable matrix function and $\lim_{i\to+\infty}C_i(x)=C(x)$ almost everywhere.
 Then for any measurable section $f=(f_1,f_2,\ldots,f_r)$ of $Q=\Omega\times \mathbb{C}^r$, denote $H_i(f):=(f_1,f_2,\ldots,f_r)C_i$ and $H(f):=(f_1,f_2,\ldots,f_r)C$. Then for any $\alpha,\beta\in \mathcal{H}_i$,
\begin{equation*}
\begin{split}
 &\langle\alpha,\beta\rangle_{h_i}=\alpha(h^i_{rl})
\bar{\beta}^{T}=\alpha C_ih_1\overline{C_i}^T\bar{\beta}^{T}=\langle H_i(\alpha),H_i(\beta)\rangle_{h_1},
\end{split}
\end{equation*}
As $C_ih_1=h_1\overline{C_i}^T$, we know that  $H^{-1}_{i}:\mathcal{H}_1\to \mathcal{H}_i\subset \mathcal{H}_1$ is self-adjoint.
Similar discussion shows that for any $\tilde{\alpha},\tilde{\beta}\in \mathcal{H}$,
$$\langle\tilde{\alpha},\tilde{\beta}\rangle_{h}=\langle H(\tilde{\alpha}),H(\tilde{\beta})\rangle_{h_1},$$
and $H^{-1}:\mathcal{H}_1\to \mathcal{H}\subset \mathcal{H}_1$ is self-adjoint.

It follows from $\lim_{i\to+\infty}C_i=C$ that we know that $H_i^{-1}(f)$ converges to $H^{-1}(f)$ almost everywhere on $\Omega$ for any measurable section $f=(f_1,f_2,\ldots,f_r)$ of $Q=\Omega\times \mathbb{C}^r$. $C_i$ and $C$ are obviously linear isomorphisms. Hence in the local case, we have found linear isomorphism satisfying all the requirements in Lemma \ref{existence of linear isomorphism}.

Now we prove the existences of $H_i$ and $H$ in the global case. Let $U_{\alpha}$ and $U_{\beta}$ be two open subsets of $X$ such that
$U_{\alpha}\cap U_{\beta}\neq\emptyset$. Let $G_{\alpha\beta}:U_{\alpha}\cap U_{\beta}\to GL_r(\mathbb{C})$ be the transition functions of $Q$. Then we know that the representative matrices of metric  $H^{\alpha}_i$ and $H^{\beta}_i$ under different basis are congruent, i.e. $H^{\alpha}_i=G_{\alpha\beta}^TH^{\beta}_i\overline{G_{\alpha\beta}}$, for all $i\in\mathbb{Z}_{\ge 0}$. On $U_{\alpha}$,  we have
$$H^{\alpha}_i=C^{\alpha}_{i}H^{\alpha}_{1}
(\overline{C^{\alpha}_{i}}^{T}) \text{\ and\ } C^{\alpha}_{i}H^{\alpha}_{1}=H^{\alpha}_{1}(\overline{C^{\alpha}_{i}}^{T}). $$
 Similarly, on $U_{\beta}$, we have
$$H^{\beta}_i=C^{\beta}_{i}H^{\beta}_{1}
(\overline{C^{\beta}_{i}}^{T})  \text{\ and\ } C^{\beta}_{i}H^{\beta}_{1}=H^{\beta}_{1}(\overline{C^{\beta}_{i}}^{T}). $$
 On $U_{\alpha}\cap U_{\beta}$, we have
\begin{equation}\label{existence of linear isomorphism: formula 1}
\begin{split}
   G_{\alpha\beta}^TH^{\beta}_i\overline{G_{\alpha\beta}}&=H^{\alpha}_i \\
     &=C^{\alpha}_{i}H^{\alpha}_{1}
(\overline{C^{\alpha}_{i}}^{T})=C^{\alpha}_{i}G_{\alpha\beta}^TH^{\beta}_1\overline{G_{\alpha\beta}}
(\overline{C^{\alpha}_{i}}^{T}).
\end{split}
\end{equation}
On $U_{\alpha}\cap U_{\beta}$, it follows from $C^{\alpha}_{i}H^{\alpha}_{1}=H^{\alpha}_{1}(\overline{C^{\alpha}_{i}}^{T})$ that
\begin{equation}\label{existence of linear isomorphism: formula 2}
\begin{split}
  C^{\alpha}_{i}G_{\alpha\beta}^TH^{\beta}_1\overline{G_{\alpha\beta}}
  =G_{\alpha\beta}^TH^{\beta}_1\overline{G_{\alpha\beta}}(\overline{C^{\alpha}_{i}}^{T})
\end{split}
\end{equation}
Denote $\hat{C}_{\alpha}:=(G_{\alpha\beta}^T)^{-1}C_i^{\alpha}G_{\alpha\beta}^T$. Then we have
$\overline{\hat{C}_{\alpha}}^{T}=(\overline{G_{\alpha\beta}})\overline{C_i^{\alpha}}^{T}
(\overline{G_{\alpha\beta}})^{-1}$. It follows from equalities \eqref{existence of linear isomorphism: formula 1}, \eqref{existence of linear isomorphism: formula 2} that we have
$$H_i^{\beta}=\hat{C}_{\alpha}H_1^{\beta}\overline{\hat{C}_{\alpha}}^{T} \text{\ and\ } \overline{\hat{C}_{\alpha}}^{T}H_1^{\beta}=H_1^{\beta}\overline{\hat{C}_{\alpha}}^{T}.$$
It follows from the uniqueness of $C_i^{\beta}$ that we have $C_i^{\beta}=\hat{C}_{\alpha}:=(G_{\alpha\beta}^T)^{-1}C_i^{\alpha}G_{\alpha\beta}^T$, which shows that $H_i(f)=(f_1,\ldots,f_r)C_i$ can be defined globally. Similar discussion shows that $H(f)=(f_1,\ldots,f_r)C$ can be defined globally.

Lemma \ref{existence of linear isomorphism} has been proved.

\end{proof}

Let $X$, $Q$, $\{h_i\}_{i=1}^{+\infty}$ and $h$ be as in Lemma \ref{existence of linear isomorphism}. Denote $P_i:=\mathcal{H}_i\to \text{Ker}D''$ and $P:=\mathcal{H}\to \text{Ker}D''$ be the orthogonal projections with respect to $h_i$ and $h$ respectively.
\begin{Lemma}\label{weakly converge lemma} For any sequence of $Q$-valued $(n,0)$-forms $\{f_i\}_{i=1}^{+\infty}$ which satisfies $f_i\in\mathcal{H}_i$ and $||f_i||_{h_i}\le C_1$ for some constant $C_1>0$, there exists a $Q$-valued $(n,0)$-form $f_0\in \mathcal{H}$ such that there exists a subsequence of $\{f_i\}_{i=1}^{+\infty}$ (also denoted by $\{f_i\}_{i=1}^{+\infty}$) weakly converges to $f_0$ in $\mathcal{H}_1$ and $P_i(f_i)$ weakly converges to $P(f_0)$ in $\mathcal{H}_1$.
\end{Lemma}
\begin{proof} Denote $a_i=P_i(f_i)$ and $b_i:=f_i-P_i(f_i)$. We note that $b_i\in (\text{Ker}D'')_i^{\bot}\subset \mathcal{H}_i$, where $(\text{Ker}D'')_i^{\bot}$ is the orthogonal complement space of $\text{Ker}D''$ in $\mathcal{H}_i$ with respect to $h_i$. It follows from $||f_i||_{h_i}\le C_1$ that we have $||a_i||_{h_i}\le C_1$ and $||b_i||_{h_i}\le C_1$. Denote
$$F_i:=H_i(f_i),\ A_i:=H_i(a_i)\text{\ and\ }B_i:=H_i(b_i).$$

Then we know that $||A_i||_{h_1}=||a_i||_{h_i}\le C_1$ is uniformly bounded. Since the closed unit ball of the Hilbert space is weakly compact, we can extract a subsequence of $\{A_i\}_{i=1}^{+\infty}$ (still denoted by $A_i$) weakly convergent to some $A_0$ in $\mathcal{H}_1$. For similar reason, we know that $\{B_i\}_{i=1}^{+\infty}$ weakly converges to some $B_0$ in $\mathcal{H}_1$ and $\{F_i\}_{i=1}^{+\infty}$ weakly converges to some $F_0$ in $\mathcal{H}_1$.

Let $\beta\in \mathcal{H}_1$. When $\{h_i\}_{i=1}^{+\infty}$ and $h$ satisfy condition $(A)$, it follows from dominated convergence theorem,
$$||H^{-1}(\beta)||_{h_1}\le ||H^{-1}(\beta)||_{h}=||\beta||_{h_1}\text{\ and\ }
||H^{-1}_i(\beta)||_{h_1}\le ||H^{-1}_i(\beta)||_{h_i}=||\beta||_{h_1},$$
 that we have $||H^{-1}_i(\beta)-H^{-1}(\beta)||_{h_1}\to 0$ as $i\to+\infty$.
 When $\{h_i\}_{i=1}^{+\infty}$ and $h$ satisfy condition $(B)$, it follows from dominated convergence theorem,
$$||H^{-1}(\beta)||_{h_1}\le C^2||H^{-1}(\beta)||_{h}=C^2||\beta||_{h_1}\text{\ and\ }
||H^{-1}_i(\beta)||_{h_1}\le C^2||H^{-1}_i(\beta)||_{h_i}=C^2||\beta||_{h_1},$$
that we have $||H^{-1}_i(\beta)-H^{-1}(\beta)||_{h_1}\to 0$ as $i\to+\infty$.

Then when $\{h_i\}_{i=1}^{+\infty}$ and $h$ satisfy condition $(A)$ or $(B)$, we have
\begin{equation}\nonumber
\begin{split}
&\lim_{i\to +\infty}\int_X\langle a_i,\beta\rangle_{h_1}dV_{\omega}\\
=&\lim_{i\to +\infty}\int_X\langle H_i(a_i),H^{-1}_i(\beta)\rangle_{h_1}dV_{\omega}\\
=&\lim_{i\to +\infty}\int_X\langle H_i(a_i),H^{-1}(\beta)\rangle_{h_1}dV_{\omega}
+\lim_{i\to +\infty}\int_X\langle H_i(a_i),H^{-1}_i(\beta)-H^{-1}(\beta)\rangle_{h_1}dV_{\omega}\\
=&\int_X\langle A_0,H^{-1}(\beta)\rangle_{h_1}dV_{\omega}\\
=&\int_X\langle H^{-1}(A_0),\beta\rangle_{h_1}dV_{\omega},
\end{split}
\end{equation}
where the first equality holds because of $H_i^{-1}$ is self-adjoint and the third equality holds because of
$$\int_M\langle H_i(a_i),H^{-1}_i(\beta)-H^{-1}(\beta)\rangle_{h_1}dV_{\omega}\le ||H_i(a_i)||_{h_1}||H^{-1}_i(\beta)-H^{-1}(\beta)||_{h_1},$$
$||H_i(a_i)||_{h_1}=||a_i||_{h_i}\le C_1$, $||H^{-1}_i(\beta)-H^{-1}(\beta)||_{h_1}\to 0$ as $i\to+\infty$. Denote $a_0:=H^{-1}(A_0)$. Then $a_0\in\mathcal{H}\subset\mathcal{H}_1$ and we know that $a_i$ weakly converges to $a_0$ in $\mathcal{H}_1$. It follows from $D''(a_i)=0$ that we have $D''(a_0)=0$. Denote $b_0:=H^{-1}(B_0)\in\mathcal{H}\subset\mathcal{H}_1$ and $f_0:=H^{-1}(F_0)\in\mathcal{H}\subset\mathcal{H}_1$. Using similar discussion, we know that $b_i$ weakly converges to $b_0$ in $\mathcal{H}_1$ and $f_i$ weakly converges to $f_0$ in $\mathcal{H}_1$.

It follows from the uniqueness of weak limit and $f_i=a_i+b_i$ that we have $f_0=a_0+b_0$ in $\mathcal{H}$. Now we prove that $b_0\in (\text{Ker}D'')^{\bot}\subset \mathcal{H}$, where $(\text{Ker}D'')^{\bot}$ is the orthogonal complement space of $\text{Ker}D''$ in $\mathcal{H}$ with respect to $h$. Let $\gamma\in \text{Ker}D''\subset \mathcal{H}$. We have
\begin{equation}\label{weakly converge lemma: b_0 orth}
\begin{split}
&\int_X\langle b_0,\gamma\rangle_{h}dV_{\omega}\\
=&\int_X\langle H(b_0),H(\gamma)\rangle_{h_1}dV_{\omega}\\
=&\lim_{i\to+\infty}\int_X\langle H_i(b_i),H_i(\gamma)\rangle_{h_1}dV_{\omega}
+\lim_{i\to+\infty}\int_X\langle H_i(b_i),H(\gamma)-H_i(\gamma)\rangle_{h_1}dV_{\omega}\\
\le &
\lim_{i\to+\infty}\int_X\langle H_i(b_i),H_i(\gamma)\rangle_{h_1}dV_{\omega}
+\lim_{i\to+\infty}||b_i||_{h_i}||H(\gamma)-H_i(\gamma)||_{h_1}\\
=&\lim_{i\to+\infty}\int_X\langle b_i,\gamma\rangle_{h_i}dV_{\omega}
+\lim_{i\to+\infty}||b_i||_{h_i}||H(\gamma)-H_i(\gamma)||_{h_1}
\end{split}
\end{equation}
When $\{h_i\}_{i=1}^{+\infty}$ and $h$ satisfy condition $(A)$, it follows from dominated convergence theorem and
$$||H(\gamma)||_{h_1}=||\gamma||_{h}\text{\ and\ }
||H_i(\gamma)||_{h_1}= ||\gamma||_{h_i}\le||\gamma||_{h},$$
 that we have $||H(\gamma)-H_i(\gamma)||_{h_1}\to 0$ as $i\to+\infty$. When $\{h_i\}_{i=1}^{+\infty}$ and $h$ satisfy condition $(B)$, it follows from dominated convergence theorem and
$$||H(\gamma)||_{h_1}=||\gamma||_{h}\text{\ and\ }
||H_i(\gamma)||_{h_1}=||\gamma||_{h_i}\le C^2||\gamma||_{h},$$
that we have $||H(\gamma)-H_i(\gamma)||_{h_1}\to 0$ as $i\to+\infty$. Note that $||b_i||_{h_i}$ is uniformly bounded with respect to $i$. It follows from above discussion, $b_i\in (\text{Ker}D'')^{\bot}_i$ and inequality \eqref{weakly converge lemma: b_0 orth} that we have
$$\int_X\langle b_0,\gamma\rangle_{h}dV_{\omega}=0.$$
Hence we know $b_0\in (\text{Ker}D'')^{\bot}\subset \mathcal{H}$. Hence we have $P(f_0)=a_0$ and we have showed that $a_i=P_i(f_i)$ weakly converges to $a_0=P(f_0)$ in $\mathcal{H}_1$.

Lemma \ref{weakly converge lemma} has been proved.

\end{proof}

\begin{Lemma}

\label{d-bar equation with error term}
Let ($M,\omega$) be a complete K\"ahler manifold equipped with a (non-necessarily
complete) K\"ahler metric $\omega$, and let $(Q,h)$ be a hermitian vector bundle over $M$.
Assume that $\eta$ and $g$ are smooth bounded positive functions on $M$ such that $\eta+g^{-1}$ is a smooth bounded positive function on $M$ and let
$B:=[\eta \sqrt{-1}\Theta_Q-\sqrt{-1}\partial \bar{\partial} \eta-\sqrt{-1}g
\partial\eta \wedge\bar{\partial}\eta, \Lambda_{\omega}]$. Assume that $\lambda \ge
0$ is a bounded continuous functions on $M$ such that $B+\lambda I$ is positive definite everywhere on
$\wedge^{n,q}T^*M \otimes Q$ for some $q \ge 1$. Then given a form $v \in
L^2(M,\wedge^{n,q}T^*M \otimes Q)$ such that $D^{''}v=0$ and $\int_M \langle
(B+\lambda I)^{-1}v,v\rangle_{Q,\omega} dV_{\omega} < +\infty$, there exists an approximate solution
$u \in L^2(M,\wedge^{n,q-1}T^*M \otimes Q)$ and a correcting term $\tau\in
L^2(M,\wedge^{n,q}T^*M \otimes Q)$ such that $D^{''}u+P_h(\sqrt{\lambda}\tau)=v$, where $P_h:L^2(M,\wedge^{n,q}T^*M \otimes Q)\to \text{Ker}{D''}$ is the orthogonal projection and
\begin{equation}
\int_M(\eta+g^{-1})^{-1}|u|^2_{Q,\omega}dV_{\omega}+\int_M|\tau|^2_{Q,\omega}dV_{\omega} \le \int_M \langle (B+\lambda
I)^{-1}v,v\rangle_{Q,\omega} dV_{\omega}.
\end{equation}
\end{Lemma}
\begin{proof} Let $\tilde{\omega}$ be a complete K\"ahler metric on $M$. Denote $\omega_{\epsilon}=\omega+\epsilon\tilde{\omega}$, where $\epsilon\in [0,1]$. Then $\omega_{\epsilon}$ is a complete K\"ahler metric on $M$ for any $\epsilon>0$. For any $Q$-valued smooth $(n,q)$ form $\alpha$ with compact support, we have $\alpha=\alpha_1+\alpha_2$ where $\alpha_1\in \text{Ker}{D''}$ and $\alpha_2 \in (\text{Ker} D'')^{\bot}=\text{Im} \overline{D''^*}\subset \text{Ker}D''^*$. It follows from $\alpha\in \text{Dom}D''\cap \text{Dom}D''^*$ and $\alpha_2\in \text{Dom}D''^*$ that we have $\alpha_1\in \text{Dom}D''^*$. For similar reason, we know that $\alpha_2\in \text{Dom}D''$. Then it follows from Lemma \ref{approximation on complete mfld} and $\eta$, $g$ and $\eta+g^{-1}$ are smooth bounded positive functions on $M$ that inequality \eqref{BKN Identiy formula} in Lemma \ref{BKN Identity} also holds for $\alpha_1$ and $\alpha_2$.
By using Cauchy-Schwarz inequality, inequality \eqref{BKN Identiy formula} and $\alpha_2 \in \text{Ker}D''^*$, we have
\begin{equation}\label{L2 method: functional is bounded}
\begin{split}
&|\langle v,\alpha\rangle|^2_{\omega_{\epsilon},h}\\
=&|\langle v,\alpha_1\rangle|^2_{\omega_{\epsilon},h}\\
\le&\big(\int_M\langle (B+\lambda
I)^{-1}v,v\rangle_{\omega_{\epsilon},h}dV_{\omega_{\epsilon}}\big)\big(\int_M\langle B\alpha_1,\alpha_1\rangle_{\omega_{\epsilon},h}dV_{\omega_{\epsilon}}+\int_M\langle \lambda\alpha_1,\alpha_1\rangle_{\omega_{\epsilon},h}dV_{\omega_{\epsilon}}\big)\\
=&\big(\int_M\langle (B+\lambda
I)^{-1}v,v\rangle_{\omega_{\epsilon},h}dV_{\omega_{\epsilon}}\big)
\big(||(\eta+g^{-1})^{\frac{1}{2}}D''^*\alpha_1||^2_{\omega_{\epsilon},h}+||\sqrt{\lambda}\alpha_1||^2_{\omega_{\epsilon},h}\big)\\
=&\big(\int_M\langle (B+\lambda
I)^{-1}v,v\rangle_{\omega_{\epsilon},h}dV_{\omega_{\epsilon}}\big)
\big(||(\eta+g^{-1})^{\frac{1}{2}}D''^*\alpha||^2_{\omega_{\epsilon},h}+||\sqrt{\lambda}P_{\omega_{\epsilon},h}\alpha||^2_{\omega_{\epsilon},h}\big),
\end{split}
\end{equation}
where $P_{\omega_{\epsilon},h}:L^2(M,\wedge^{n,q}T^*M \otimes Q,\omega_{\epsilon},h)\to \text{Ker}{D''}$ is the projection map.
Denote $H_{1,\epsilon}:=L^2(M,\wedge^{n,q-1}T^*M \otimes Q,\omega_{\epsilon},h)$ and $H_{2,\epsilon}:=L^2(M,\wedge^{n,q}T^*M \otimes Q,\omega_{\epsilon},h)$.
Then it follows from Hahn-Banach theorem and inequality $\eqref{L2 method: functional is bounded}$ that we have a bounded linear map $H_{1,\epsilon}\oplus H_{2,\epsilon}\to \mathbb{C}$, which is an extension of the following linear map
$$\big((\eta+g^{-1})^{\frac{1}{2}}D''^*\alpha,P_{\omega_{\epsilon},h}(\alpha)\big)\to \langle v,\alpha\rangle_{\omega_{\epsilon},h}.$$
Then there exist $\tilde{u}_{\epsilon}$ and $\tau_{\epsilon}$ such that
$$\langle v,\alpha\rangle_{\omega_{\epsilon},h}=\langle \tilde{u}_{\epsilon},(\eta+g^{-1})^{\frac{1}{2}}D''^*\alpha\rangle_{\omega_{\epsilon},h}+
\langle \tau_{\epsilon},\sqrt{\lambda}P_{\omega_{\epsilon},h}(\alpha)\rangle_{\omega_{\epsilon},h},$$
and
$$||\tilde{u}_\epsilon||_{\omega_{\epsilon},h}+||\tau_\epsilon||_{\omega_{\epsilon},h}\le \int_M\langle (B+\lambda
I)^{-1}v,v\rangle_{\omega_{\epsilon},h}dV_{\omega_{\epsilon}}.$$
Denote $u_\epsilon:=(\eta+g^{-1})^{\frac{1}{2}}\tilde{u}_{\epsilon}$, then we have
\begin{equation}\label{L2method: d-bar equation epsilon}
v=D''u_\epsilon+P_{\omega_{\epsilon},h}(\sqrt{\lambda}\tau_\epsilon)
\end{equation}
and
\begin{equation}\label{L2method: estimate epsilon}
||(\eta+g^{-1})^{-\frac{1}{2}}u_\epsilon||_{\omega_{\epsilon},h}+||\tau_\epsilon||_{\omega_{\epsilon},h}\le \int_M\langle (B+\lambda
I)^{-1}v,v\rangle_{\omega_{\epsilon},h}dV_{\omega_{\epsilon}}.
\end{equation}
Note that $\int_M\langle (B+\lambda
I)^{-1}v,v\rangle_{\omega_{\epsilon},h}dV_{\omega_{\epsilon}}\le \int_M\langle (B+\lambda
I)^{-1}v,v\rangle_{\omega,h}dV_{\omega}$ for any $\epsilon>0$. Then we know that $||(\eta+g^{-1})^{-\frac{1}{2}}u_\epsilon||_{\omega_{\epsilon},h}$, $||\tau_\epsilon||_{\omega_{\epsilon},h}$ and $||\sqrt{\lambda}\tau_\epsilon||_{\omega_{\epsilon},h}$ is uniformly bounded with respect to $\epsilon$.

 Note that on any compact subset $K\subset M$, we have $\omega\le \omega_\epsilon\le \omega_1\le C_K \omega$ for some $C_K>0$. It follows from $||(\eta+g^{-1})^{-\frac{1}{2}}u_\epsilon||_{\omega_{\epsilon},h}$ is uniformly bounded with respect to $\epsilon$ and $\eta+g^{-1}$ is a smooth bounded positive function on any compact subset $K$ of $M$ that we know that $u_\epsilon$ weakly converges to some $u_0$ in $L^2(M,\wedge^{n,q-1}T^*M \otimes Q,\text{loc})$. It follows from Lemma \ref{weakly convergence} that $(\eta+g^{-1})^{-\frac{1}{2}}u_\epsilon$ weakly converges to some $(\eta+g^{-1})^{-\frac{1}{2}}u_0$ in $L^2(M,\wedge^{n,q-1}T^*M \otimes Q,\text{loc})$. Let $\epsilon_1>0$ be given. Then we have
 \begin{equation}\label{L2 method: estimate for u_0}
 \begin{split}
 \int_K ||(\eta+g^{-1})^{\frac{1}{2}}u_0||_{\omega_{\epsilon_1}}
 \le&\liminf_{\epsilon\to 0} \int_K ||(\eta+g^{-1})^{\frac{1}{2}}u_{\epsilon}||_{\omega_{\epsilon_1}}\\
 \le &
 \liminf_{\epsilon\to 0} \int_K ||(\eta+g^{-1})^{\frac{1}{2}}u_{\epsilon}||_{\omega_{\epsilon}}\\
 \le &\liminf_{\epsilon\to 0} \int_M ||(\eta+g^{-1})^{\frac{1}{2}}u_{\epsilon}||_{\omega_{\epsilon}}.
 \end{split}
 \end{equation}

It follows from Lemma \ref{weakly converge lemma} that we know that $\tau_\epsilon$ weakly converges to $\tau_0$ in $\mathcal{H}_{2,1}$, $\sqrt{\lambda}\tau_\epsilon$ weakly converges to $\tilde{\tau}_0$ in $\mathcal{H}_{2,1}$ and $P_{\omega_\epsilon,h}(\sqrt{\lambda}\tau_\epsilon)$ weakly converges to $P_{\omega,h}(\tilde{\tau}_0)$ in $\mathcal{H}_{2,1}$. Lemma \ref{weakly convergence} shows that $\sqrt{\lambda}\tau_\epsilon$ weakly converges to $\sqrt{\lambda}\tau_0$ in $\mathcal{H}_{2,1}$ and we know that $\tilde{\tau}_0=\sqrt{\lambda}\tau_0$ in $\mathcal{H}_{2,1}$ since the weak limit is unique. Hence we have $P_{\omega_\epsilon,h}(\sqrt{\lambda}\tau_\epsilon)$ weakly converges to $P_{\omega,h}(\sqrt{\lambda}\tau_0)$ in $\mathcal{H}_{2,1}$.

It follows from $\tau_\epsilon$ weakly converges to $\tau_0$ and $P_{\omega_\epsilon,h}(\sqrt{\lambda}\tau_\epsilon)$ weakly converges to $P_{\omega,h}(\sqrt{\lambda}\tau_0)$ in $\mathcal{H}_{2,1}$ that we have $\tau_\epsilon$ weakly converges to $\tau_0$ in $L^2(M,\wedge^{n,q}T^*M\otimes Q,\text{loc})$ and $P_{\omega_\epsilon,h}(\sqrt{\lambda}\tau_\epsilon)$ weakly converges to $P_{\omega,h}(\sqrt{\lambda}\tau_0)$ in $L^2(M,\wedge^{n,q}T^*M\otimes Q,\text{loc})$. Let $K$ be any compact subset of $M$. We have
 \begin{equation}\label{L2 method: estimate for tau_0}
 \begin{split}
 \int_K ||\tau_0||_{\omega_{\epsilon_1}}
 \le&\liminf_{\epsilon\to 0} \int_K ||\tau_{\epsilon}||_{\omega_{\epsilon_1}}\\
 \le &
 \liminf_{\epsilon\to 0} \int_K ||\tau_{\epsilon}||_{\omega_{\epsilon}}\\
 \le &\liminf_{\epsilon\to 0} \int_M ||\tau_{\epsilon}||_{\omega_{\epsilon}}.
 \end{split}
 \end{equation}

Note that $u_\epsilon$ weakly converges to $u_0$ in $L^2(M,\wedge^{n,q-1}T^*M \otimes Q,\text{loc})$ and $P_{\omega_\epsilon,h}(\sqrt{\lambda}\tau_\epsilon)$ weakly converges to $P_{\omega,h}(\sqrt{\lambda}\tau_0)$ in $L^2(M,\wedge^{n,q}T^*M\otimes Q,\text{loc})$. Letting $\epsilon\to 0$ in \eqref{L2method: d-bar equation epsilon}, then we have
\begin{equation}\nonumber
v=D''u_0+P_{\omega,h}(\sqrt{\lambda}\tau_0).
\end{equation}

It follows from inequalities \eqref{L2method: estimate epsilon}, \eqref{L2 method: estimate for u_0} and \eqref{L2 method: estimate for tau_0} that we have
 \begin{equation}\label{L2 method: estimate for tau and u 1}
 \begin{split}
 &\int_K ||(\eta+g^{-1})^{\frac{1}{2}}u_0||_{\omega_{\epsilon_1}}dV_{\omega_{\epsilon_1}}+\int_K ||\tau_0||_{\omega_{\epsilon_1}}dV_{\omega_{\epsilon_1}}\\
 \le &\liminf_{\epsilon\to 0} \big( \int_M ||(\eta+g^{-1})^{\frac{1}{2}}u_{\epsilon}||_{\omega_{\epsilon}}dV_{\omega_{\epsilon}}+\int_M ||\tau_{\epsilon}||_{\omega_{\epsilon}}dV_{\omega_{\epsilon}}\big)\\
 \le &\liminf_{\epsilon\to 0}  \int_M\langle (B+\lambda
I)^{-1}v,v\rangle_{\omega_{\epsilon},h}dV_{\omega_{\epsilon}}\\
\le & \int_M\langle (B+\lambda
I)^{-1}v,v\rangle_{\omega,h}dV_{\omega}.
 \end{split}
 \end{equation}
When $\epsilon_1 \to 0$ in \eqref{L2 method: estimate for tau and u 1}, by monotone convergence theorem, we have
 \begin{equation}\nonumber
 \begin{split}
 \int_K ||(\eta+g^{-1})^{\frac{1}{2}}u_0||_{\omega}+\int_K ||\tau_0||_{\omega}
\le  \int_M\langle (B+\lambda
I)^{-1}v,v\rangle_{\omega,h}dV_{\omega}.
 \end{split}
 \end{equation}
 As $K$ is any compact subset of $M$, we have
  \begin{equation}\nonumber
 \begin{split}
 \int_M ||(\eta+g^{-1})^{\frac{1}{2}}u_0||_{\omega}+\int_M ||\tau_0||_{\omega}
\le  \int_M\langle (B+\lambda
I)^{-1}v,v\rangle_{\omega,h}dV_{\omega}.
 \end{split}
 \end{equation}
Lemma
\ref{d-bar equation with error term} has been proved.

\end{proof}

\begin{Lemma}
[Theorem 6.1 in \cite{DemaillyReg}, see also Theorem 2.2 in \cite{ZZ2019}]
\label{regularization on cpx mfld}
Let ($M,\omega$) be a complex manifold equipped with a hermitian metric
$\omega$, and $\Omega \subset \subset M $ be an open set. Assume that
$T=\widetilde{T}+\frac{\sqrt{-1}}{\pi}\partial\bar{\partial}\varphi$ is a closed
(1,1)-current on $M$, where $\widetilde{T}$ is a smooth real (1,1)-form and
$\varphi$ is a quasi-plurisubharmonic function. Let $\gamma$ be a continuous real
(1,1)-form such that $T \ge \gamma$. Suppose that the Chern curvature tensor of
$TM$ satisfies
\begin{equation}
\begin{split}
(\sqrt{-1}&\Theta_{TM}+\varpi \otimes Id_{TM})(\kappa_1 \otimes \kappa_2,\kappa_1
\otimes \kappa_2)\ge 0 \\
&\forall \kappa_1,\kappa_2 \in TM \quad with \quad \langle \kappa_1,\kappa_2
\rangle=0
\end{split}
\end{equation}
for some continuous nonnegative (1,1)-form $\varpi$ on M. Then there is a family
of closed (1,1)-currents
$T_{\zeta,\rho}=\widetilde{T}+\frac{\sqrt{-1}}{\pi}\partial\bar{\partial}
\varphi_{\zeta,\rho}$ on M ($\zeta \in (0,+\infty)$ and $\rho \in (0,\rho_1)$ for
some positive number $\rho_1$) independent of $\gamma$, such that
\par
$(i)\ \varphi_{\zeta,\rho}$ is quasi-plurisubharmonic on a neighborhood of
$\bar{\Omega}$, smooth on $M\backslash E_{\zeta}(T)$,
\\
increasing with respect to
$\zeta$ and $\rho$ on $\Omega $ and converges to $\varphi$ on $\Omega$ as $\rho
\to 0$.
\par
$(ii)\ T_{\zeta,\rho}\ge\gamma-\zeta\varpi-\delta_{\rho}\omega$ on $\Omega$.
\par
where $E_{\zeta}(T):=\{x\in M:v(T,x)\ge \zeta\}$ ($\zeta>0$) is the $\zeta$-upper level set of
Lelong numbers and $\{\delta_{\rho}\}$ is an increasing family of positive numbers
such that $\lim\limits_{\rho \to 0}\delta_{\rho}=0$.
\end{Lemma}
\begin{Remark}[see Remark 2.1 in \cite{ZZ2019}]
Lemma \ref{regularization on cpx mfld} is stated in \cite{DemaillyReg} in the case $M$ is a compact complex manifold. The similar proof as in \cite{DemaillyReg} shows that Lemma \ref{regularization on cpx mfld} on noncompact complex manifold still holds where the uniform estimate $(i)$ and $(ii)$ are obtained only on a relatively compact subset $\Omega$.
\end{Remark}

\begin{Remark}\label{psh fun is singular metric}
Let $M$ be a weakly pseudoconvex K\"ahler manifold. Let $\varphi$ be a plurisubharmonic function on $M$. Then $h:=e^{-\varphi}$ is a singular metric on $E:=M\times \mathbb{C}$ in the sense of Definition \ref{singular metric} and $h$ satisfies $\Theta_h(E)\ge^s_{Nak} 0$ in the sense of Definition \ref{singular nak}.
\end{Remark}
\begin{proof}As $M$ is weakly pseudoconvex, there exists a smooth plurisubharmonic
exhaustion function $P$ on $M$. Let $M_j:=\{P<j\}$ $(k=1,2,...,) $. We choose $P$ such that
$M_1\ne \emptyset$. Then $M_j$ satisfies $M_1 \Subset  M_2\Subset  ...\Subset
M_j\Subset  M_{j+1}\Subset  ...$ and $\cup_{j=1}^n M_j=M$. Each $M_j$ is weakly
pseudoconvex K\"ahler manifold.

Let $\delta>0$ be a real number. Denote $\varphi_l:=\max\{\varphi,-l\}+\frac{\delta}{l}$. Note that $\varphi_{l+1}-\varphi_l\le -\frac{1}{l(l+1)}\delta$. We also note that $\varphi_l$ is a plurisubharmonic function on $M$ and $v(T,x)=0$ for any $x\in M$.

Let $M=M_{j+1}$, $\Omega=M_{j}$, $T=\frac{\sqrt{-1}}{\pi}\partial\bar{\partial}\psi$
, $\gamma =0$ in Lemma \ref{regularization on cpx mfld}, then there exists a family of functions $\varphi_{j,l,\zeta,\rho}$ ($\zeta\in(0,+\infty)$ and $\rho\in(0,\rho_1)$ for some positive $\rho_1$)
on $M_{j+1}$ such that\\
(1) $\varphi_{j,l,\zeta,\rho}$ is a quasi-plurisubharmonic function on a neighborhood of $\overline{M_j}$, smooth on $M_{j+1}$, increasing with respect to $\zeta$ and $\rho$ on $M_j$ and converges to $\varphi_l$ on $M_j$ as $\rho \to 0$,\\
(2) $\frac{\sqrt{-1}}{\pi}\partial\bar{\partial}\varphi_{j,l,\zeta,\rho}\ge-\zeta\varpi-\delta_{\rho}\omega$ on $M_j$,\\
where $\{\delta_{\rho}\}$ is an increasing family of positive numbers such that $\lim_{\rho\to 0}\delta_{\rho}=0$.

Let $\rho=\frac{1}{m}$. Let $\tilde{\delta}_m:=\delta_{\frac{1}{m}}$ and $\zeta=\tilde{\delta}_m$. Denote $\varphi_{j,l,m}:=\varphi_{j,l,\tilde{\delta}_m,\frac{1}{m}}$. Then we have a sequence of functions $\{\varphi_{j,l,m}\}$ satisfying\\
(1') $\varphi_{j,l,m}$ is quasi-plurisubharmonic function on $\overline{M_j}$, smooth on $M_{j+1}$, decreasing with respect to $m$ and converges to $\varphi_l$ on $M_j$ as $m\to+\infty$,\\
(2') $\frac{\sqrt{-1}}{\pi}\partial\bar{\partial}\varphi_{j,l,m}\ge-\tilde{\delta}_m\varpi-\tilde{\delta}_m\omega$ on $M_j$,\\
where $\{\tilde{\delta}_m\}$ is an decreasing family of positive numbers such that $\lim_{m\to+\infty}\tilde{\delta}_m=0$.

As $M_j$ is relatively compact in $M$, there exists a positive number $b\ge 1$ such that $b\omega\ge\varpi$ on $M_j$. Then condition (2') becomes\\
(2'') $\frac{\sqrt{-1}}{\pi}\partial\bar{\partial}\varphi_{j,l,m}\ge-\tilde{\delta}_m\varpi-\tilde{\delta}_m\omega
\ge-2b\tilde{\delta}_m\omega$ on $M_j$.

Now, for each $l\ge 1$, we choose $m_l\in\mathbb{Z}_{\ge 1}$ such that $\varphi_{j,l}:=\varphi_{j,l,m_l}$ is decreasing with respect to $l$ and converges to $\varphi$ when $l\to +\infty$. Note that $M_{j-1}\subset\subset M_j\subset \subset M_{j+1}$. Let $m_1$ be any positive integer. Now we assume that $\{m_1,m_2,\ldots,m_l\}$ has been chosen. Now we choose $m_{l+1}$.

 Denote $E_{j,l+1,m}:=\{x\in M_j |\ \varphi_{j,l+1,m}(x)-\varphi_{j,l,m_l}(x)< 0\}$ and denote $E_{j,l+1}:=\{x\in M_j|\ \varphi_{l+1}(x)-\varphi_{j,l,m_l}(x)<0\}$. Note that $\varphi_{j,l+1,m}$ and $\varphi_{j,l,m_l}$ is smooth on $M_{j+1}$, we know that $E_{j,2,m}$ is open subset of $M_j$ for any $m\ge 1$. As $\varphi_{j,l+1,m+1}\le \varphi_{j,l+1,m}$ on $M_j$ and $\varphi_{l+1}<\varphi_l\le\varphi_{j,l,m_l}$, we have $E_{j,l+1,1}\subset E_{j,l+1,2}\subset\cdots \subset E_{j,l+1,m}\subset E_{j,l+1,m+1}\subset\cdots \subset E_{j,l+1}=M_j$ for any $m\ge 1$. Hence we know that $\cup_{m=1}^{+\infty} E_{j,l+1,m}$ is an open cover of $M_j$ and then an open cover of $\overline{M_{j-1}}$. By the relatively compactness of $M_{j-1}$, we know that there exists a positive integer $m_{l+1}$ such that $M_{j-1}\subset E_{j,l+1,m_{l+1}}$. Let $\varphi_{j,l+1}:=\varphi_{2,l+1,m_{l+1}}$ and we have $\varphi_{j,l+1}<\varphi_{j,l}=\varphi_{1,l,m_l}$ on $M_{j-1}$.

Hence we can find a sequence of smooth plurisubharmonic functions $\varphi_{j,l}:=\varphi_{j,l,m_l}$ on $M_{j-1}$ which is decreasing with respect to $l$ and converges to $\varphi$ when $l\to +\infty$. We note that
$$\frac{\sqrt{-1}}{\pi}\partial\bar{\partial}\varphi_{j,l}
\ge-2b\tilde{\delta}_{m_l}\omega$$
on $M_{j-1}$. Then we know that  $(M,M\times\mathbb{C},\emptyset,M_j,e^{-\varphi},e^{-\varphi_{j,l}})$ is a singular metric on $M\times\mathbb{C}$ and $\Theta_{e^{-\varphi}}(E)\ge^s_{Nak} 0$.
\end{proof}

\begin{Lemma}
[Theorem 1.5 in \cite{Demailly82}]
\label{completeness}
Let M be a K\"ahler manifold, and Z be an analytic subset of M. Assume that
$\Omega$ is a relatively compact open subset of M possessing a complete K\"ahler
metric. Then $\Omega\backslash Z $ carries a complete K\"ahler metric.

\end{Lemma}

\begin{Lemma}
[Lemma 6.9 in \cite{Demailly82}]
\label{extension of equality}
Let $\Omega$ be an open subset of $\mathbb{C}^n$ and Z be a complex analytic subset of
$\Omega$. Assume that $v$ is a (p,q-1)-form with $L^2_{loc}$ coefficients and h is
a (p,q)-form with $L^1_{loc}$ coefficients such that $\bar{\partial}v=h$ on
$\Omega\backslash Z$ (in the sense of distribution theory). Then
$\bar{\partial}v=h$ on $\Omega$.
\end{Lemma}

The following notations can be referred to \cite{Boucksom note}.

Let $X$ be a complex manifold. An upper semi-continuous function $u:X\to [-\infty,+\infty)$ is quasi-plurisubharmonic if it is locally of the form $u=\varphi+g$ where $\varphi$ is plurisubharmonic and $g$ is smooth.
Let $\theta$ be a closed real $(1,1)$ form on $X$. By Poincar\'{e} lemma, $\theta$ is locally of the form $\theta=dd^c f$ for a smooth real-valued function $f$ which is called a local potential of $\theta$. We call a quasi-plurisubharmonic function $u$ is $\theta$-plurisubharmonic if $\theta+dd^c u\ge 0$ in the sense of currents.
\begin{Lemma}[see \cite{Demaillybook}, see also \cite{Boucksom note}]
\label{regular of max} For arbitrary $\eta=(\eta_1,\ldots,\eta_p)\in (0,+\infty)^p$, the function
$$M_{\eta}(t_1,\ldots,t_p)=\int_{\mathbb{R}^p}\max\{t_1+h_1,\ldots,t_p+h_p\}\prod\limits_{1\le j\le p}\theta(\frac{h_j}{\eta_j})dh_1\ldots dh_p$$
possesses the following properties:\par
(1) $M_{\eta}(t_1,\ldots,t_p)$ is non decreasing in all variables, smooth and convex on $\mathbb{R}^p$;\par
(2) $\max\{t_1,\ldots,t_p\}\le M_{\eta}(t_1,\ldots,t_p)\le \max\{t_1+\eta_1,\ldots,t_p+\eta_p\}$;\par
(3) $M_{\eta}(t_1,\ldots,t_p)=M_{\eta_1,\ldots,\hat{\eta}_j,\ldots,\eta_p}(t_1,\ldots,\hat{t}_j,\ldots,t_p)$ if $t_j+\eta_j\le \max\limits_{k\neq j}\{t_k-\eta_k\}$;\par
(4) $M_{\eta}(t_1+a,\ldots,t_p+a)=M_{\eta}(t_1,\ldots,t_p)+a$ for any $a\in\mathbb{R}$;\par
(5) if $u_1,\ldots,u_p$ are plurisubharmonic functions, then $u=M_{\eta}(u_1,\ldots,u_p)$ is plurisubharmonic;\par
(6) if $u_1,\ldots,u_p$ are $\theta$-plurisubharmonic functions, then $u=M_{\eta}(u_1,\ldots,u_p)$ is $\theta$-plurisubharmonic function.
\end{Lemma}
\begin{proof}The proof of (1)-(5) can be referred to \cite{Demaillybook} and the proof of (6) can be referred to \cite{Boucksom note}. For the convenience of the readers, we recall the proof of (6).

 Let $f$ be a local potential of $\theta$. We know $f+u_i$ is plurisubharmonic function. It follows from (4) and (5) that $M_{\eta}(u_1+f,\ldots,u_p+f)=M_{\eta}(u_1,\ldots,u_p)+f$ is plurisubharmonic. Hence $u=M_{\eta}(u_1,\ldots,u_p)$ is $\theta$-plurisubharmonic function.
\end{proof}

\subsection{Proof of Lemma \ref{L2 method}}

Note that $X\backslash \{F=0\}$ is a weakly pseudoconvex K{\"a}hler manifold. The following remark shows that we can assume that $F$ has no zero points on $M$.
\begin{Remark}

As $(X,E,\Sigma,X_j,h,h_s)$ is a singular hermitian metric on $E$ and $\Theta_h(E)\ge^s_{Nak}0$. We know that for any compact subset $K$, there exist a relatively compact subset $X_{j_K}\subset X$ containing $K$ and  a $C^2$ smooth hermitian metric $h_{j_K,1}$ on $X_{j_K}$ such that $h_{j_K,1}\le h$ on $K\subset X_{j_K}$.

Assume that there exists a holomorphic $E$-valued $(n,0)$ form $\hat{F}$ on $X\backslash \{F=0\}$ such that
\begin{equation*}
  \begin{split}
      & \int_{X\backslash \{F=0\}}|\hat{F}-(1-b_{t_0,B}(\Psi))fF^{1+\delta}|^2_h e^{v_{t_0,B}(\Psi)-\delta \tilde{M}}c(-v_{t_0,B}(\Psi)) \\
      \le & \left(\frac{1}{\delta}c(T)e^{-T}+\int_{T}^{t_0+B}c(s)e^{-s}ds\right)
       \int_{X\backslash \{F=0\}}\frac{1}{B}\mathbb{I}_{\{-t_0-B<\Psi<-t_0\}}|fF|^2_h.
  \end{split}
\end{equation*}

Let $K$ be any compact subset of $X$. Both $h_{j_K,1}$ and $\hat{h}$ are $C^2$ smooth hermitian metrics on $E$, then there exists a constant $c_{K}>0$, such that $h_{j_K,1}\le c_K \hat{h}$ on $K$. Note that $M_k:=\inf_{K} e^{v_{t_0,B}(\Psi)-\delta M}c(-v_{t_0,B}(\Psi))>0$ and $h_{j_K,1}\le h$.
Then we have
\begin{equation}
 \begin{split}
& \int_{(X\backslash \{F=0\})\cap K}|\hat{F}|^2_{h_{j_K,1}}\\
 \le& 2 \int_{(X\backslash \{F=0\})\cap K}|\hat{F}-(1-b_{t_0,B}(\Psi))fF^{1+\delta}|^2_{h_{j_K,1}} +2\int_{(X\backslash \{F=0\})\cap K}|(1-b_{t_0,B}(\Psi))fF^{1+\delta}|^2_{h_{j_K,1}}\\
 \le & \frac{2}{M_K}\int_{X\backslash \{F=0\}}|\hat{F}-(1-b_{t_0,B}(\Psi))fF^{1+\delta}|^2_h e^{v_{t_0,B}(\Psi)-\delta \tilde{M}}c(-v_{t_0,B}(\Psi))\\
 +&2\bigg(\sup_K |F^{1+\delta}|^2\bigg)\int_{\{\Psi<-t_0\}\cap K}|f|^2_{\hat{h}}<+\infty.
 \end{split}
\end{equation}
As $K$ is arbitrarily chosen, by Montel theorem and diagonal method, we know that there exists a holomorphic $E$-valued $(n,0)$ form $\tilde{F}$ on $X$ such that $\tilde{F}=\hat{F}$ on $X\backslash \{F=0\}$. And we have
\begin{equation*}
  \begin{split}
     & \int_{X}|\hat{F}-(1-b_{t_0,B}(\Psi))fF^{1+\delta}|^2_h e^{v_{t_0,B}(\Psi)-\delta \tilde{M}}c(-v_{t_0,B}(\Psi)) \\
      \le & \left(\frac{1}{\delta}c(T)e^{-T}+\int_{T}^{t_0+B}c(s)e^{-s}ds\right)
       \int_{X}\frac{1}{B}\mathbb{I}_{\{-t_0-B<\Psi<-t_0\}}|fF|^2_h.
  \end{split}
\end{equation*}

\end{Remark}
The following remark shows that we can assume that $c(t)$ is a smooth function.
\begin{Remark} We firstly introduce the regularization process of $c(t)$.

Let $f(x)=2\mathbb{I}_{(-\frac{1}{2},\frac{1}{2})}\ast\rho(x)$ be a smooth function on $\mathbb{R}$, where $\rho$ is the kernel of convolution satisfying $\text{supp}(\rho)\subset (-\frac{1}{3},\frac{1}{3})$ and $\rho>0$.

Let $g_i(x)=\left\{ \begin{array}{rcl}
if(ix) & \mbox{if}
&x\le 0 \\ if(i^2 x) & \mbox{if} & x>0
\end{array}\right.$, then $\{g_i\}_{i\in \mathbb{N}^+}$ is a family of smooth functions on $\mathbb{R}$ satisfying:

(1) $\text{supp}(g)\subset [-\frac{1}{i},\frac{1}{i}]$, $g_i(x)\ge 0$ for any $x\in\mathbb{R}$,

(2) $\int_{-\frac{1}{i}}^0 g_i(x)dx=1$, $\int^{\frac{1}{i}}_0 g_i(x)dx\le\frac{1}{i}$ for any $i \in \mathbb{N}^+$.

Let $\tilde{h}(t)$ be an extension of the function $c(t)e^{-t}$ from $[T,+\infty)$ to $\mathbb{R}$ such that

(1) $\tilde{h}(t)=h(t):=c(t)e^{-t}$ on $[T,+\infty)$;

(2) $\tilde{h}(t)$ is decreasing with respect to $t$;

(3) $\lim_{t\to T-0}\tilde{h}(t)=c(T)e^{-T}$.

Denote $c_{i}(t):=e^t\int_{\mathbb{R}}\tilde{h}(t+y)g_{i}(y)dy$. By the construction of convolution, we know $c_{i}(t)\in C^{\infty}(-\infty,+\infty)$.  For any $t\ge T$, we have
$$c_i(t)-c(t)\ge e^t\left(\int_{-\frac{1}{i}}^0(\tilde{h}(t+y)-\tilde{h}(t))g_i(y)dy\right)\ge 0.$$

As $\tilde{h}(t)$ is decreasing with respect to $t$, we know that $c_i(t)e^{-t}$ is also decreasing with respect to $t$. Hence $c_i(t)e^{-t}$ is locally $L^1$ integrable on $\mathbb{R}$.

As $\tilde{h}(t)$ is decreasing with respect to $t$, then set $\tilde{h}^{-}(t)=\lim\limits_{s\to t-0}\tilde{h}(s)\ge h(t)$ for any $t\in\mathbb{R}$. Note that $c^{-}(t):=\lim\limits_{s\to t-0}\tilde{h}(s)e^t\ge c(t)$ for any $t\ge T$.

Now we prove $\lim\limits_{i\to +\infty}c_i(t)e^{-t}=\tilde{h}^{-}(t)$. In fact, we have
 \begin{equation}
\begin{split}
|c_i(t)e^{-t}-\tilde{h}^{-}(t)|&\le \int_{-\frac{1}{i}}^0|\tilde{h}(t+y)-h^{-}(t)|g_i(y)dy\\
&+\int_{0}^{\frac{1}{i}}\tilde{h}(t+y)g_i(y)dy.
\label{clt}
\end{split}
\end{equation}
For any $\epsilon>0 $, there exists $\delta>0$ such that $|h(t-\delta)-h^{-}(t)|<\epsilon$. Then $\exists N>0$, such that for any $n>N$, $t\ge t+y>t-\delta$ for all $y \in [-\frac{1}{i},0)$ and $\frac{1}{i}<\epsilon$. It follows from \eqref{clt} that
$$|c_i(t)e^{-t}-\tilde{h}^{-}(t)|\le \epsilon +\epsilon \tilde{h}(t),$$
hence $\lim\limits_{i\to +\infty}c_i(t)e^{-t}=\tilde{h}^{-}(t)$ for any $t\in\mathbb{R}$. Especially, we have $\lim\limits_{i\to +\infty}c_i(T)e^{-T}=\tilde{h}^{-}(T)=c(T)e^{-T}$.

Assume that for each $i$, we have an $E$-valued holomorphic $(n,0)$ form $\tilde{F}_i$ on $X$  such that
\begin{equation}
\label{estimate for F_n}
  \begin{split}
      & \int_{X}|\tilde{F}_i-(1-b_{t_0,B}(\Psi))fF^{1+\delta}|^2_he^{v_{t_0,B}(\Psi)-\delta \tilde{M}}c(-v_{t_0,B}(\Psi)) \\
      \le & (\frac{1}{\delta}c_i(T)e^{-T}+\int_{T}^{t_0+B}c_i(s)e^{-s}ds)
       \int_{X}\frac{1}{B}\mathbb{I}_{\{-t_0-B<\Psi<-t_0\}}|fF|^2_h,
  \end{split}
\end{equation}

By construction of $c_i(t)$, we have
\begin{equation}
\begin{split}
&\int_T^{t_0+B}c_i(t_1)e^{-t_1}dt_1\\
=&\int_T^{t_0+B}\int_{\mathbb{R}}\tilde{h}(t_1+y)g_i(y)dydt_1\\
=&\int_{\mathbb{R}}g_i(y)\left(\int_{T}^{t_0+B}\tilde{h}(t_1+y)dt_1\right)dy\\
=&\int_{\mathbb{R}}g_i(y)\left(\int_{T+y}^{t_0+B+y}\tilde{h}(s)ds\right)dy\\
=&\int_{\mathbb{R}}g_i(y)\left(\int_{T}^{t_0+B}\tilde{h}(s)ds
+\int_{t_0+B}^{t_0+B+y}\tilde{h}(s)ds-\int_{T}^{T+y}\tilde{h}(s)ds\right)dy,
\label{integral of cn}
\end{split}
\end{equation}
then it follows from the construction of $g_i(t)$, $\tilde{h}(t)$ is decreasing with respect to $t$, inequality \eqref{integral of cn} and $\tilde{h}(t)=c(t)e^{-t}$ on $[T,+\infty)$ that we have
\begin{equation}
\begin{split}
\lim_{i\to+\infty}\int_T^{t_0+B}c_i(t_1)e^{-t_1}dt_1=\int_{T}^{t_0+B}c(t_1)e^{-t_1}dt_1.
\label{limitof integral of cn}
\end{split}
\end{equation}

As $(X,E,\Sigma,X_j,h,h_s)$ is a singular hermitian metric on $E$ and $\Theta_h(E)\ge^s_{Nak}0$. We know that for any compact subset $K$, there exist a relatively compact subset $M_{j_K}\subset M$ containing $K$ and  a $C^2$ smooth hermitian metric $h_{j_K,1}$ on $M_{j_K}$ such that $h_{j_K,1}\le h$ on $K\subset M_{j_K}$.
For any compact subset $K$ of $M$, we have $$\inf\limits_i\inf\limits_K e^{v_{t_0,B}(\Psi)-\delta \tilde{M}}c_i(-v_{t_0,B}(\Psi))\ge\inf\limits_K e^{v_{t_0,B}(\Psi)-\delta \tilde{M}}c(-v_{t_0,B}(\Psi)),$$ then
\begin{equation}\nonumber
\begin{split}
\sup\limits_i\int_K|\tilde{F}_i-(1-b_{t_0,B}(\Psi))fF^{1+\delta}|^2_h<+\infty.
\end{split}
\end{equation}
Hence we have
\begin{equation}\nonumber
\begin{split}
\sup\limits_i\int_K|\tilde{F}_i-(1-b_{t_0,B}(\Psi))fF^{1+\delta}|^2_{h_{j_K,1}}<+\infty.
\end{split}
\end{equation}
Note that there exists a constant $C_K>0$ such that $h_{j_K,1}\le C_K \hat{h}$ on $K$. We have
\begin{equation}\nonumber
\begin{split}
\int_K|(1-b_{t_0,B}(\Psi))fF^{1+\delta}|^2_{h_{j_K,1}}\le C_K (\sup_K|F^{1+\delta}|^2)\int_{K\cap \{\Psi<-t_0\}} |f|^2_{\hat{h}}<+\infty,
\end{split}
\end{equation}
then $\sup\limits_i\int_K|\tilde{F}_i|^2_{h_{j_K,1}}<+\infty$,  by Montel theorem and diagonal method, we know that there exists a subsequence of $\{\tilde{F}_i\}$ (also denoted by $\{\tilde{F}_i\}$), which is compactly convergent to an $E$-valued holomorphic $(n,0)$ form $\tilde{F}$ on $X$. Then it follows from inequality \eqref{estimate for F_n} and Fatou's Lemma that
\begin{equation}\nonumber
  \begin{split}
  &\int_{X}|\tilde{F}-(1-b_{t_0,B}(\Psi))fF^{1+\delta}|^2_h e^{v_{t_0,B}(\Psi)-\delta \tilde{M}}c(-v_{t_0,B}(\Psi)) \\
 \le &\int_{X}|\tilde{F}-(1-b_{t_0,B}(\Psi))fF^{1+\delta}|^2_h e^{v_{t_0,B}(\Psi)-\delta \tilde{M}}c^-(-v_{t_0,B}(\Psi)) \\
      \le&\liminf_{i\to+\infty}\int_{X}|\tilde{F}_i-(1-b_{t_0,B}(\Psi))fF^{1+\delta}|^2_h e^{v_{t_0,B}(\Psi)-\delta \tilde{M}}c_i(-v_{t_0,B}(\Psi)) \\
      \le & \liminf_{i\to+\infty}\left(\frac{1}{\delta}c_i(T)e^{-T}+\int_{T}^{t_0+B}c_i(s)e^{-s}ds\right)
       \int_{X}\frac{1}{B}\mathbb{I}_{\{-t_0-B<\Psi<-t_0\}}|fF|^2_h\\
       = & \left(\frac{1}{\delta}c(T)e^{-T}+\int_{T}^{t_0+B}c(s)e^{-s}ds\right)
       \int_{X}\frac{1}{B}\mathbb{I}_{\{-t_0-B<\Psi<-t_0\}}|fF|^2_h.
  \end{split}
\end{equation}

\end{Remark}

In the following discussion, we assume that $F$ has no zero points on $X$ and $c(t)$ is smooth.

As $X$ is weakly pseudoconvex, there exists a smooth plurisubharmonic
exhaustion function $P$ on $X$. Let $X_j:=\{P<j\}$ $(k=1,2,...,) $. We choose $P$ such that
$X_1\ne \emptyset$.\par
Then $X_j$ satisfies $X_1 \Subset  X_2\Subset  ...\Subset
X_j\Subset  X_{j+1}\Subset  ...$ and $\cup_{j=1}^n X_j=X$. Each $X_j$ is weakly
pseudoconvex K\"ahler manifold with exhaustion plurisubharmonic function
$P_j=1/(j-P)$.
\par
\emph{We will fix $j$ during our discussion until step 7.}

\

\emph{Step 1: Regularization of $\Psi$. }

We note that there must exists a continuous nonnegative $(1,1)$-form $\varpi$ on $X_{j+1}$ satisfying
\begin{equation}\nonumber
(\sqrt{-1}\Theta_{TM}+\varpi \otimes Id_{TM})(\kappa_1 \otimes \kappa_2,\kappa_1
\otimes \kappa_2)\ge 0,
\end{equation}
for $\forall \kappa_1,\kappa_2 \in TM$ on $M_{j+1}$.

\par
Let $M=X_{j+1}$, $\Omega=X_{j}$, $T=\frac{\sqrt{-1}}{\pi}\partial\bar{\partial}\psi$
, $\gamma =0$ in Lemma \ref{regularization on cpx mfld}, then there exists a family of functions $\psi_{\zeta,\rho}$ ($\zeta\in(0,+\infty)$ and $\rho\in(0,\rho_1)$ for some positive $\rho_1$)
on $X_{j+1}$ such that\\
(1) $\psi_{\zeta,\rho}$ is a quasi-plurisubharmonic function on a neighborhood of $\overline{X_j}$, smooth on $X_{j+1}\backslash E_{\zeta}(\psi)$, increasing with respect to $\zeta$ and $\rho$ on $X_j$ and converges to $\psi$ on $X_j$ as $\rho \to 0$,\\
(2) $\frac{\sqrt{-1}}{\pi}\partial\bar{\partial}\psi_{\zeta,\rho}\ge-\zeta\varpi-\delta\omega$ on $X_j$,\\
where $E_{\zeta}(\psi):=\{x\in X: v(\psi,x)\ge \zeta\}$ is the upper-level set of Lelong number and $\{\delta_{\rho}\}$ is an increasing family of positive numbers such that $\lim_{\rho\to 0}\delta_{\rho}=0$.

Let $\rho=\frac{1}{m}$. Let $\tilde{\delta}_m:=\delta_{\frac{1}{m}}$ and $\zeta=\tilde{\delta}_m$. Denote $\psi_m:=\psi_{\tilde{\delta}_m,\frac{1}{m}}$. Then we have a sequence of functions $\{\psi_m\}$ satisfying\\
(1') $\psi_m$ is quasi-plurisubharmonic function on $\overline{X_j}$, smooth on $X_{j+1}\backslash E_{m}(\psi)$, decreasing with respect to $m$ and converges to $\psi$ on $X_j$ as $m\to+\infty$,\\
(2') $\frac{\sqrt{-1}}{\pi}\partial\bar{\partial}\psi_{m}\ge-\tilde{\delta}_m\varpi-\tilde{\delta}_m\omega$ on $X_j$,\\
where $E_{m}(\psi)=\{x\in X:v(\psi,x)\ge\frac{1}{m}\}$ is the upper level set of Lelong number and $\{\tilde{\delta}_m\}$ is an decreasing family of positive numbers such that $\lim_{m\to+\infty}\tilde{\delta}_m=0$.

As $X_j$ is relatively compact in $X$, there exists a positive number $b\ge 1$ such that $b\omega\ge\varpi$ on $X_j$. Then condition (2') becomes\\
(2'') $\frac{\sqrt{-1}}{\pi}\partial\bar{\partial}\psi_{m}\ge-\tilde{\delta}_m\varpi-\tilde{\delta}_m\omega
\ge-2b\tilde{\delta}_m\omega$ on $X_j$.

Let $\eta_m=\{\frac{t_0-T}{3m},\frac{t_0-T}{3m}\}$ and we have the function $M_{\eta_m}(\psi_m+T,2\log|F|)$. Denote $M_{\eta_m}:=M_{\eta_m}(\psi_m+T,2\log|F|)$ for simplicity.  Note that $\psi_m+T$ is a $2b\tilde{\delta}_m\omega$-plurisubharmonic function. As $F$ is a holomorphic function, $\omega$ is a K\"ahler form and $b\tilde{\delta}_m>0$, we know that $2\log|F|$ is a $2b\tilde{\delta}_m\omega$-plurisubharmonic function. It follows from Lemma \ref{regular of max} that $M_{\eta_m}$ is a $2b\tilde{\delta}_m\omega$-plurisubharmonic function, i.e.,
$$\frac{\sqrt{-1}}{\pi}\partial\bar{\partial}M_{\eta_m}\ge -2\pi b\tilde{\delta}_m\omega.$$

Denote $\Psi_m:=\psi_m-M_{\eta_m}(\psi_m+T,2\log|F|)$. Then $\Psi_m$ is smooth on $X_j\backslash E_m$. It is easy to verify that when $m\to+\infty$, $\Psi_m\to \Psi$. It follows from Lemma \ref{regular of max} that we know\par
(1) if $\psi_m+T\le 2\log|F|-\frac{2(t_0-T)}{3m}$ holds, we have $\Psi_m=\psi_m-2\log|F|$;\par
(2) if $\psi_m+T\ge 2\log|F|+\frac{2(t_0-T)}{3m}$ holds, we have $\Psi_m=-T$;\par
(3) if $2\log|F|-\frac{2(t_0-T)}{3m}<\psi_m+T< 2\log|F|+\frac{2(t_0-T)}{3m}$ holds, we have $\max\{\psi_m+T,2\log|F|\}\le M_{\eta_m}\le (\psi_m+T+\frac{t_0-T}{m})$ and hence $-T-\frac{t_0-T}{m}\le\Psi_m\le -T$.

Thus we have $\{\Psi_m<-t_0\}= \{\psi_m-2\log|F|<-t_0\}\subset \{\psi-2\log|F|<-t_0\}=\{\Psi<-t_0\}$. We also note that $\Psi_m\le-T$ on $M_{j+1}$.

\

\emph{Step 2: Recall some constructions. }

To simplify our notations, we denote $b_{t_0,B}(t)$ by $b(t)$ and $v_{t_0,B}(t)$ by $v(t)$.

Let $\epsilon \in (0,\frac{1}{8}B)$. Let $\{v_\epsilon\}_{\epsilon \in
(0,\frac{1}{8}B)}$ be a family of smooth increasing convex functions on $\mathbb{R}$, such
that:
\par
(1) $v_{\epsilon}(t)=t$ for $t\ge-t_0-\epsilon$, $v_{\epsilon}(t)=constant$ for
$t<-t_0-B+\epsilon$;\par
(2) $v_{\epsilon}{''}(t)$ are convergence pointwisely
to $\frac{1}{B}\mathbb{I}_{(-t_0-B,-t_0)}$,when $\epsilon \to 0$, and $0\le
v_{\epsilon}{''}(t) \le \frac{2}{B}\mathbb{I}_{(-t_0-B+\epsilon,-t_0-\epsilon)}$
for ant $t \in \mathbb{R}$;\par
(3) $v_{\epsilon}{'}(t)$ are convergence pointwisely to $b(t)$ which is a continuous
function on $\mathbb{R}$ when $\epsilon \to 0$ and $0 \le v_{\epsilon}{'}(t) \le 1$ for any
$t\in \mathbb{R}$.\par
One can construct the family $\{v_\epsilon\}_{\epsilon \in (0,\frac{1}{8}B)}$  by
 setting
\begin{equation}\nonumber
\begin{split}
v_\epsilon(t):=&\int_{-\infty}^{t}(\int_{-\infty}^{t_1}(\frac{1}{B-4\epsilon}
\mathbb{I}_{(-t_0-B+2\epsilon,-t_0-2\epsilon)}*\rho_{\frac{1}{4}\epsilon})(s)ds)dt_1\\
&-\int_{-\infty}^{-t_0}(\int_{-\infty}^{t_1}(\frac{1}{B-4\epsilon}
\mathbb{I}_{(-t_0-B+2\epsilon,-t_0-2\epsilon)}*\rho_{\frac{1}{4}\epsilon})(s)ds)dt_1-t_0,
\end{split}
\end{equation}
where $\rho_{\frac{1}{4}\epsilon}$ is the kernel of convolution satisfying
$\text{supp}(\rho_{\frac{1}{4}\epsilon})\subset
(-\frac{1}{4}\epsilon,{\frac{1}{4}\epsilon})$.
Then it follows that
\begin{equation}\nonumber
v_\epsilon{''}(t)=\frac{1}{B-4\epsilon}
\mathbb{I}_{(-t_0-B+2\epsilon,-t_0-2\epsilon)}*\rho_{\frac{1}{4}\epsilon}(t),
\end{equation}
and
\begin{equation}\nonumber
v_\epsilon{'}(t)=\int_{-\infty}^{t}(\frac{1}{B-4\epsilon}
\mathbb{I}_{(-t_0-B+2\epsilon,-t_0-2\epsilon)}*\rho_{\frac{1}{4}\epsilon})(s)ds.
\end{equation}

Let $\eta=s(-v_\epsilon(\Psi_m))$ and $\phi=u(-v_\epsilon(\Psi_m))$, where $s \in
C^{\infty}([T,+\infty))$ satisfies $s\ge \frac{1}{\delta}$ and $u\in C^{\infty}([T,+\infty))$, such that $s'(t)\neq 0$ for any $t$, $u''s-s''>0$
and $s'-u's=1$.
Let $\Phi_m=\phi+\delta M_{\eta_m}$. Recall that $(X,E,\Sigma,X_j,h,h_{j,m'})$ is a singular hermitian metric on $E$ and $\Theta_h(E)\ge^{m'}_{Nak}0$. Then there exists a sequence of hermitian metrics $\{h_{j,m'}\}_{m'=1}^n$ on $X_j$ of class $C^2$ such that $\lim\limits_{m'\to+\infty}h_{j,m'}=h$ almost everywhere on $X_j$ and $\{h_{j,m'}\}_{m'=1}^n$ satisfies the conditions of Definition \ref{singular nak}. Since $j$ is fixed until the last step, we simply denote $\{h_{j,m'}\}_{m'=1}^{+\infty}$ by $h_{m'}$ and denote $\tilde{h}:=h_{m'}e^{-\Phi_m}$.

\

\emph{Step 3: Solving $\bar{\partial}$-equation with error term. }

Set $B=[\eta \sqrt{-1}\Theta_{\tilde{h}}-\sqrt{-1}\partial \bar{\partial}
\eta\otimes\text{Id}_E-\sqrt{-1}g\partial\eta \wedge\bar{\partial}\eta\otimes\text{Id}_E, \Lambda_{\omega}]$, where
$g$
is a positive function. We will determine $g$ by calculations. On $X_j\backslash E_m$, direct calculation shows that
\begin{equation}\nonumber
\begin{split}
\partial\bar{\partial}\eta=&
-s'(-v_{\epsilon}(\Psi_m))\partial\bar{\partial}(v_{\epsilon}(\Psi_m))
+s''(-v_{\epsilon}(\Psi_m))\partial(v_{\epsilon}(\Psi_m))\wedge
\bar{\partial}(v_{\epsilon}(\Psi_m)),\\
\eta\Theta_{\tilde{h}}=&\eta\partial\bar{\partial}\phi\otimes\text{Id}_E+\eta\Theta_{h_{m'}}+
\eta\partial\bar{\partial}(\delta M_{\eta_m})\otimes\text{Id}_E\\
=&su''(-v_{\epsilon}(\Psi_m))\partial(v_{\epsilon}(\Psi_m))\wedge
\bar{\partial}(v_{\epsilon}(\Psi_m))\otimes\text{Id}_E
-su'(-v_{\epsilon}(\Psi_m))\partial\bar{\partial}(v_{\epsilon}(\Psi_m))\otimes\text{Id}_E\\
+&s\Theta_{h_{m'}}+
s\partial\bar{\partial}(\delta M_{\eta_m})\otimes\text{Id}_E.
\end{split}
\end{equation}
\par
Hence
\begin{equation}\nonumber
\begin{split}
&\eta \sqrt{-1}\Theta_{\tilde{h}}-\sqrt{-1}\partial \bar{\partial}
\eta\otimes\text{Id}_E-\sqrt{-1}g\partial\eta \wedge\bar{\partial}\eta\otimes\text{Id}_E\\
=&s\Theta_{h_{m'}}+
s\sqrt{-1}\partial\bar{\partial}(\delta M_{\eta_m})\otimes\text{Id}_E\\
+&(s'-su')(v'_{\epsilon}(\Psi_m)\sqrt{-1}\partial\bar{\partial}(\Psi_m)+
v''_\epsilon(\psi_m)\sqrt{-1}\partial(\Psi_m)\wedge\bar{\partial}(\Psi_m))\otimes\text{Id}_E\\
+&[(u''s-s'')-gs'^2]\sqrt{-1}\partial(v_\epsilon(\Psi_m))\wedge\bar{\partial}(v_\epsilon(\Psi_m))\otimes\text{Id}_E,
\end{split}
\end{equation}
where we omit the term $-v_{\epsilon}(\Psi_m)$ in $(s'-su')(-v_{\epsilon}(\Psi_m))$ and $[(u''s-s'')-gs'^2](-v_{\epsilon}(\Psi_m))$ for simplicity.
\par
Let $g=\frac{u''s-s''}{s'^2}(-v_\epsilon(\Psi_m))$ and note that $s'-su'=1$,
$0\le v'_\epsilon(\Psi_m)\le 1$. Then
\begin{equation}\label{calculation of curvature 1}
\begin{split}
&\eta \sqrt{-1}\Theta_{\tilde{h}}-\sqrt{-1}\partial \bar{\partial}
\eta\otimes\text{Id}_E-\sqrt{-1}g\partial\eta \wedge\bar{\partial}\eta\otimes\text{Id}_E\\
=&s\sqrt{-1}\Theta_{h_{m'}}+
s\sqrt{-1}\partial\bar{\partial}(\delta M_{\eta_m})\otimes\text{Id}_E\\
&+
v'_{\epsilon}(\Psi_m)\sqrt{-1}\partial\bar{\partial}(\Psi_m)\otimes\text{Id}_E+
v''_\epsilon(\psi_m)\sqrt{-1}\partial(\Psi_m)\wedge\bar{\partial}(\Psi_m)\otimes\text{Id}_E\\
=&v''_\epsilon(\psi_m)\sqrt{-1}\partial(\Psi_m)\wedge\bar{\partial}(\Psi_m)\otimes\text{Id}_E+
v'_{\epsilon}(\Psi_m)\sqrt{-1}\partial\bar{\partial}(\Psi_m)\otimes\text{Id}_E\\
&+s(\sqrt{-1}\Theta_{h_{m'}}+\lambda_{m'}\omega\otimes\text{Id}_E)-s\lambda_{m'}\omega\otimes\text{Id}_E\\
&+s(\sqrt{-1}\partial\bar{\partial}(\delta M_{\eta_m})+2\pi b\delta\tilde{\delta}_{m}\omega)\otimes\text{Id}_E-2\pi b s \delta\tilde{\delta}_{m}\omega\otimes\text{Id}_E\\
\ge&v''_\epsilon(\psi_m)\sqrt{-1}\partial(\Psi_m)\wedge\bar{\partial}(\Psi_m)\otimes\text{Id}_E+
v'_{\epsilon}(\Psi_m)\sqrt{-1}\partial\bar{\partial}(\Psi_m)\otimes\text{Id}_E\\
&+\frac{1}{\delta}(\sqrt{-1}\Theta_{h_{m'}}+\lambda_{m'}\omega\otimes\text{Id}_E)
+\frac{1}{\delta}(\sqrt{-1}\partial\bar{\partial}(\delta M_{\eta_m})+2\pi b\delta\tilde{\delta}_{m}\omega)\otimes\text{Id}_E\\
&-s\lambda_{m'}\omega\otimes\text{Id}_E-2\pi b s\delta\tilde{\delta}_{m}\omega\otimes\text{Id}_E.\\
\end{split}
\end{equation}
Note that
\begin{equation}\label{calculation of curvature 2}
\begin{split}
&\delta v'_{\epsilon}(\Psi_m)\sqrt{-1}\partial\bar{\partial}(\Psi_m)\otimes\text{Id}_E
+(\sqrt{-1}\Theta_{h_{m'}}+\lambda_{m'}\omega\otimes\text{Id}_E)\\
&+(\sqrt{-1}\partial\bar{\partial}(\delta M_{\eta_m})+2\pi b\delta\tilde{\delta}_{m}\omega)\otimes\text{Id}_E\\
=&(1-v'_{\epsilon}(\Psi_m))(\sqrt{-1}\Theta_{h_{m'}}+\lambda_{m'}\omega\otimes\text{Id}_E+\sqrt{-1}\partial\bar{\partial}(\delta M_{\eta_m})\otimes\text{Id}_E+2\pi b\delta\tilde{\delta}_{m}\omega\otimes\text{Id}_E)\\
&+v'_{\epsilon}(\Psi_m)(\sqrt{-1}\Theta_{h_{m'}}+\lambda_{m'}\omega\otimes\text{Id}_E+\sqrt{-1}\partial\bar{\partial}(\delta M_{\eta_m})\otimes\text{Id}_E+2\pi b\delta\tilde{\delta}_{m}\omega\otimes\text{Id}_E)\\
&+v'_{\epsilon}(\Psi_m)(\partial\bar{\partial}(\delta\psi_m)\otimes\text{Id}_E-\partial\bar{\partial}(\delta M_{\eta_m})\otimes\text{Id}_E)\\
=&(1-v'_{\epsilon}(\Psi_m))(\sqrt{-1}\Theta_{h_{m'}}+\lambda_{m'}\omega\otimes\text{Id}_E+\sqrt{-1}\partial\bar{\partial}(\delta M_{\eta_m})\otimes\text{Id}_E+2\pi b\delta\tilde{\delta}_{m}\omega\otimes\text{Id}_E)\\
&+v'_{\epsilon}(\Psi_m)(\sqrt{-1}\Theta_{h_{m'}}+\lambda_{m'}\omega\otimes\text{Id}_E+\sqrt{-1}\partial\bar{\partial}
(\delta \psi_m)\otimes\text{Id}_E+2\pi b\delta\tilde{\delta}_{m}\omega\otimes\text{Id}_E)\\
\ge& 0.
\end{split}
\end{equation}
It follows from inequality \eqref{calculation of curvature 1} and inequality \eqref{calculation of curvature 2} that
\begin{equation}\nonumber
\begin{split}
&\eta \sqrt{-1}\Theta_{\tilde{h}}-\sqrt{-1}\partial \bar{\partial}
\eta\otimes\text{Id}_E-\sqrt{-1}g\partial\eta \wedge\bar{\partial}\eta\otimes\text{Id}_E\\
\ge&v''_\epsilon(\Psi_m)\sqrt{-1}\partial(\Psi_m)\wedge\bar{\partial}(\Psi_m)\otimes\text{Id}_E
-2\pi b s\delta\tilde{\delta}_{m}\omega\otimes\text{Id}_E-s\lambda_{m'}\omega\otimes\text{Id}_E.
\end{split}
\end{equation}

By the constructions of $s(t)$, $v_{\epsilon}(t)$ and $\sup_m\sup_{X_j}\Psi_m\le-T$, we have $s(-v_{\epsilon}(\Psi_m))$ is uniformly bounded on $X_j$ with respect to $\epsilon$ and $m$. Let $N_1$ be the uniformly upper bound of $s(-v_{\epsilon}(\Psi_m))$ on $X_j$. Then on $X_j\backslash E_m$, we have

\begin{equation}\nonumber
\begin{split}
&\eta \sqrt{-1}\Theta_{\tilde{h}}-\sqrt{-1}\partial \bar{\partial}
\eta\otimes\text{Id}_E-\sqrt{-1}g\partial\eta \wedge\bar{\partial}\eta\otimes\text{Id}_E\\
\ge&v''_\epsilon(\Psi_m)\sqrt{-1}\partial(\Psi_m)\wedge\bar{\partial}(\Psi_m)\otimes\text{Id}_E
-2\pi b N_1\delta\tilde{\delta}_{m}\omega\otimes\text{Id}_E-N_1\lambda_{m'}\omega\otimes\text{Id}_E.
\end{split}
\end{equation}

Hence, for any $E$-valued $(n,1)$ form $\alpha$, we have
\begin{equation}
\label{curvature inequality}
\begin{split}
&\langle(B+(2\pi bN_1 \delta\tilde{\delta}_{m}+N_1 \lambda_{m'})\text{Id}_E)\alpha,\alpha\rangle_{\tilde h}\\
\ge&\langle[v''_\epsilon(\Psi_m)\partial(\Psi_m)\wedge\bar{\partial}(\Psi_m)\otimes\text{Id}_E,
\Lambda_{\omega}]\alpha,\alpha\rangle_{\tilde h}\\
=&\langle(v''_\epsilon(\Psi_m)\bar{\partial}(\Psi_m)
\wedge(\alpha\llcorner(\bar{\partial}\Psi_m)^{\sharp})),\alpha\rangle_{\tilde
h}.
\end{split}
\end{equation}
It follows from Lemma \ref{sempositive lemma} that $B+(2\pi bN_1 \delta\tilde{\delta}_{m}+N_1 \lambda_{m'})\text{Id}_E$ is semi-positive. Denote $\tilde{\lambda}_{m'}:=\lambda_{m'}+\frac{1}{m'}$, then $B+(2\pi bN_1 \delta\tilde{\delta}_{m}+N_1 \tilde{\lambda}_{m'})\text{Id}_E$ is positive. Using the definition of contraction, Cauchy-Schwarz inequality
and inequality \eqref{curvature inequality}, we have
\begin{equation}
\label{cs inequality}
\begin{split}
|\langle
v''_\epsilon(\Psi_m)\bar{\partial}\Psi_m\wedge\gamma,\tilde{\alpha}\rangle_
{\tilde h}|^2=
&|\langle
v''_\epsilon(\Psi_m)\gamma,\tilde{\alpha}\llcorner(\bar{\partial}\Psi_m)^{\sharp}
\rangle_{\tilde h}|^2\\
\le&\langle
(v''_\epsilon(\Psi_m)\gamma,\gamma)
\rangle_{\tilde h}
(v''_\epsilon(\Psi_m))|\tilde{\alpha}\llcorner(\bar{\partial}\Psi_m)^{\sharp}|^2_{\tilde
h}\\
=&\langle
(v''_\epsilon(\Psi_m)\gamma,\gamma)
\rangle_{\tilde h}
\langle
(v''_\epsilon(\Psi_m))\bar{\partial}\Psi_m\wedge
(\tilde{\alpha}\llcorner(\bar{\partial}\Psi_m)^{\sharp}),\tilde{\alpha}
\rangle_{\tilde h}\\
\le&\langle
(v''_\epsilon(\Psi_m)\gamma,\gamma)
\rangle_{\tilde h}
\langle
(B+(2\pi bN_1 \delta\tilde{\delta}_{m}+N_1 \tilde{\lambda}_{m'})\text{Id}_E)\tilde{\alpha},\tilde{\alpha})
\rangle_{\tilde h}
\end{split}
\end{equation}
for any $E$-valued $(n,0)$ form $\gamma$ and $E$-valued $(n,1)$ form $\tilde{\alpha}$.

As $fF^{1+\delta}$ is holomorphic on $\{\Psi<-t_0\}$ and $\{\Psi_m<-t_0-\epsilon\}\subset \{\Psi_m<-t_0\}\subset\{\Psi<-t_0\}$, then $\lambda:=\bar{\partial}\big((1-v'_{\epsilon}(\Psi_m))fF^{1+\delta}\big)$ is well defined and smooth on $X_j\backslash E_m$.

Taking $\gamma=fF^{1+\delta}$, $\tilde{\alpha}=(B+(2\pi bN_1 \delta\tilde{\delta}_{m}+N_1 \tilde{\lambda}_{m'})\text{Id}_E)^{-1}(\bar{\partial}v'_{\epsilon}(\Psi_m))\wedge fF^{1+\delta}$. Then it follows from inequality \eqref{cs inequality} that
$$\langle
(B+(2\pi bN_1 \delta\tilde{\delta}_{m}+N_1 \tilde{\lambda}_{m'})\text{Id}_E)^{-1}\lambda,\lambda\rangle_
{\tilde h}\le v''_\epsilon(\Psi_m)|fF^{1+\delta}|^2_{\tilde{h}}.$$

Thus we have
$$\int_{X_j\backslash E_m}\langle
(B+(2\pi bN_1 \delta\tilde{\delta}_{m}+N_1 \tilde{\lambda}_{m'})\text{Id}_E)^{-1}\lambda,\lambda\rangle_
{\tilde h}\le \int_{X_j\backslash E_m}v''_\epsilon(\Psi_m)|fF^{1+\delta}|^2_{\tilde{h}}.$$

Recall that $\tilde{h}=h_{m'}e^{-\Phi_m}$ and $\Phi_m=\phi+\delta M_{\eta_m}$. As $h_{m'}$ is $C^2$ smooth, on $X_j\subset\subset X$, there exists a constant $C_{j,m'}> 0$ such that $C_{j,m'}^{-1}|e_x|_{\hat{h}}\le |e_x|_{h_{m'}}\le C_{j,m'} |e_x|_{\hat{h}},$ for any $e_x\in E_x$. Note that $0\le
v_{\epsilon}^{''}(t) \le \frac{2}{B}\mathbb{I}_{(-t_0-B+\epsilon,-t_0-\epsilon)}$, $e^{-\phi}$ is smooth function on $X_j$ and $\delta M_{\eta_m}\ge \delta 2\log|F|$. It follows from \eqref{estimate on sublevel set of lemma2.1} that
$$\int_{X_j\backslash E_m}v''_\epsilon(\Psi_m)|fF^{1+\delta}|^2_{\tilde{h}}\le C_{j,m'} \sup_{X_k}(|F|^2e^{-\phi})\int_{\overline{X_j}\cap\{\Psi<-t_0\}}|f|_{\hat{h}}^2<+\infty.$$

By Lemma \ref{completeness}, $X_j\backslash E_m$ carries a complete K\"ahler metric. Then it follows from Lemma \ref{d-bar equation with error term} that there exists
$$u_{m,m',\epsilon,j}\in L^2(X_j\backslash E_m, K_X\otimes E,h_{m'}e^{-\Phi_m}),$$
$$h_{m,m',\epsilon,j}\in L^2(X_j\backslash E_m, \wedge^{n,1}T^*X \otimes E,h_{m'}e^{-\Phi_m}),$$
such that $\bar\partial u_{m,m',\epsilon,j}+P_{m,m'}\big(\sqrt{2\pi bN_1\delta\tilde\delta_m+N_1 \tilde{\lambda}_{m'}}h_{m,m',\epsilon,j}\big)=\lambda$ holds on $X_j\backslash E_m$ where $P_{m,m'}:L^2(X_j\backslash E_m, \wedge^{n,1}T^*X \otimes E,h_{m'}e^{-\Phi_m})\to \text{Ker}{D''}$ is the orthogonal projection, and

\begin{equation}\nonumber
\begin{split}
&\int_{X_j \backslash E_m}
\frac{1}{\eta+g^{-1}}|u_{m,m',\epsilon,j}|^2_{h_{m'}}e^{-\Phi_m}+
\int_{X_j \backslash E_m}
|h_{m,m',\epsilon,j}|^2_{h_{m'}}e^{-\Phi_m} \\
\le&
\int_{X_j \backslash E_m}
\langle
(B+(2\pi bN_1\tilde \delta\delta_{m}+N_1 \tilde{\lambda}_m') \text{Id}_E)^{-1}\lambda,\lambda
\rangle_{\tilde h}\\
\le&
\int_{X_j \backslash E_m}
v''_{\epsilon}(\Psi_m)|fF^{1+\delta}|^2_{h_{m'}}e^{-\Phi_m}
<+\infty.
\end{split}
\end{equation}

Assume that we can choose
$\eta$ and $\phi$ such that
$(\eta+g^{-1})^{-1}=e^{v_\epsilon(\Psi_m)}e^{\phi}c(-v_\epsilon(\Psi_m))$. Then we have

\begin{equation}
\label{estimate 1}
\begin{split}
&\int_{X_j \backslash E_m}
|u_{m,m',\epsilon,j}|^2_{h_{m'}} e^{v_\epsilon(\Psi_m)-\delta M_{\eta_m}}c(-v_\epsilon(\Psi_m))+
\int_{X_j \backslash E_m}
|h_{m,m',\epsilon,j}|^2_{h_{m'}} e^{-\phi-\delta M_{\eta_m}}\\
\le&
\int_{X_j \backslash E_m}
v''_{\epsilon}(\Psi_m)|fF^{1+\delta}|^2_{h_{m'}} e^{-\phi-\delta M_{\eta_m}}
<+\infty.
\end{split}
\end{equation}

By the construction of $v_{\epsilon}(t)$ and $c(t)e^{-t}$ is decreasing with respect to $t$, we know $c(-v_{\epsilon}(\Psi_m))e^{v_{\epsilon}(\Psi_m)}$ has a positive lower bound on $X_j\Subset X$. By the constructions of $v_{\epsilon}(t)$ and $u$, we know $e^{-\phi}=e^{-u(-v_{\epsilon}(\Psi_m))}$ has a positive lower bound on $X_j\Subset X$. By the upper semi-continuity of $M_{\eta_m}$, we know $e^{-\delta M_{\eta_m}}$ has a positive lower bound on $X_j\Subset M$. Note that $h_{m'}$ is $C^2$ smooth on $X_j\Subset X$. Hence it follows from inequality \eqref{estimate 1} that

$$u_{m,m',\epsilon,j}\in L^2(X_j, K_M\otimes E,h_{m'}e^{-\Phi_m}),$$
$$h_{m,m',\epsilon,j}\in L^2(X_j, \wedge^{n,1}T^*M\otimes E, h_{m'}e^{-\Phi_m}).$$
It follows from Lemma \ref{extension of equality} that we know
\begin{equation}
\label{d-bar realation u,h,lamda 1}
D''u_{m,m',\epsilon,j}+P_{m,m'}\big(\sqrt{2\pi bN_1\delta\tilde\delta_m+N_1 \tilde{\lambda}_{m'}}h_{m,m',\epsilon,j}\big)=\lambda
\end{equation}
holds on $X_j$. And we have
\begin{equation}
\label{estimate 2}
\begin{split}
&\int_{X_j }
|u_{m,m',\epsilon,j}|^2_{h_{m'}} e^{v_\epsilon(\Psi_m)-\delta M_{\eta_m}}c(-v_\epsilon(\Psi_m))+
\int_{X_j}
|h_{m,m',\epsilon,j}|^2_{h_{m'}} e^{-\phi-\delta M_{\eta_m}}\\
\le&
\int_{X_j}
v''_{\epsilon}(\Psi_m)|fF^{1+\delta}|^2_{h_{m'}} e^{-\phi-\delta M_{\eta_m}}
<+\infty.
\end{split}
\end{equation}

\

\emph{Step 4: Letting $m\to+\infty$. }

In the Step 4, we note that $m'$ is fixed.

Note that $\sup_m\sup_{X_j}e^{-\phi}=\sup_m\sup_{X_j}e^{-u(-v_{\epsilon}(\Psi_m))}<+\infty$ and $e^{-\delta M_{\eta_m}}\le e^{-\delta 2\log|F|}$. As $\{\Psi_m<-t_0-\epsilon\}\subset \{\Psi_m<-t_0\}\subset \{\Psi<-t_0\}$, we have
$$v''_{\epsilon}(\Psi_m)|fF^{1+\delta}|^2_{h_{m'}} e^{-\phi-\delta M_{\eta_m}}\le \frac{2}{B}C_{j,m'}\left(\sup_m\sup_{X_j}e^{-\phi}\right)\left(\sup_{X_j}|F|^2\right)\mathbb{I}_{\{\Psi<-t_0\}}|f|^2_{\hat{h}}$$
holds on $M_j$. It follows from $\int_{\{\Psi<-t_0\}\cap \overline{M_j}}|f|^2_{\hat{h}}<+\infty$ and dominated convergence theorem that

\begin{equation}\nonumber
\begin{split}
&\lim_{m\to+\infty}\int_{X_j}
v''_{\epsilon}(\Psi_m)|fF^{1+\delta}|^2_{h_{m'}} e^{-\phi-\delta M_{\eta_m}}\\
=&\int_{X_j}v''_{\epsilon}(\Psi)|fF^{1+\delta}|^2_{h_{m'}} e^{-u(-v_{\epsilon}(\Psi))-\delta \max{\{\psi+T,2\log|F|\}}}.
\end{split}
\end{equation}

Note that $\inf_{m}\inf_{X_j}c(-v_{\epsilon}(\Psi_m))e^{-v_{\epsilon}(\Psi_m)}>0$.
It follows from Lemma \ref{regular of max} that $M_{\eta_m}\le \max{\{\psi_m+T,2\log|F|\}}+\frac{t_0-T}{3m}\le \max{\{\psi_m+T,2\log|F|\}}+t_0-T\le \max{\{\psi_1+T,2\log|F|\}}+t_0-T$. As $\psi_1$ is a quasi-plurisubharmonic function on $\overline{X_j}$, we know $\max{\{\psi_1+T,2\log|F|\}}$ is upper semi-continuous function on $X_j$. Hence
\begin{equation}
\label{estimate for M eta}
\inf_{m}\inf_{X_j}e^{-M_{\eta_m}}\ge \inf_{X_j}e^{-\max{\{\psi_1+T,2\log|F|\}}-t_0}>0.
\end{equation}
 Then it follows from inequality \eqref{estimate 2} that
$$\sup_{m}\int_{X_j}|u_{m,m',\epsilon,j}|^2_{h_{m'}}<+\infty.$$

Therefore the solutions $u_{m,m',l,\epsilon,j}$ are uniformly bounded with respect to $m$ in $L^2(X_j,K_M, h_{m'})$. Since the closed unit ball of the Hilbert space is
weakly compact, we can extract a subsequence of $\{u_{m,m',\epsilon,j}\}$ (also denoted by $\{u_{m,m',\epsilon,j}\}$) weakly
convergent to $u_{m',\epsilon,j}$ in $L^2(X_j,K_M, h_{m'})$ as $m\to+\infty$.

Note that $\sup_m\sup_{M_j}e^{v_{\epsilon}(\Psi_m)}c(-v_{\epsilon}(\Psi_m))< +\infty$. As $M_{\eta_m}\ge \max\{\psi_m+T,2\log|F|\}\ge 2\log|F|$ and $F$ has no zero points on $M$, we have $\sup_m\sup_{M_j}e^{-M_{\eta_m}}\le \sup_{M_j}\frac{1}{|F|^2}<+\infty$. Hence we know
$$\sup_m\sup_{M_j}e^{v_{\epsilon}(\Psi_m)}c(-v_{\epsilon}(\Psi_m))e^{-\delta M_{\eta_m}}<+\infty.$$
It follows from Lemma \ref{weakly convergence} that we know $u_{m,m',\epsilon,j}\sqrt{e^{v_{\epsilon}(\Psi_{m_1})}c(-v_{\epsilon}(\Psi_{m}))e^{-\delta M_{\eta_{m}}}}$ weakly converges to $u_{m',\epsilon,j}\sqrt{e^{v_{\epsilon}(\Psi)}c(-v_{\epsilon}(\Psi))e^{-\delta \max\{\psi+T,2\log|F|\}}}$ in $L^2(X_j,K_M, h_{m'})$ as $m\to+\infty$ . Hence we have
\begin{equation}
\label{estimate for m'}
\begin{split}
&\int_{X_j}
|u_{m',\epsilon,j}|^2_{h_{m'}} e^{v_\epsilon(\Psi)-\delta \max\{\psi+T,2\log|F|\}}c(-v_\epsilon(\Psi))\\
\le&\liminf_{m\to+\infty}\int_{X_j}
|u_{m,m',\epsilon,j}|^2_{h_{m'}} e^{v_\epsilon(\Psi_{m})-\delta M_{\eta_{m}}}c(-v_\epsilon(\Psi_{m}))\\
\le&\liminf_{m\to+\infty}
\int_{X_j}
v''_{\epsilon}(\Psi_{m})|fF^{1+\delta}|^2_{h_{m'}} e^{-u(-v_{\epsilon}(\Psi_{m}))-\delta M_{\eta_{m}}}\\
\le & \int_{X_j}v''_{\epsilon}(\Psi)|fF^{1+\delta}|^2_{h_{m'}} e^{-u(-v_{\epsilon}(\Psi))-\delta \max{\{\psi+T,2\log|F|\}}}
<+\infty.
\end{split}
\end{equation}

It follows from Lemma \ref{weakly converge lemma} that we know that $h_{m,m',\epsilon,j}$ weakly converges to $h_{m',\epsilon,j}$ in $L^2(X_j, \wedge^{n,1}T^*M\otimes E, h_{m'}e^{-\Phi_1})$ and then
$\sqrt{2\pi bN_1\delta\tilde\delta_m+N_1 \tilde{\lambda}_{m'}}h_{m,m',\epsilon,j}$ weakly converges to $\sqrt{N_1 \tilde{\lambda}_{m'}}h_{m',\epsilon,j}$ in $L^2(X_j, \wedge^{n,1}T^*M\otimes E, h_{m'}e^{-\Phi_1})$. Hence by Lemma \ref{weakly converge lemma} and the uniqueness of weak limit, we know that
$P_{m,m'}\big(\sqrt{2\pi bN_1\delta\tilde\delta_m+N_1 \tilde{\lambda}_{m'}}h_{m,m',\epsilon,j}\big)$ weakly converges to $P_{m'}\big(\sqrt{N_1 \tilde{\lambda}_{m'}}h_{m',\epsilon,j}\big)$ in $L^2(X_j, \wedge^{n,1}T^*M\otimes E, h_{m'}e^{-\Phi_1})$.

It follows from $\inf_{X_j}e^{-\Phi_1}=\inf_{X_j}e^{-u(-v_{\epsilon}(\Psi_{1}))}>0$ and inequality \eqref{estimate for M eta} that we have $h_{m,m',\epsilon,j}$ weakly converges to $h_{m',\epsilon,j}$ in $L^2(X_j, \wedge^{n,1}T^*M\otimes E, h_{m'})$ and $P_{m,m'}\big(\sqrt{2\pi bN_1\delta\tilde\delta_m+N_1 \tilde{\lambda}_{m'}}h_{m,m',\epsilon,j}\big)$ weakly converges to $P_{m'}\big(\sqrt{N_1 \tilde{\lambda}_{m'}}h_{m',\epsilon,j}\big)$ in $L^2(X_j, \wedge^{n,1}T^*M\otimes E, h_{m'})$.

Note that $\sup_{m}\sup_{X_j}e^{-u(-v_{\epsilon}(\Psi_{m}))}<+\infty$ and $\sup_{m}\sup_{X_j}e^{-M_{\eta_{m}}}\le \sup_{M_j}\frac{1}{|F|^2}<+\infty$. We know $$\sup_{m}\sup_{X_j}e^{-u(-v_{\epsilon}(\Psi_{m}))-\delta M_{\eta_{m}}}<+\infty.$$
It follows from Lemma \ref{weakly convergence} that we have
$h_{m,m',l,\epsilon,j}\sqrt{e^{-u(-v_{\epsilon}(\Psi_{m}))-\delta M_{\eta_{m}}}}$ is weakly convergent to
$h_{m',l,\epsilon,j}\sqrt{e^{-u(-v_{\epsilon}(\Psi))-\delta \max{\{\psi+T,2\log|F|\}}}}$ in $L^2(X_j,\wedge^{n,1}T^*M\otimes E, h_{m'})$ as $m\to+\infty$. Hence we have

\begin{equation}
\label{estimate for h m'}
\begin{split}
&\int_{X_j}
|h_{m',\epsilon,j}|^2_{h_{m'}}e^{-u(-v_{\epsilon}(\Psi))-\delta \max{\{\psi+T,2\log|F|\}}}\\
\le&\liminf_{m\to+\infty}\int_{M_j}
|h_{m,m',\epsilon,j}|^2_{h_{m'}} e^{-u(-v_{\epsilon}(\Psi_{m}))-\delta M_{\eta_{m}}}\\
\le&\liminf_{m\to+\infty}
\int_{X_j}
v''_{\epsilon}(\Psi_{m})|fF^{1+\delta}|^2_{h_{m'}} e^{-u(-v_{\epsilon}(\Psi_{m}))-\delta M_{\eta_{m}}}\\
\le & \int_{X_j}v''_{\epsilon}(\Psi)|fF^{1+\delta}|^2_{h_{m'}} e^{-u(-v_{\epsilon}(\Psi))-\delta \max{\{\psi+T,2\log|F|\}}}
<+\infty.
\end{split}
\end{equation}

Letting $m\to +\infty$ in \eqref{d-bar realation u,h,lamda 1}, we have
\begin{equation}
\label{d-bar realation u,h,lamda 2}
D''u_{m',\epsilon,j}+P_{m'}\big(\sqrt{N_1 \tilde{\lambda}_{m'}}h_{m',\epsilon,j}\big)
=D''\left((1-v'_{\epsilon}(\Psi))fF^{1+\delta}\right).
\end{equation}

\

\emph{Step 5: Letting $m'\to+\infty$. }

When $\Psi<-t_0-\epsilon<-t_0$, we have $\psi-2\log|F|<-T$ and then $\max{\{\psi+T,2\log|F|\}}=2\log|F|$. Hence
$$\int_{X_j}v''_{\epsilon}(\Psi)|fF^{1+\delta}|^2_{h_{m'}} e^{-u(-v_{\epsilon}(\Psi))-\delta \max{\{\psi+T,2\log|F|\}}}=\int_{X_j}v''_{\epsilon}(\Psi)|fF|^2_{h_{m'}} e^{-u(-v_{\epsilon}(\Psi))}.$$

Note that $\sup_{X_j}(e^{-u(-v_{\epsilon}(\Psi))})<+\infty$,
$0\le
v_{\epsilon}{''}(t) \le \frac{2}{B}\mathbb{I}_{(-t_0-B+\epsilon,-t_0-\epsilon)}$ and $|e_x|_{h_{m'}}\le |e_x|_{h_{m'+1}}\le |e_x|_{h}$ for any $m'\in \mathbb{Z}_{\ge 0}$. We have
\begin{equation}
\label{upper bound for rhs 1}
v''_{\epsilon}(\Psi)|fF|^2_{h_{m'}} e^{-u(-v_{\epsilon}(\Psi))}\le \sup_{X_j}(e^{-u(-v_{\epsilon}(\Psi))})\frac{2}{B}\mathbb{I}_{\{-t_0-B+\epsilon<\Psi<-t_0-\epsilon\}} |fF|_h^2.
\end{equation}

It follows from inequality \eqref{estimate on annaul of lemma2.1} and dominated convergence theorem that we have
\begin{equation}\nonumber
\begin{split}
&\lim_{m'\to+\infty}\int_{X_j}
v''_{\epsilon}(\Psi)|fF|^2_{h_{m'}} e^{-u(-v_{\epsilon}(\Psi))}\\
=&\int_{X_j}v''_{\epsilon}(\Psi)|fF|^2_h e^{-u(-v_{\epsilon}(\Psi))}<+\infty.
\end{split}
\end{equation}

Let $C_j:=\inf_{X_j}e^{v_\epsilon(\Psi)-\delta \max\{\psi+T,2\log|F|\}}c(-v_\epsilon(\Psi))$ be a real number and we note that $C_j>0$. Then it follows from $C_j>0$, inequalities \eqref{estimate for m'}, \eqref{upper bound for rhs 1} and \eqref{estimate on annaul of lemma2.1} that we have
$$\sup_{m'} \int_{X_j}
|u_{m',\epsilon,j}|^2_{h_{m'}}<+\infty.$$
As $|e_x|_{h_{m'}}\le |e_x|_{h_{m'+1}}$ for any $m'\in \mathbb{Z}_{\ge 0}$, for any fixed $i$, we have
$$\sup_{m'} \int_{X_j}
|u_{m',\epsilon,j}|^2_{h_{i}}<+\infty.$$
Especially letting $h_i=h_1$, since the closed unit ball of the Hilbert space is
weakly compact, we can extract a subsequence $u_{m'',\epsilon,j}$ weakly
convergent to $u_{\epsilon,j}$ in $L^2(M_j,K_M\otimes E, h_1)$ as $m''\to+\infty$. Note that $$\sup_{X_j}e^{v_\epsilon(\Psi)-\delta \max\{\psi+T,2\log|F|\}}c(-v_\epsilon(\Psi))
\le \sup_{X_j}\frac{e^{v_\epsilon(\Psi)}c(-v_\epsilon(\Psi))}{|F|^{2\delta}}<+\infty.$$
It follows from Lemma \ref{weakly convergence} that $u_{m'',\epsilon,j}\sqrt{e^{v_{\epsilon}(\Psi)}c(-v_{\epsilon}(\Psi))e^{-\delta \max\{\psi+T,2\log|F|\}}}$ weakly converges to $u_{\epsilon,j}\sqrt{e^{v_{\epsilon}(\Psi)}c(-v_{\epsilon}(\Psi))e^{-\delta \max\{\psi+T,2\log|F|\}}}$ in $L^2(X_j,K_X\otimes E, h_{1})$ as $m''\to+\infty$ .

For fixed $i\in\mathbb{Z}_{\ge 0}$, as $h_1$ and $h_i$ are both $C^2$ smooth hermitian metrics on $X_k$ and $X_j\subset\subset X$, we know that the two norms in $L^2(X_j, K_X\otimes E, h_1)$ and $ L^2(X_j, K_X\otimes E, h_i)$ are equivalent. Note that $\sup_{m''} \int_{X_j}
|u_{m'',\epsilon,j}|^2_{h_{i}}<+\infty$. Hence we know that $u_{m'',\epsilon,j}\sqrt{e^{v_{\epsilon}(\Psi)}c(-v_{\epsilon}(\Psi))e^{-\delta \max\{\psi+T,2\log|F|\}}}$ also weakly converges to $u_{\epsilon,j}\sqrt{e^{v_{\epsilon}(\Psi)}c(-v_{\epsilon}(\Psi))e^{-\delta \max\{\psi+T,2\log|F|\}}}$ in $L^2(X_j,K_X\otimes E, h_{i})$ as $m''\to+\infty$.
Then we have
\begin{equation}\nonumber
\begin{split}
&\int_{X_j}|u_{\epsilon,j}|^2_{h_i}e^{v_\epsilon(\Psi)-\delta \max\{\psi+T,2\log|F|\}}c(-v_\epsilon(\Psi))\\
\le& \liminf_{m''\to+\infty}\int_{X_j}|u_{m'',\epsilon,j}|^2_{h_i}e^{v_\epsilon(\Psi)-\delta \max\{\psi+T,2\log|F|\}}c(-v_\epsilon(\Psi))\\
\le& \liminf_{m''\to+\infty}\int_{X_j}v''_{\epsilon}(\Psi)|fF^{1+\delta}|^2_{h_{m''}} e^{-u(-v_{\epsilon}(\Psi))-\delta \max{\{\psi+T,2\log|F|\}}}\\
\le&
\int_{X_j}v''_{\epsilon}(\Psi)|fF|^2_h e^{-u(-v_{\epsilon}(\Psi))}
<+\infty.
\end{split}
\end{equation}
Letting $i\to +\infty$, by monotone convergence theorem, we have
\begin{equation}
\begin{split}
\int_{X_j}|u_{\epsilon,j}|^2_{h}e^{v_\epsilon(\Psi)-\delta \max\{\psi+T,2\log|F|\}}c(-v_\epsilon(\Psi))
\le
\int_{X_j}v''_{\epsilon}(\Psi)|fF|^2_h e^{-u(-v_{\epsilon}(\Psi))}
<+\infty.
\label{estimate for uej}
\end{split}
\end{equation}

Let $\tilde{C}_j:=\inf_{X_j}e^{-u(-v_{\epsilon}(\Psi))-\delta\max\{\psi+T,2\log|F|\}}$ and note that $\tilde{C}_j$ is a positive number.
Then it follows from $\tilde{C}_j>0$, inequalities \eqref{estimate for h m'}, \eqref{upper bound for rhs 1} and \eqref{estimate on annaul of lemma2.1} that we have
$$\sup_{m''} \int_{X_j}
|h_{m'',\epsilon,j}|^2_{h_{m''}}<+\infty.$$
As $|e_x|_{h_{m'}}\le |e_x|_{h_{m'+1}}$ for any $m'\in \mathbb{Z}_{\ge 0}$, for $h_1$, we have
$$\sup_{m''} \int_{X_j}
|h_{m'',\epsilon,j}|^2_{h_{1}}<+\infty.$$

Since the closed unit ball of the Hilbert space is
weakly compact, we can extract a subsequence of $\{h_{m'',\epsilon,j}\}$ (also denote by $h_{m'',\epsilon,j}$) weakly
convergent to $h_{\epsilon,j}$ in $L^2(X_j,\wedge^{n,1}T^*M\otimes E, h_1)$ as $m''\to+\infty$.
As $0\le \tilde{\lambda}_{m''}\le\lambda+1$ and $X_j$ is relatively compact in $X$, we know that
$$\sup_{m''} \int_{X_j} N_1 \tilde{\lambda}_{m''}
|h_{m'',\epsilon,j}|^2_{h_{m''}}<+\infty.$$
It follows from Lemma \ref{weakly converge lemma} that we know that $\sqrt{N_1\tilde{\lambda}_{m''}} h_{m'',\epsilon,j}$ weakly converges to some $\tilde{h}_{\epsilon,j}$ and $P_{m'}\big(\sqrt{N_1\tilde{\lambda}_{m'}}h_{m',\epsilon,j}\big)$ weakly converges to $P\big(\tilde{h}_{\epsilon,j}\big)$ in $L^2(X_j,\wedge^{n,1}T^*M\otimes E, h_1)$.

It follows from  $0\le \tilde{\lambda}_{m''}\le\lambda+1$, $X_j$ is relatively compact in $X$ and Lemma \ref{weakly convergence} that we know $\sqrt{N_1\tilde{\lambda}_{m''}}h_{m'',\epsilon,j}$ weakly
convergent to $0$ in $L^2(X_j,\wedge^{n,1}T^*M\otimes E, h_1)$. It follows from the uniqueness of weak limit that we know $\tilde{h}_{\epsilon,j}=0$. Then we have $P_{m'}\big(\sqrt{N_1\tilde{\lambda}_{m'}}h_{m',\epsilon,j}\big)$ weakly converges to $0=P\big(\tilde{h}_{\epsilon,j}\big)$ in $L^2(X_j,\wedge^{n,1}T^*M\otimes E, h_1)$

Replace $m'$ by $m''$ in \eqref{d-bar realation u,h,lamda 2} and let $m''$ go to $+\infty$, we have
\begin{equation}
\label{d-bar realation u,h,lamda 3}
D'' u_{\epsilon,j}
=D''\left((1-v'_{\epsilon}(\Psi))fF^{1+\delta}\right).
\end{equation}

Denote $F_{\epsilon,j}:=-u_{\epsilon,j}+(1-v'_{\epsilon}(\Psi))fF^{1+\delta}$. It follows from \eqref{d-bar realation u,h,lamda 3} and inequality \eqref{estimate for uej} that we know $F_{\epsilon,j}$ is an $E$-valued holomorphic $(n,0)$ form on $X_j$ and
\begin{equation}
\label{estimate for F lej}
\begin{split}
&\int_{X_j}
|F_{\epsilon,j}-(1-v'_{\epsilon}(\Psi))fF^{1+\delta}|^2_h e^{v_\epsilon(\Psi)-\delta \max\{\psi+T,2\log|F|\}}c(-v_\epsilon(\Psi))\\
\le&\int_{X_j}v''_{\epsilon}(\Psi)|fF|^2_h e^{-u(-v_{\epsilon}(\Psi))}<+\infty.
\end{split}
\end{equation}

\

\emph{Step 6: Letting $\epsilon\to 0$. }

Note that $\sup_{\epsilon}\sup_{X_j}(e^{-u(-v_{\epsilon}(\Psi))})<+\infty$,
$0\le
v_{\epsilon}{''}(t) \le \frac{2}{B}\mathbb{I}_{(-t_0-B+\epsilon,-t_0-\epsilon)}$. We have
\begin{equation}
\label{upper bound for rhs 2 in L2 method}
v''_{\epsilon}(\Psi)|fF|^2_{h} e^{-u(-v_{\epsilon}(\Psi))}\le \sup_{\epsilon}\sup_{X_j}(e^{-u(-v_{\epsilon}(\Psi))})\frac{2}{B}\mathbb{I}_{\{-t_0-B<\Psi<-t_0\}} |fF|_h^2.
\end{equation}
It follows from inequality \eqref{estimate on annaul of lemma2.1} and dominated convergence theorem that we have
\begin{equation}\nonumber
\begin{split}
&\lim_{\epsilon\to 0}\int_{X_j}
v''_{\epsilon}(\Psi)|fF|^2_{h} e^{-u(-v_{\epsilon}(\Psi))}\\
=&\int_{X_j}v''(\Psi)|fF|^2_h e^{-u(-v(\Psi))}\\
\le&\bigg(\sup_{X_j}e^{-u(-v(\Psi))}\bigg)\int_{X_j}\frac{1}{B}\mathbb{I}_{\{-t_0-B<\Psi<-t_0\}} |fF|^2_h.
\end{split}
\end{equation}

Combining with $$\inf_{\epsilon}\inf_{X_j}e^{v_\epsilon(\Psi)-\delta \max\{\psi+T,2\log|F|\}}c(-v_\epsilon(\Psi))>0,$$
we have
$$\sup_{\epsilon}\int_{X_j}|F_{\epsilon,j}-(1-v'_{\epsilon}(\Psi))fF^{1+\delta}|^2_h<+\infty.$$
For any $i\in\mathbb{Z}_{\ge 0}$, as $h_i\le h$, we have
$$\sup_{\epsilon}\int_{X_j}|F_{\epsilon,j}-(1-v'_{\epsilon}(\Psi))fF^{1+\delta}|^2_{h_i}<+\infty.$$
For any fixed $i\in\mathbb{Z}_{\ge0}$, note that $\overline{X_j}$ is compact and both $h_i$ and $\hat{h}$ are $C^2$ smooth hermitian metrics on $E$, then there exists a constant $c_i>0$, such that $h_i\le c_i \hat{h}$ on $X_k$. By \eqref{estimate on sublevel set of lemma2.1}, we have
\begin{equation}\nonumber
\begin{split}
\sup_{\epsilon}\int_{X_j}|(1-v'_{\epsilon}(\Psi))fF^{1+\delta}|_{h_i}^2
\le c_i(\sup_{X_j}|F|^{2+2\delta}) \int_{X_j}\mathbb{I}_{\{\Psi<-t_0\}}|f|_{\hat{h}}^2<+\infty,
\end{split}
\end{equation}
one can obtain that $\sup_{\epsilon}\int_{X_j}|F_{\epsilon,j}|_{h_i}^2<+\infty$.

Especially, we know
$\sup_{\epsilon}\int_{X_k}|F_{\epsilon,j}|_{h_1}^2<+\infty$. Note that $h_1$ is a $C^2$ hermitian metric on $E$, $X_j\subset\subset X$ and $F_{\epsilon,j}$ is $E$-valued holomorphic $(n,0)$ form on $X_j$, there exists a subsequence of $\{F_{\epsilon},j\}_{\epsilon}$ (also denoted by $\{F_{\epsilon,j}\}_{\epsilon}$) compactly convergent to an $E$-valued holomorphic $(n,0)$ form $F_{j}$ on $X_j$.

It follows from Fatou's lemma that we have

\begin{equation}
\label{step epsilon to 0 formula 1}
\begin{split}
&\int_K|F_{j}-(1-b(\Psi))fF^{1+\delta}|_{h_i}^2
e^{v(\Psi)-\delta \max\{\psi+T,2\log|F|\}}c(-v(\Psi))\\
=&\liminf_{\epsilon\to 0}\int_K|F_{\epsilon,j}-(1-v'_{\epsilon}(\Psi))fF^{1+\delta}|_{h_i}^2
e^{v_\epsilon(\Psi)-\delta \max\{\psi+T,2\log|F|\}}c(-v_\epsilon(\Psi))\\
\le&\limsup_{\epsilon\to 0}\int_K|F_{\epsilon,j}-(1-v'_{\epsilon}(\Psi))fF^{1+\delta}|_{h}^2
e^{v_\epsilon(\Psi)-\delta \max\{\psi+T,2\log|F|\}}c(-v_\epsilon(\Psi))\\
\le&\limsup_{\epsilon\to 0}\int_{X_j}v''_{\epsilon}(\Psi)|fF|^2_h e^{-u(-v_{\epsilon}(\Psi))}\\
\le&\bigg(\sup_{X_j}e^{-u(-v(\Psi))}\bigg)\int_{X_j}\frac{1}{B}\mathbb{I}_{\{-t_0-B<\Psi<-t_0\}} |fF|^2_h.
\end{split}
\end{equation}
Letting $i\to +\infty$ in inequality \eqref{step epsilon to 0 formula 1} and by monotone convergence Theorem, we have
\begin{equation}\nonumber
\begin{split}
&\int_K|F_{j}-(1-b(\Psi))fF^{1+\delta}|_{h}^2
e^{v(\Psi)-\delta \max\{\psi+T,2\log|F|\}}c(-v(\Psi))\\
\le&\bigg(\sup_{X_j}e^{-u(-v(\Psi))}\bigg)\int_{X_j}\frac{1}{B}\mathbb{I}_{\{-t_0-B<\Psi<-t_0\}} |fF|^2_h.
\end{split}
\end{equation}
As $K$ is any compact subset of $X_j$ and by monotone convergence Theorem, we know
\begin{equation}
\label{step epsilon to 0 formula 3}
\begin{split}
&\int_{X_j}|F_{j}-(1-b(\Psi))fF^{1+\delta}|_{h}^2
e^{v(\Psi)-\delta \max\{\psi+T,2\log|F|\}}c(-v(\Psi))\\
\le&\bigg(\sup_{X_j}e^{-u(-v(\Psi))}\bigg)\int_{X_j}\frac{1}{B}\mathbb{I}_{\{-t_0-B<\Psi<-t_0\}} |fF|^2_h.
\end{split}
\end{equation}

\

\emph{Step 7: Letting $j\to+\infty$. }

It is easy to see that
\begin{equation}\nonumber
\begin{split}
&\bigg(\sup_{X_j}e^{-u(-v(\Psi))}\bigg)\int_{X_j}\frac{1}{B}\mathbb{I}_{\{-t_0-B<\Psi<-t_0\}} |fF|^2_h\\
\le& \bigg(\sup_{X}e^{-u(-v(\Psi))}\bigg)\int_{X}\frac{1}{B}\mathbb{I}_{\{-t_0-B<\Psi<-t_0\}} |fF|^2_h<+\infty.
\end{split}
\end{equation}

For fixed $j$, as $e^{v(\Psi)-\delta\max\{\psi+T,2\log|F|\}}c(-v(\Psi))$ has a positive lower bound on any $\overline{X_j}$, we have for $j_1>j$,
$$\sup_{j_1>j}\int_{X_j}
|F_{j_1}-(1-b(\Psi))fF^{1+\delta}|^2_h<+\infty.$$

For any $i\in\mathbb{Z}_{\ge 0}$, as $h_i\le h$, we have
$$\sup_{j_1>j}\int_{X_j}
|F_{j_1}-(1-b(\Psi))fF^{1+\delta}|^2_{h_i}<+\infty.$$
Note that $\overline{X_j}$ is compact and both $h_i$ and $\hat{h}$ are $C^2$ smooth hermitian metrics on $E$, then there exists a constant $c_i>0$, such that $h_i\le c_i \hat{h}$ on $X_k$. By \eqref{estimate on sublevel set of lemma2.1}, we have
\begin{equation}\nonumber
\begin{split}
\int_{X_j}|(1-b(\Psi))fF^{1+\delta}|_{h_i}^2
\le c_i(\sup_{X_j}|F|^{2+2\delta}) \int_{X_j}\mathbb{I}_{\{\Psi<-t_0\}}|f|_{\hat{h}}^2<+\infty,
\end{split}
\end{equation}
one can obtain that $\sup_{j_1>j}\int_{X_j}|F_{j_1}|_{h_i}^2<+\infty$. Especially
$$\sup_{j_1>j}\int_{X_j}|F_{j_1}|_{h_1}^2<+\infty.$$
By diagonal method, there exists a subsequence $F_{j''}$ uniformly convergent on any
$\overline{X_j}$ to an $E$-valued holomorphic $(n,0)$-form on $X$ denoted by $\tilde F$. It follows from Fatou's lemma that we have

\begin{equation}
\label{step j to infty formula 1}
\begin{split}
&\int_{X_j}|\tilde{F}-(1-b(\Psi))fF^{1+\delta}|_{h_i}^2
e^{v(\Psi)-\delta \max\{\psi+T,2\log|F|\}}c(-v(\Psi))\\
\le&\liminf_{j''\to+\infty}\int_{X_j}|F_{j''}-(1-b(\Psi))fF^{1+\delta}|_{h_i}^2
e^{v(\Psi)-\delta \max\{\psi+T,2\log|F|\}}c(-v(\Psi))\\
\le&\limsup_{j''\to+\infty}\int_{X_j}|F_{j''}-(1-b(\Psi))fF^{1+\delta}|_{h}^2
e^{v(\Psi)-\delta \max\{\psi+T,2\log|F|\}}c(-v(\Psi))\\
\le&\limsup_{j''\to+\infty}\int_{X_{j''}}|F_{j''}-(1-b(\Psi))fF^{1+\delta}|_{h}^2
e^{v(\Psi)-\delta \max\{\psi+T,2\log|F|\}}c(-v(\Psi))\\
\le&\limsup_{j''\to+\infty}\bigg(\sup_{X_{j''}}e^{-u(-v(\Psi))}\bigg)\int_{X_{j''}}\frac{1}{B}\mathbb{I}_{\{-t_0-B<\Psi<-t_0\}} |fF|^2_h\\
\le &\bigg(\sup_{X}e^{-u(-v(\Psi))}\bigg)\int_{X}\frac{1}{B}\mathbb{I}_{\{-t_0-B<\Psi<-t_0\}} |fF|^2_h<+\infty.
\end{split}
\end{equation}
Letting $i\to +\infty$ in inequality \eqref{step j to infty formula 1}, and by monotone convergence theorem, we have
\begin{equation}
\label{step j to infty formula 2}
\begin{split}
&\int_{X_j}|\tilde{F}-(1-b(\Psi))fF^{1+\delta}|_{h}^2
e^{v(\Psi)-\delta \max\{\psi+T,2\log|F|\}}c(-v(\Psi))\\
\le &\bigg(\sup_{X}e^{-u(-v(\Psi))}\bigg)\int_{X}\frac{1}{B}\mathbb{I}_{\{-t_0-B<\Psi<-t_0\}} |fF|^2_h<+\infty.
\end{split}
\end{equation}

Letting $j\to +\infty$ in inequality \eqref{step j to infty formula 2}, and by monotone convergence theorem, we have
\begin{equation}
\label{step j to infty formula 3}
\begin{split}
&\int_{X}|\tilde{F}-(1-b(\Psi))fF^{1+\delta}|_{h}^2
e^{v(\Psi)-\delta \max\{\psi+T,2\log|F|\}}c(-v(\Psi))\\
\le &\bigg(\sup_{X}e^{-u(-v(\Psi))}\bigg)\int_{X}\frac{1}{B}\mathbb{I}_{\{-t_0-B<\Psi<-t_0\}} |fF|^2_h<+\infty.
\end{split}
\end{equation}

\

\emph{Step 9: ODE System.}

Now we want to find $\eta$ and $\phi$ such that
$(\eta+g^{-1})=e^{-v_\epsilon(\Psi_m)}e^{-\phi}\frac{1}{c(-v_{\epsilon}(\Psi_m))}$.
As $\eta=s(-v_{\epsilon}(\Psi_m))$ and $\phi=u(-v_{\epsilon}(\Psi_m))$, we have
$(\eta+g^{-1})e^{v_\epsilon(\Psi_m)}e^{\phi}=\left((s+\frac{s'^2}{u''s-s''})e^{-t}e^u\right)\circ(-v_\epsilon(\Psi_m))$.\\

Summarizing the above discussion about $s$ and $u$, we are naturally led to a system of
ODEs:
\begin{equation}
\begin{split}
&1)(s+\frac{s'^2}{u''s-s''})e^{u-t}=\frac{1}{c(t)},\\
&2)s'-su'=1,
\end{split}
\label{ODE SYSTEM}
\end{equation}
when $t\in(T,+\infty)$.

We  solve the ODE system \eqref{ODE SYSTEM} and get
$u(t)=-\log(\frac{1}{\delta}c(T)e^{-T}+\int^t_T c(t_1)e^{-t_1}dt_1)$ and
$s(t)=\frac{\int^t_T(\frac{1}{\delta}c(T)e^{-T}+\int^{t_2}_T c(t_1)e^{-t_1}dt_1)dt_2+\frac{1}{\delta^2}c(T)e^{-T}}{\frac{1}{\delta}c(T)e^{-T}+\int^t_T
c(t_1)e^{-t_1}dt_1}$.

It follows that $s\in C^{\infty}([T,+\infty))$ satisfies
$s \ge \frac{1}{\delta}$ and $u\in C^{\infty}([T,+\infty))$ satisfies
$u''s-s''>0$.

As $u(t)=-\log(\frac{1}{\delta}c(T)e^{-T}+\int^t_T c(t_1)e^{-t_1}dt_1)$ is decreasing with respect to $t$, then
it follows from $-T \ge v(t) \ge \max\{t,-t_0-B_0\} \ge -t_0-B_0$, for any $t \le 0$
that
\begin{equation}
\begin{split}
\sup\limits_{X}e^{-u(-v(\Psi))} \le
\sup\limits_{t\in[T,t_0+B]}e^{-u(t)}=\frac{1}{\delta}c(T)e^{-T}+\int^{t_0+B}_T c(t_1)e^{-t_1}dt_1.
\end{split}
\end{equation}

Combining with inequality \eqref{step j to infty formula 3}, we have

\begin{equation}\nonumber
\begin{split}
&\int_{X}|\tilde{F}-(1-b(\Psi))fF^{1+\delta}|_{h}^2
e^{v(\Psi)-\delta \max\{\psi+T,2\log|F|\}}c(-v(\Psi))\\
\le & \left(\frac{1}{\delta}c(T)e^{-T}+\int^{t_0+B}_T c(t_1)e^{-t_1}dt_1\right)
\int_{X}\frac{1}{B}\mathbb{I}_{\{-t_0-B<\Psi<-t_0\}} |fF|^2_h<+\infty.
\end{split}
\end{equation}
We get Lemma \ref{L2 method}.


\vspace{.1in} {\em Acknowledgements}.
The first author and the second author were supported by National Key R\&D Program of China 2021YFA1003103.
The first author was supported by NSFC-11825101, NSFC-11522101 and NSFC-11431013. We would like to thank Shijie Bao for
checking the manuscript.

\bibliographystyle{references}
\bibliography{xbib}

\end{document}